\documentclass{amsart}

\usepackage{amscd,amsfonts,mathrsfs,amsthm,amssymb}
\usepackage{comment}

\usepackage{mathtools}

\usepackage{enumerate}
\usepackage{multicol}        			
\usepackage{color}           			
\usepackage{array}
\usepackage{ae}
\usepackage{accents}
\usepackage{epsfig}

\usepackage[e]{esvect}

\usepackage{yfonts}

\usepackage{empheq}

\usepackage[latin1]{inputenc}

\usepackage{hyperref}
\usepackage[alphabetic, msc-links, backrefs]{amsrefs}

\definecolor{r}{rgb}{.9,0.1,.3}

\definecolor{verdeosc}{rgb}{0,0.6,0}

\protected\def\ignorethis#1\endignorethis{}
\let\endignorethis\relax
\def\TOCstop{\addtocontents{toc}{\ignorethis}}
\def\TOCstart{\addtocontents{toc}{\endignorethis}}

\usepackage{calrsfs}
\DeclareMathAlphabet{\pazocal}{OMS}{zplm}{m}{n}
\newcommand{\Aa}{\mathcal{A}}

\newcommand{\Tt}{\mathcal{T}}

\newcommand{\domain}{\operatorname{domain}}

\newcommand{\R}{\mathbf{R}}%reals paco
\newcommand{\uup}{\phi^\textnormal{upper}}
\newcommand{\ulow}{\phi^\textnormal{lower}}

\newcommand{\unup}{\phi_n^\textnormal{upper}}
\newcommand{\unlow}{\phi_n^\textnormal{lower}}

\newcommand{\ddx}{\partial_x}

\newcommand{\Ee}{\pazocal{E}}
\newcommand{\Ff}{\mathcal F}
\newcommand{\Gg}{\mathcal{G}}
\newcommand{\Rr}{\mathcal R}
 \newcommand{\Cc}{\mathcal C}
 \newcommand{\Dd}{\mathcal{D}}
 \newcommand{\Hh}{\mathcal{H}}
 
 \newcommand{\Ll}{\mathcal{L}}

  \newcommand{\Pairs}{\mathcal{P}}
\newcommand{\Pp}{\mathcal{P}}

  \newcommand{\Myout}{M^Y_\textnormal{outer}}
  \newcommand{\Myin}{M^Y_\textnormal{inner}}

  \newcommand{\Ss}{\mathcal S}
 
 \newcommand{\RR}{\mathbf{R}}  % reals
  
 \newcommand{\ZZ}{\mathbf{Z}}  % integers
 \newcommand{\BB}{\mathbf{B}}  %ball
 \renewcommand{\SS}{\mathbf{S}}  %circle, sphere
    \newcommand{\dist}{\operatorname{dist}}
 
 \newcommand{\area}{\operatorname{area}}
 \newcommand{\eps}{\epsilon}
 \newcommand{\Omegain}{\Omega^\textnormal{in}}
 \newcommand{\Omegaout}{\Omega^\textnormal{out}}
 \newcommand{\Tan}{\operatorname{Tan}}
 
 \newcommand{\KK}{\mathbf{K}}

 \newcommand{\NN}{\mathbf{N}}
\newcommand{\uu}{\mathbf u}
\renewcommand{\vv}{\mathbf v}
\newcommand{\ee}{\mathbf e}

  \newcommand{\waist}{\operatorname{waist}}
    
    \newcommand{\cin}{C_{\rm in}}
        \newcommand{\cout}{C_{\rm out}}

\newcommand{\pdf}[2]{\frac{\partial #1}{\partial #2}}

\newcommand{\partialin}{\partial_\textnormal{inner}}
\newcommand{\partialout}{\partial_\textnormal{outer}}

\newcommand{\Gammain}{\Gamma_\textnormal{in}}
\newcommand{\Gammaout}{\Gamma_\textnormal{out}}

\newtheorem*{theorem*}{Theorem}
\newtheorem{theorem}{Theorem}[section]

\newtheorem{conjecture}[theorem]{Conjecture}
\newtheorem{lemma}[theorem]{Lemma}

\newtheorem{corollary}[theorem]{Corollary}
\newtheorem*{corollary*}{Corollary}
\newtheorem{proposition}[theorem]{Proposition}
\newtheorem{definition}[theorem]{Definition}

\newtheorem{claim}{Claim}
\newtheorem*{claim*}{Claim}
\theoremstyle{definition}
\newtheorem{remark}[theorem]{Remark}

\newcommand{\graph}{\operatorname{graph}}

\newcommand{\yin}{y^\textnormal{inner}}
\newcommand{\yout}{y^\textnormal{outer}}

\newcommand{\ynin}{y^\textnormal{inner}_n}
\newcommand{\ynout}{y^\textnormal{outer}_n}

\def\pproof#1{\@ifnextchar[\opargproof
{\opargproof[\it Proof of #1.]}}
\def\opargproof[#1]{\par\noindent {\bf #1 }}

\makeatother

\newcommand{\Mup}{M^\textnormal{upper}}
\newcommand{\Mlow}{M^\textnormal{lower}}

\title[Translating Annuli]{Translating Annuli for Mean Curvature Flow}
\author[D. Hoffman]{\textsc{D. Hoffman}}

\address{David Hoffman\newline
 Department of Mathematics\newline
 Stanford University \newline
   Stanford, CA 94305, USA\newline
{\sl E-mail address:} {\bf dhoffman@stanford.edu}}

\author[F. Martin]{\textsc{F. Martín}}

\address{Francisco Martín\newline
Departamento de Geometría y Topología  \newline
Instituto de Matemáticas  de Granada (IMAG) \newline
Universidad de Granada\newline
18071 Granada, Spain\newline
{\sl E-mail address:} {\bf fmartin@ugr.es}
}
\author[B. White]{\textsc{B. White}}

\address{Brian White\newline
Department of Mathematics \newline
 Stanford University \newline 
  Stanford, CA 94305, USA\newline
{\sl E-mail address:} {\bf bcwhite@stanford.edu}
}

\makeindex

\begin{document}

\date{25 August, 2023.  Revised 24 July, 2024}
\thanks{F. Martín  was partially supported by the MICINN grant PID2020-116126-I00, by the IMAG--Maria de Maeztu grant CEX2020-001105-M / AEI / 10.13039/501100011033 and  by the Regional Government of Andalusia and ERDEF grant P20-01391.
B. White was partially supported by grants from the Simons Foundation
(\#396369) and from the National Science Foundation (DMS~1404282, DMS~1711293).}
\thanks{The authors would like to thank Leonor Ferrer for helpful suggestions. They  would also like to thank the reviewer for the careful reading and helpful suggestions.}

\begin{abstract}
We construct a family $\Aa$ of complete, properly embedded, annular translators $M$
such that $M$ lies in a slab and is invariant under reflections in the vertical coordinate planes.
For each $M$ in $\Aa$, $M$ is asymptotic as $z\to -\infty$ to four vertical planes $\{y=\pm b\}$
and $\{y=\pm B\}$ where $0<b\le B <\infty$.  We call $b$ and $B$ the {\bf inner width} and the
{\bf (outer) width} of $M$.  We show that for each $b\ge \pi/2$ and each $s>0$, there
is an $M\in \Aa$ with inner width $b$ and with necksize $s$.
(We also show that there are no translators with inner width $< \pi/2$ having the properties of the
examples we construct.)
\end{abstract}

\maketitle

\setcounter{tocdepth}{1}

\tableofcontents

\section{Introduction}\label{introduction}

A (normalized) {\bf translator} is a surface $M$ in $\RR^3$ such that 
\[
    t\mapsto M - t\ee_3
\]
is a mean-curvature flow.  This is equivalent to the condition that the mean curvature vector
at each point of $M$ is equal to $(-\ee_3)^\perp$. 

A vertical plane is a translator, and it is complete.  There are also graphical translators, i.e., complete,
translating surfaces that are graphs $z=z(x,y)$ over open subsets of $\RR^2$.  They have been completely classified.
In particular, every graphical translator is a grim reaper surface (tilted or untilted), a $\Delta$-wing, or a bowl soliton.
See Figure~\ref{graphs-figure}.  These surfaces are described in Section~\ref{graphs-section}.
The graphical examples are all simply connected.
\begin{figure}[h]\label{graphs-figure}
{\includegraphics[width=3cm]{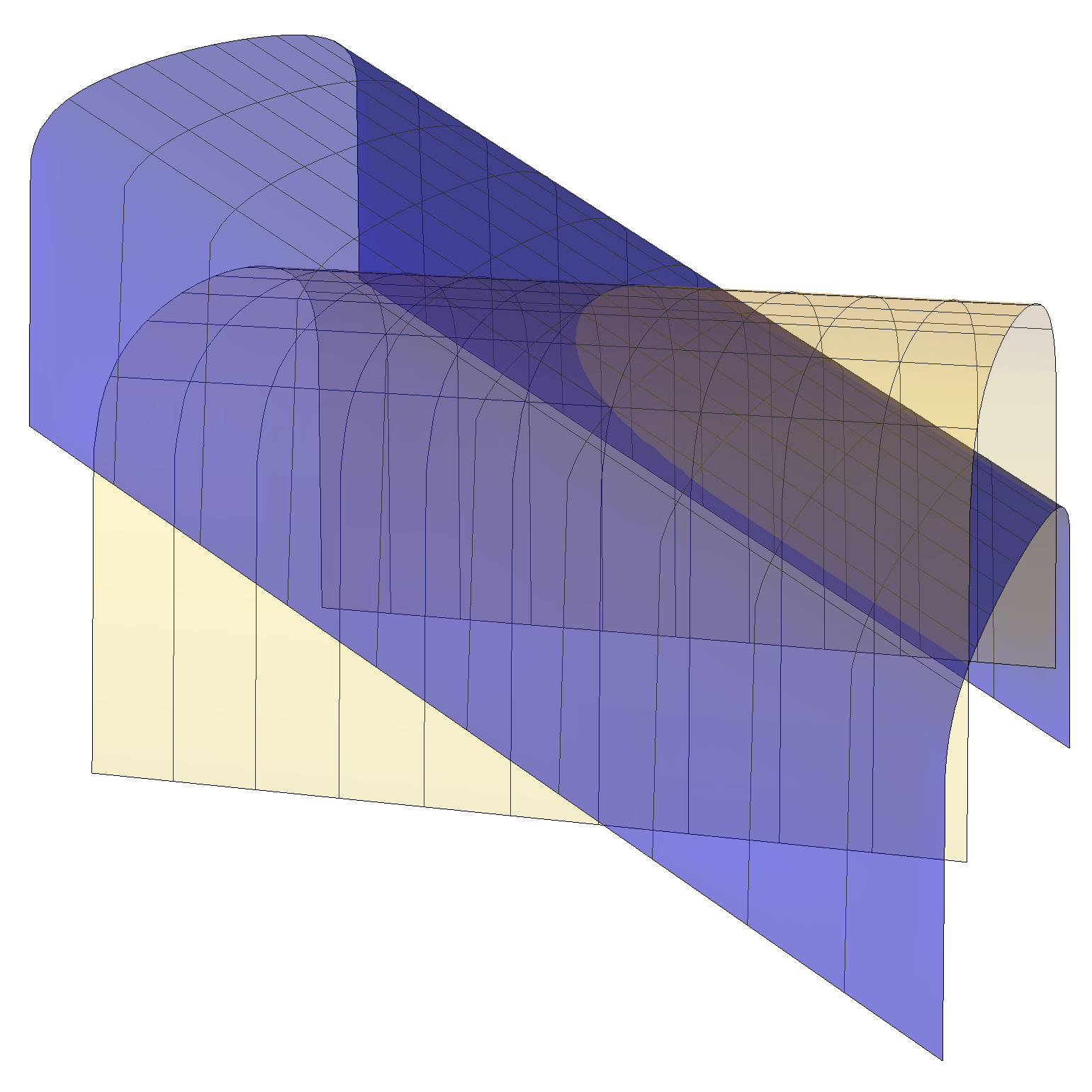}}
{\includegraphics[width=3.5cm]{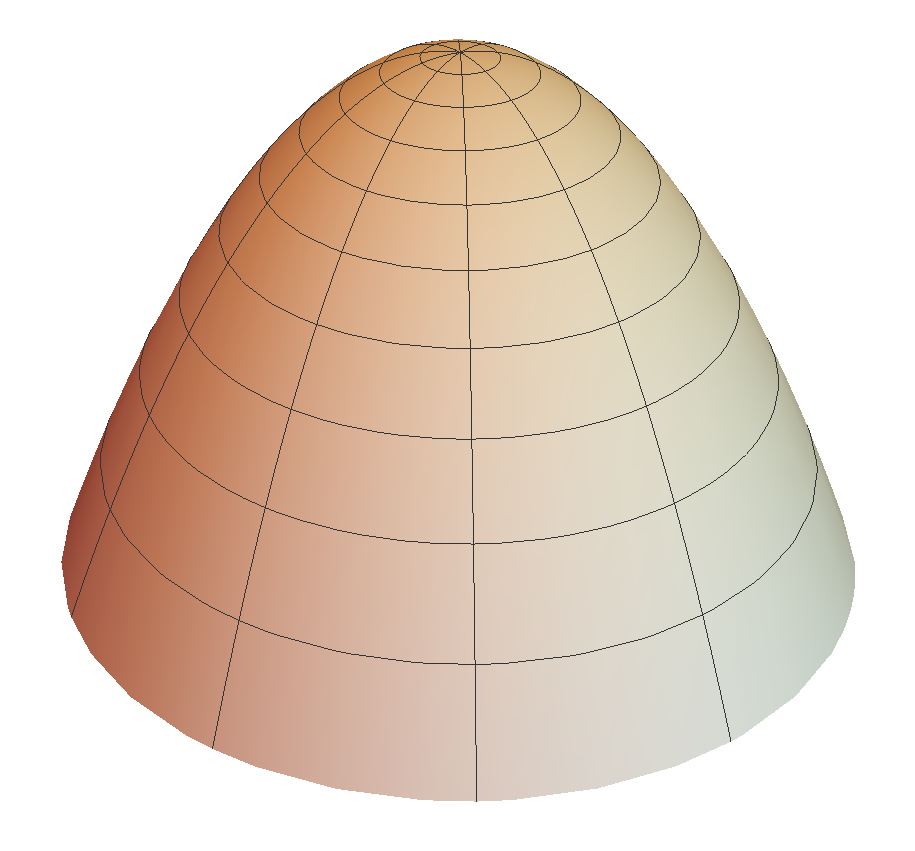}}
{\includegraphics[width=3.9cm]{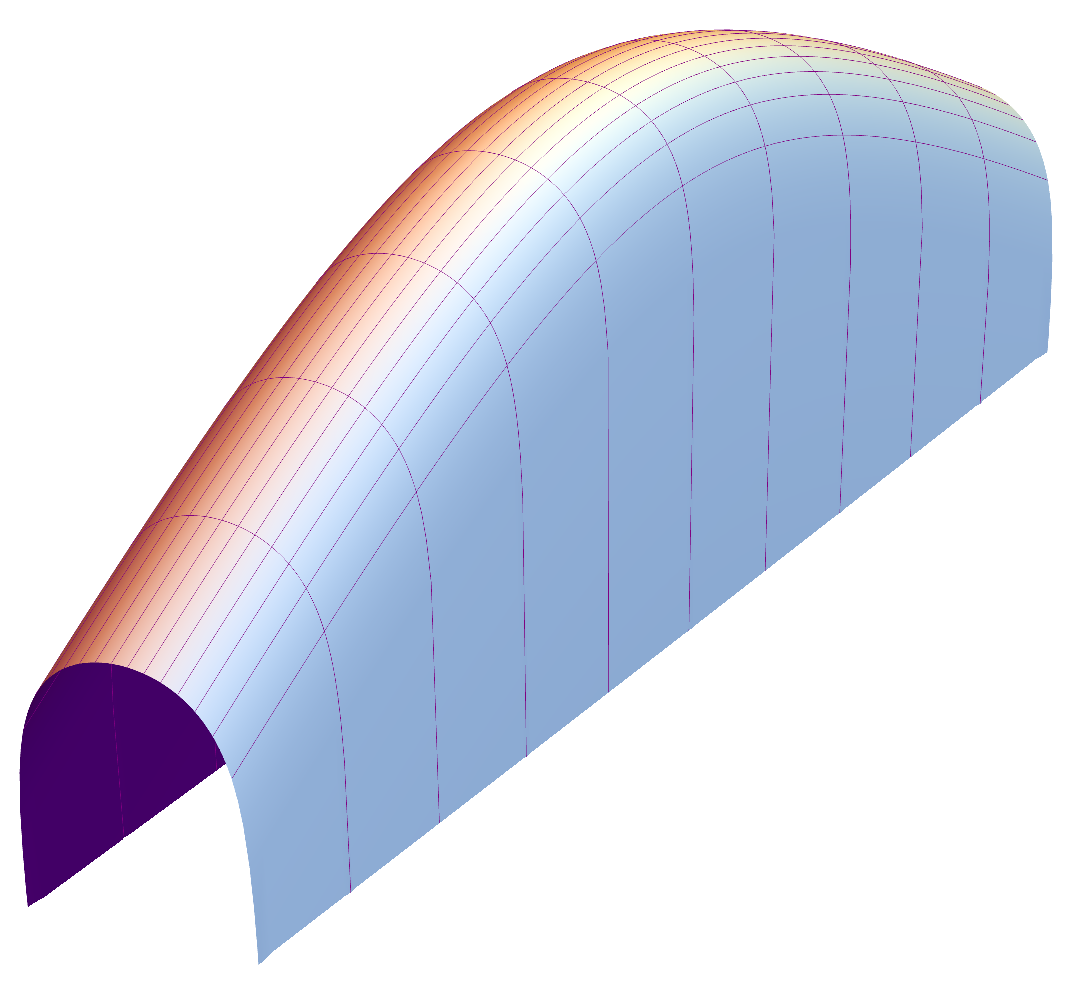}}
\caption{\small From left to right: Two  grim reaper surfaces ($b=\pi/2$ in yellow and $b>\pi/2$ in blue), a bowl soliton, and a $\Delta$-wing.}
\end{figure}

The next simplest examples are the rotationally-invariant translating annuli (also known as translating catenoids):
for every $R>0$, there is a complete translating annulus $M$ that is rotationally invariant about $Z$ (the $z$-axis)
and whose distance from $Z$ is $R$.  Furthermore, it is unique up to a vertical translation.
See~\cite{CSS}.
\begin{figure}[h]
{\includegraphics[width=4cm]{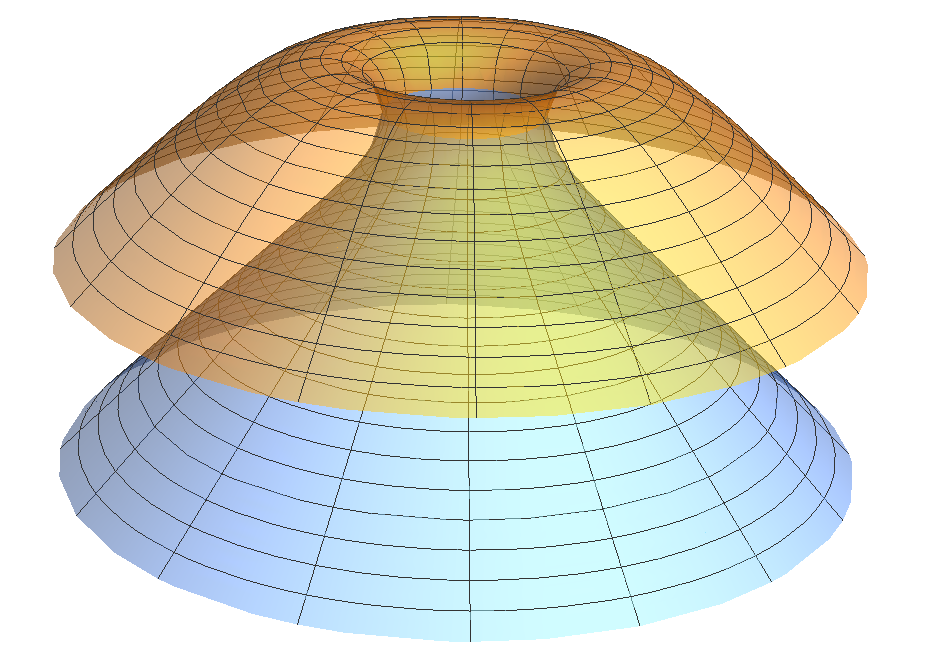}}
\caption{\small A translating catenoid of revolution.}
\label{catenoid}
\end{figure}

It is natural to wonder whether there are other complete, translating annuli.   That is, are there
 complete translating annuli that are not surfaces of revolution?
In this paper, we show that there is a large family of such annuli.

Gluing is a powerful tool that has been successfully used
in many geometric problems, including construction of
 translators~\cite{del-pino, nguyen09, nguyen13, nguyen-survey, smith21}.
It would be natural to try to use gluing to construct non-rotational translating annuli by
connecting two $\Delta$-wings (or two untilted grim reaper surfaces), one slightly above
the other, by a small catenoidal neck.  
(Grim reaper surfaces and $\Delta$-wings are shown in Figure~\ref{graphs-figure}
 and discussed in Section~\ref{graphs-section}.) 
We do in fact construct surfaces that fit that description (Figure~\ref{fig:capped-uncapped}, left).
However, much to our surprise, our method (a continuity method) also produced other surfaces,
  ``uncapped annuloids'',  that exhibit strikingly different behavior
   (Figure~\ref{fig:capped-uncapped}, right).
We believe that such uncapped annuloids could not arise from any gluing method.

We define an {\bf annuloid}\index{annuloid} to be a complete, properly embedded translator $M$ such that
\begin{enumerate}
\item\label{annuloid-def-1} $M$ is an annulus.
\item\label{annuloid-def-2} $M$ lies in a slab $\{|y|\le B'\}$.
\item\label{annuloid-def-3} $M$ is symmetric with respect to the vertical coordinate planes.
\item\label{annuloid-def-4} {$M$ is disjoint from the $z$-axis, $Z$.}
\item\label{annuloid-def-5} $M+(0,0,z)$ converges smoothly as $z\to \infty$ to four planes $\{y=\pm b\}$ and $\{y=\pm B\}$ for some
   $0\le b\le B<\infty$.
\item\label{annuloid-def-6} $M - (0,0,z)$ converges as $z\to\infty$ to the empty set.
\end{enumerate}
We define the {\bf width} $B(M)$  of $M$ to be the number $B$.
\index{$B(M)$ (width of $M$)}\index{width}
 (One can prove that $B$ is also the smallest $B'$ such that~\eqref{annuloid-def-2} holds; see Corollary~\ref{outermost-corollary}.)
    We define the {\bf inner width} $b(M)$  
    of $M$ to be the number $b$.  
    \index{$b(M)$}\index{inner width}

In this paper, we prove existence of a large family of annuloids.
In particular, we prove 

\begin{theorem}\label{theorem-1}
For each $b\ge \pi/2$ and for each $0<s< \infty$,
there exists an annuloid with inner width $b$ and with necksize $s$.
\end{theorem}

The restriction $b\ge \pi/2$ is not arbitrary.
All of our examples have ``finite type" (as defined in Section~\ref{finite-type-section}.)
There are no finite-type annuloids with inner width $<\pi/2$.  See Theorem~\ref{th:b>pi2}.

{The annuloids in Theorem~\ref{theorem-1} are obtained as follows.  Using a path-lifting argument,
we prove an analogous existence theorem for compact translating annuli bounded by pairs of nested rectangles.
As the lengths of the rectangles tend to infinity, the rectangles converge to four parallel lines.  We prove that
suitable vertical translates of the compact annuli converge to annuloids (without boundary).
}

To make Theorem~\ref{theorem-1} precise, one needs to define necksize.
There are various natural definitions, such as:
 the length of the shortest homotopically nontrivial curve in $M$,
or the radius of the smallest ball containing a nontrivial curve, or the distance from the $z$-axis, $Z$, to the surface.
   Our existence result is true for any of those definitions.
For small necks, the different  notions of necksize are essentially equivalent.  
(See, for example, Corollary~\ref{neck-notions-corollary}.)
However, the following definition turns out to be most convenient notion of necksize.  In particular,  for large necks, it
is more suitable than the other notions.

\begin{definition} \label{def:x(M)}
If $M$ is a surface, we let $x(M)$ be the distance from $Z$ to $M\cap \{y=0\}$.
\end{definition}
One can think of $x(M)$ as measuring the size of the neck in the $x$-direction. 
See Figure~\ref{fig:x(M)}.
\index{$x(M)$}
\index{neck-size}

Conditions~\eqref{annuloid-def-4}, \eqref{annuloid-def-5}, and~\eqref{annuloid-def-6}
 in the definition of  annuloids imply that if $M$ is an annuloid, then $x(M)>0$.
We do not know whether there is a {\bf unique}  annuloid with a given inner width and a given
necksize.
However,  the family of annuloids we construct behaves, in some ways, like a nice $2$-parameter family,
 as indicated by the following theorem.
 (See also Theorems~\ref{complete-existence-theorem},
  \ref{connected-family-theorem}, \ref{properness-theorem}, and~\ref{b-B-x-continuity-theorem}.)
\begin{figure}[htbp]
\begin{center}
\includegraphics[width=.75\textwidth]{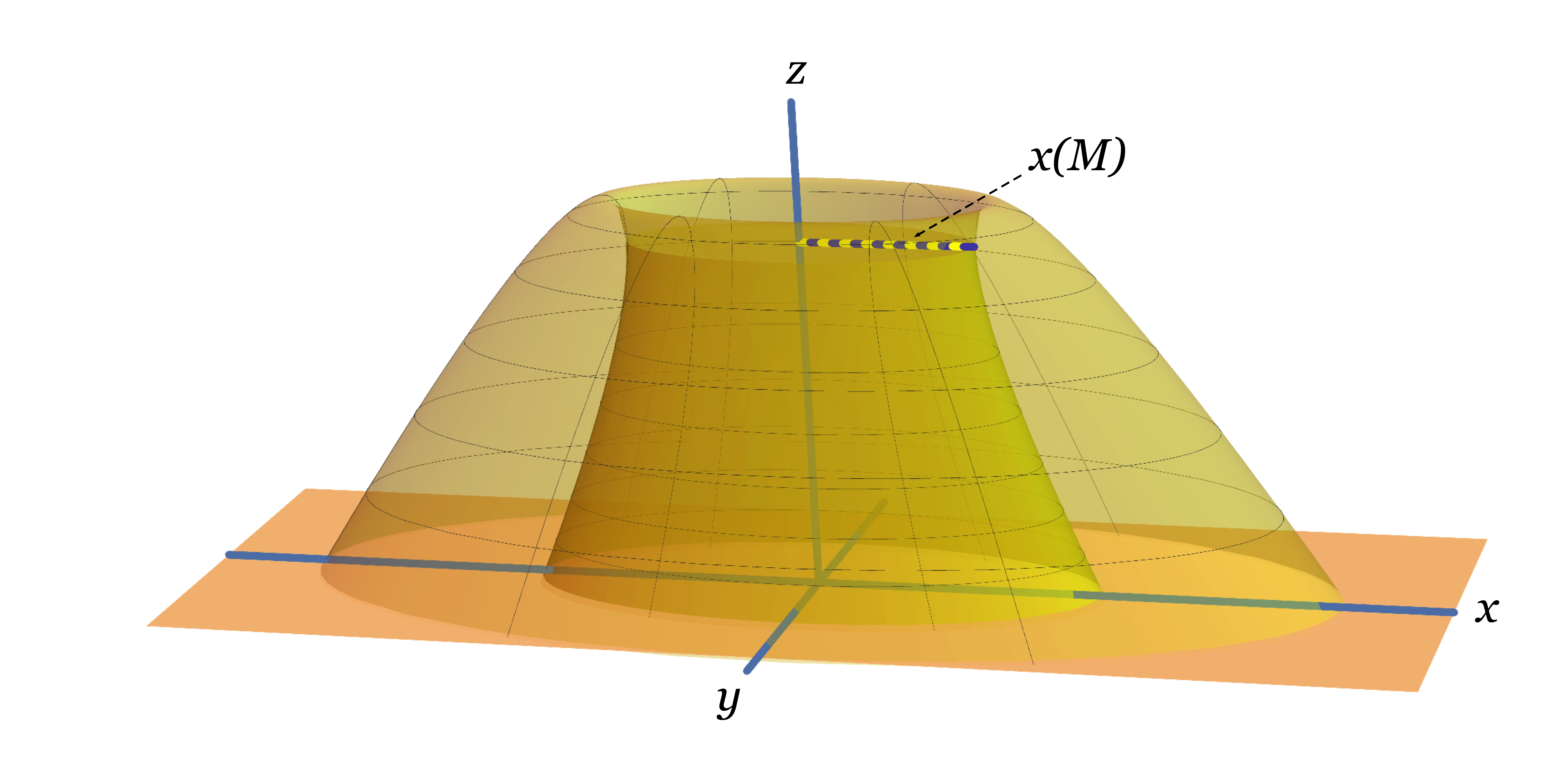}
\caption{\small The distance $x(M)$ for a translating annulus.
This annulus is bounded by a pair of nested convex curves.}
\label{fig:x(M)}
\end{center}
\end{figure}

\begin{theorem} \label{theorem-2}
There is a family $\Aa$ of annuloids in $\RR^3$ such that
the map $M\in \Aa \mapsto B(M)$ is continuous, and such that  the
map
\begin{align*}
&\Phi:   \Aa \to [\pi/2, \infty)\times (0,\infty), \\
&\Phi(M)  = ( b(M), x(M))
\end{align*}
is continuous, proper, and surjective.

For $b\ge \pi/2$, let 
\[ 
   \Aa(b) := \{M\in \Aa: b(M)=b\}.
\]
There is a closed, connected subset $\Ff=\Ff(b)$ of $\Aa(b)$ such that 
$M\mapsto x(M)$ is a proper, surjective map from $\Ff$ onto $(0,\infty)$.
\end{theorem}
\index{$\Aa$}

The behavior as necksize tends to $0$ or to infinity is described by the following theorem.
(See Theorems~\ref{infinite-neck-theorem}~and~\ref{small-neck-theorem}.)

\begin{theorem}
 \label{theorem-3}
Suppose that $M_n\in \Aa$ and that $b(M_n)\to b<\infty$.
\begin{enumerate}
\item\label{intro-3-1} If $x(M_n)\to 0$, then $M_n$ converges with multiplicity $2$
 to the translating graph 
 \[
    f_b: \RR\times (-b,b)\to \RR
 \]
 with $f_b(0,0)=0$ and $Df_b(0,0)=0$.  The convergence is smooth away
 from the origin.  
 { After suitable rescaling, the surfaces converge to a catenoid.}
 \item\label{intro-3-2} If $x(M_n)\to\infty$, then $M_n-(0,0,\zeta_n)$ converges
   smoothly to a pair of grim reaper surfaces, one over $\RR\times (b,b+\pi)$
   and the other over $\RR\times (-(b+\pi),-b)$. Here
   \[ 
     \zeta_n:= \max\{z: (0,y,z)\in M_n\}.
   \]
\end{enumerate}
\end{theorem}
By the classification of graphical translators (Theorem~\ref{classification-theorem}),
 the graph of $f_b$ in Assertion~\eqref{intro-3-1} is an untilted grim reaper surface if $b=\pi/2$
and a $\Delta$-wing if $b>\pi/2$.

It is natural to ask what happens if we fix necksize and let the inner width tend to infinity.

\begin{conjecture}\label{rotational-conjecture}
Suppose that $M_n$ are annuloids in $\Aa$ such that $x(M_n)=R$ and $b(M_n)\to\infty$.
Then the $M_n$ converge to the rotationally symmetric
translating annulus whose neck is a horizontal circle of radius $R$.
\end{conjecture}

In Conjecture~\ref{rotational-conjecture}, 
if $R>0$ is small, then, by Theorem~\ref{theorem-3}, $M_n$ resembles two copies of $\graph(f_{b_n})$ joined by a neck.
The fact that $\graph(f_{b_n})$ converges smoothly to a rotationally symmetric surface (a bowl soliton)
makes Conjecture~\ref{rotational-conjecture} plausible.

\TOCstop
\section*{Behavior away from the $z$-axis}
\TOCstart

Let $M\in \Aa$.  We know that as $z\to -\infty$, the surface looks like $4$ planes,
and as $z\to\infty$, it looks like the empty set.

What does it look like as $x\to\infty$?  (The symmetric behavior occurs as $x\to -\infty$.)
We show that $M\cap\{x>x(M)\}$ has two connected components, $\Mup$ and $\Mlow$.
For large $x$, $\Mlow$ looks like a downward-tilted grim reaper surface over $\RR\times (-b,b)$.
\begin{figure}[htbp]
\begin{center}
\includegraphics[width=.45\textwidth]{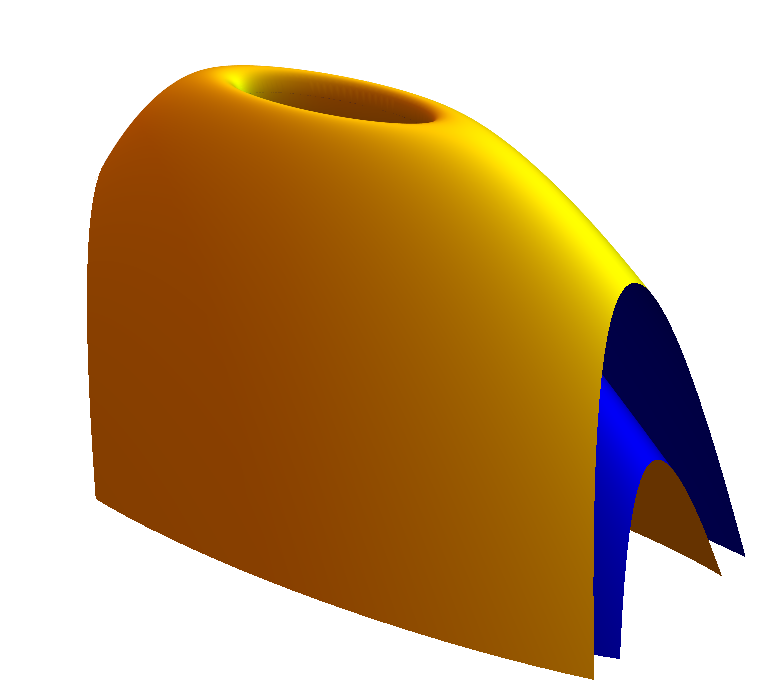}\includegraphics[width=.45\textwidth]{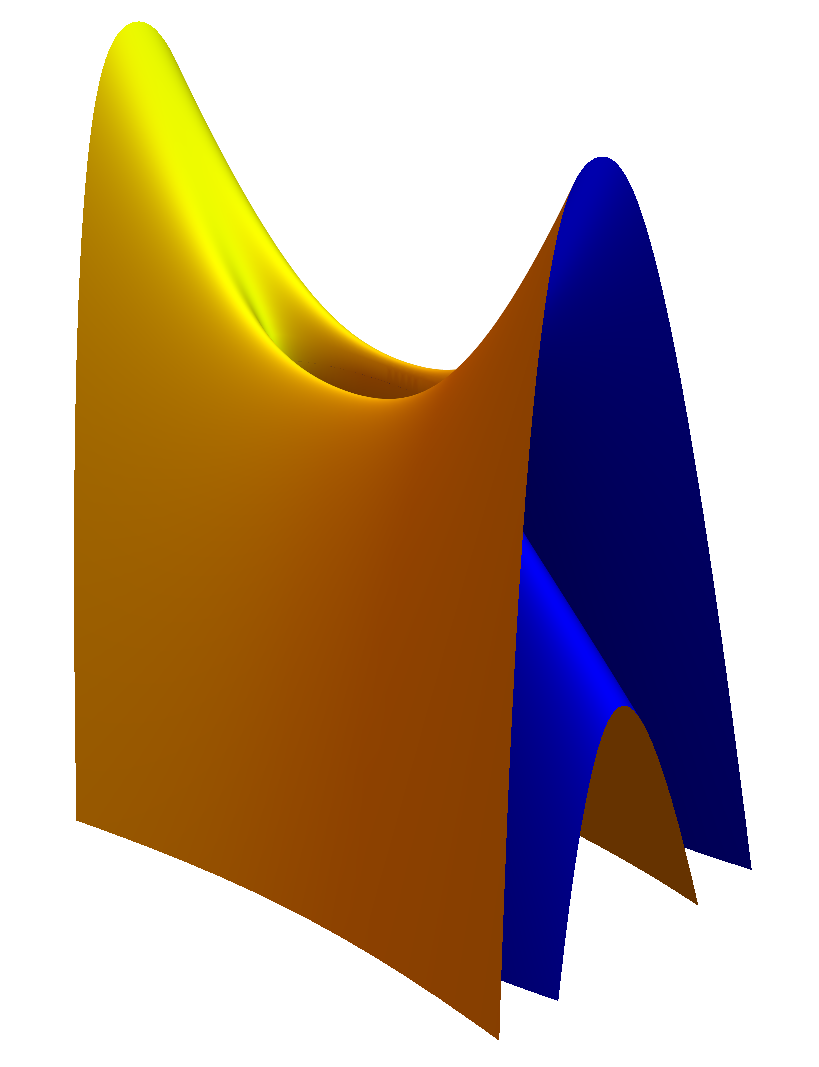}
\caption{\small A {\bf capped} example (left) and an {\bf uncapped} one (right).}
\label{fig:capped-uncapped}
\end{center}
\end{figure}

Also, for large $x$, $\Mup$ looks like a tilted grim reaper  surface $G$ over $\RR\times (-B,B)$.
But (assuming $B>\pi/2$) there are two kinds of $M$:
\begin{enumerate}
\item For some $M$, $G$ is  a downward-tilted grim reaper surface over $\RR\times (-B,B)$.  In this case,
we say that $M$ is {\bf capped}.
\item For other $M$, $G$ is an upward-tilted grim reaper surface over $\RR\times (-B,B)$.
In this case, we say that $M$ is {\bf uncapped}.
\end{enumerate}
Equivalently (assuming $B>\pi/2$), $M$ is capped if and only $z(\cdot)|M$ is bounded,
and $M$ is uncapped if and only if $z(\cdot)|M$ is unbounded.
See \S\ref{capping-section}.

%Then $M\in \Aa(b)$ is capped if and only if $M$ lies below some horizontal plane.

If $B(M)>b(M)$, then $M$ is uncapped
 (by Theorem~\ref{main-theorem-concluded}\eqref{uncapped-item}).
However, the converse is not true: there are examples of uncapped $M$ with $b(M)=B(M)$.
   See Corollary~\ref{uncapped-but-squeezed}.

The following theorem 
describes the relation between necksize and capping (see Theorem~\ref{cutoffs-theorem})
when $b>\pi/2$. 

\begin{theorem}\label{necksize-limits-theorem}
Suppose $b>\pi/2$.
There are constants $0<c\le d< \infty$ (depending on $b$) with the following properties:
\begin{enumerate}
\item\label{necksize-capped-item} If $M\in \Aa(b)$ and $x(M)< c$, then $M$ is capped (and thefore $b(M)=B(M)$).
\item\label{necksize-uncapped-item}
 If $M\in \Aa(b)$ and $x(M)\ge d$, then $b(M)<B(M)$ (and therefore $M$ is uncapped).
\end{enumerate}
\end{theorem}

In the case $b=\pi/2$, there is a $d<\infty$ for which~\eqref{necksize-uncapped-item} holds.
But (in that case) we do not know whether $B(M)=\pi/2$ for all sufficiently small neck sizes. 
Indeed, we do not know if there are any annuloids with $B(M)=\pi/2$.
We do know that for every $\eps>0$, there is a $c=c(\eps)$ such that if $M\in \Aa(\pi/2)$
and if $x(M)<c$, then $B(M)< \pi/2+ \eps$.  
See Theorems~\ref{small-neck-theorem}\eqref{small-neck-5} and~\ref{cutoffs-theorem}.

\TOCstop
\section*{Prongs}
\TOCstart

Consider an $M\in \Aa(b)$ with $x(M)$ very large.
Then (according to Theorem~\ref{theorem-3}), near $Z$ (the $z$-axis), $M$ looks like a pair of grim reaper surfaces.
What does it look like away from $Z$?
In particular, what does it look like in the region where $x\sim x(M)$? 
That is perhaps the most interesting region, since $x(M)$ is the value of $t$ 
for which the topology of $M\cap \{|x|\le t\}$
undergoes a change.

To answer that question, consider a sequence $M_n\in \Aa(b)$ with $x(M_n)\to\infty$.
We show that the surfaces 
\[
  M - (x(M_n),0,0)
\]
converge smoothly, perhaps after passing to a subsequence, to a simply connected translator $\Sigma$ that we 
call a {\bf prong}.
\begin{figure}[htbp]
\begin{center}
\includegraphics[width=.45\textwidth]{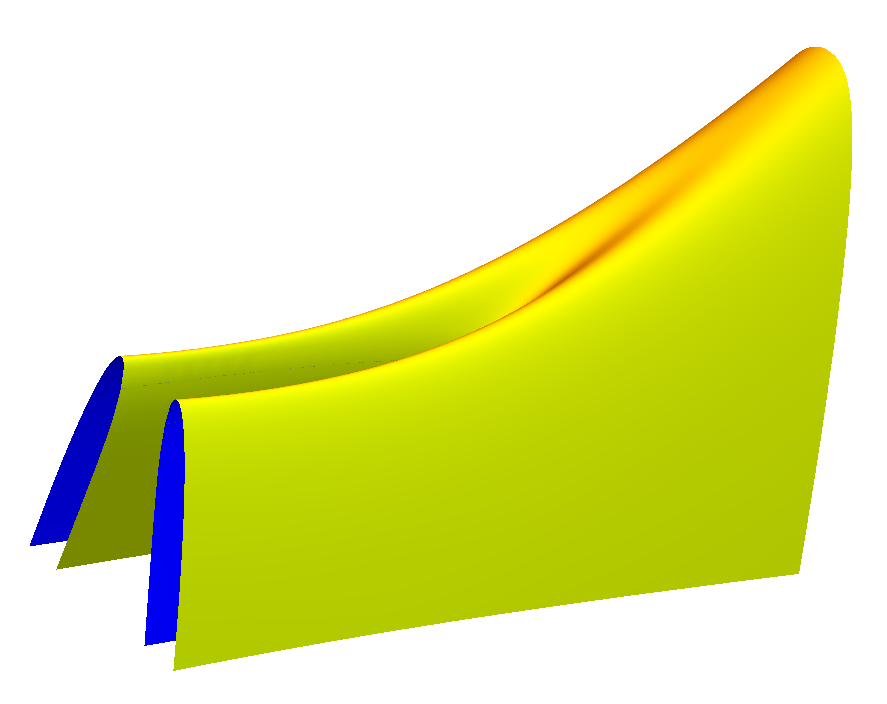}\includegraphics[width=.45\textwidth]{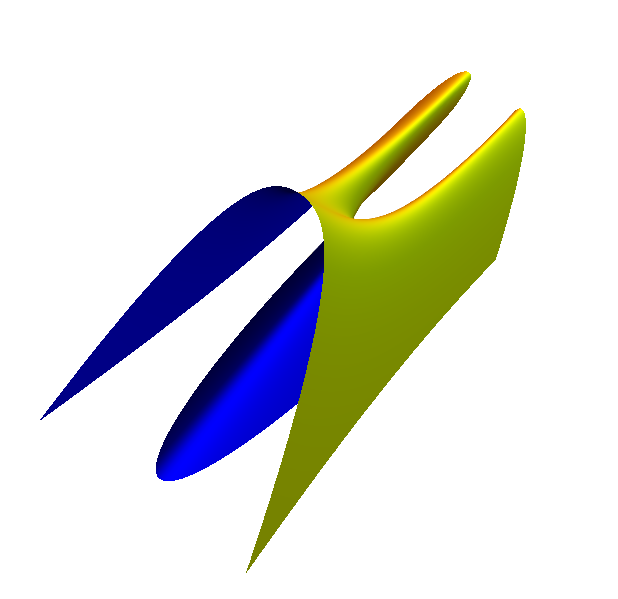}
\caption{\small Two different views of a prong.}
\label{fig:prongs}
\end{center}
\end{figure}

As $z\to\infty$, $\Sigma+(0,0,z)$ converges to the planes $\{y=\pm b\}$ and $\{y=\pm B\}$, where $B=b+\pi$.

The portion of $\Sigma$ with $\{x>0\}$ behaves just like an uncapped annuloid.  In particular, $\Sigma\cap \{x>0\}$
has two components, $\Sigma^\textnormal{upper}$ and $\Sigma^\textnormal{lower}$.   
Where $x$ is large, $\Sigma^\textnormal{upper}$ looks like an upward-tilted grim reaper surface over $\RR\times (-b-\pi,b+\pi)$
and $\Sigma^\textnormal{lower}$ looks like a downward-tilted grim reaper surface over $\RR\times (-b,b)$.

Where $x$ is very negative, $\Sigma$ looks like a pair of untilted grim reaper surfaces over
$\RR\times (-b-\pi,-b)$ and $\RR\times (b,b+\pi)$.  

The prong $\Sigma$ projects diffeomorphically (under $(x,y,z)\mapsto (y,z)$) onto an open subset of the $yz$-plane.  Thus it is the sideways graph of a function $x=x(y,z)$.

Prongs are discussed in Section~\ref{large-neck-section}.
The $x=x(y,z)$ sideways graph property is proved in Section~\ref{final-graphicality-section}.

\TOCstop
\section*{Singularity Models}
\TOCstart

All of the examples of complete translators produced in this paper have finite topology and finite entropy,
and thus could conceivably arise by blowing up (at a singularity) the mean curvature flow of 
an initially smooth, closed surface.   
   If the closed surface is embedded, then, by recent work of Bamler
and Kleiner~\cite{bamler-kleiner}, the bowl soliton is the only
translator (other than a vertical plane) that can arise as a blow-up.
Whether any of the new examples in this paper can arise as blow-ups of immersed surfaces
 is an open question that seems very challenging.

The paper~\cite{annuloid-survey} discusses how the construction of annuloids is related to the way $\Delta$-wings
arise as limits of compact translating graphs.

\section{Graphical Translators}\label{graphs-section}

As mentioned in the introduction, graphical translators have been completely classified.  In this section, we describe the classification.
See also Figure~\ref{graphs-figure}.   For more information about graphical translators, see the survey article~\cite{himw-survey}.

There is a unique entire translating graph, the {\bf bowl soliton}. It is a surface of revolution. Here, and below, `unique' means unique up to
translations in $\RR^3$ and rotation about a vertical axis.  

Every other properly embedded translating graph is defined on a strip of width $2b\geq \pi$. 
We may take the strip to be $\RR\times (-b,b)$. 
As $z\rightarrow-\infty$, the surface is asymptotic  to the planes $\{y=\pm b\}$. Up to the  Euclidean motions mentioned above, here is a
complete list of the graphical translators defined on strips.

\begin{enumerate}
\item For $b=\pi/2$, the {\bf untilted grim reaper surface} is the graph
\begin{align*}
&u:\RR\times(-\pi/2,\pi/2)\rightarrow \RR, \\
&u(x,y)=\log(\cos y).
\end{align*}
The grim reaper surface $u$ is ruled by horizontal lines. It is symmetric with respect to reflection in the  vertical coordinate planes. 

\item\label{u_b} For $b>\pi/2$, the {\bf tilted grim reaper surface} 
\[
u_b:\RR\times(-b,b)\rightarrow \RR,
\]
is produced from the untilted grim reaper surface by  dilation  followed by rotation around the $y$-axis. 
Let $\sigma_b=\frac{2b}{\pi}$ and $s(b)=\sqrt{\sigma_b^2-1}$.  Then
\[
u_b(x,y)=\sigma_b^2 \log (\cos (y/\sigma_b)) +s(b)x.   
\]
The angle of rotation is $\theta_b= \arctan s(b)$, and  the dilation is by a factor of $\sigma_b$. 
The tilted  grim reapers are  ruled by horizontal lines making an angle of $\theta_b$ with the horizontal plane. 
Note that $u_b(x,y)\equiv u_b(x,-y)$.  
Note also that $u_b$  is the unique translator on $\RR\times (-b,b)$ such that
\[
u_b(0,0)=0 \,\, \mbox{and} \,\, \frac{\partial u_b}{\partial x}\equiv s(b).
\]
When $b=\pi/2$,  $u_{\pi/2}$ is the untilted grim reaper surface.
We let $w_b(x,y)=u_b(-x,y)$.  Thus $w_b$ is the unique translator on $\RR\times (-b,b)$ such that
\[
  w_b(0,0)=0 \quad\text{and}\quad \pdf{w_b}{x}\equiv -s(b).
\]
\index{$u_b$}
\index{grim reaper surface}
\noindent The grim reaper surfaces (tilted and untilted)  are   the complete translating graphs that are intrinsically flat.
\item For each $b\ge \pi/2$, there is a unique graphical translator 
\[
   f_b: \RR\times (-b,b)\to \RR
\]
such that
\[
  f_b(0,0)=0 \quad\text{and}\quad Df_b(0,0)=0.
\]
When $b=\pi/2$, $f_{\pi/2}$ is the untilted grim reaper surface. 
When $b>\pi/2$, $D^2f_b$ is everywhere negative definite.  In this case, $\graph(f_b)$ is called a
  {\bf $\Delta$-wing}.
It is invariant under reflection in the vertical coordinate planes: $f_b(x,-y)= f_b(x,y)=f_b(-x,y)$. 
The function $f_b$ attains its maximum value $0$ at the origin.
 The $\Delta$-wings have strictly positive curvature, as does the bowl soliton. 
 \end{enumerate}
\index{$\Delta$-wing}
\index{$f_b$}
 \begin{theorem}\label{classification-theorem}\cite{graphs}*{Theorem~7.1}
Every complete graphical translator is a grim reaper surface, a $\Delta$-wing, or a bowl soliton.
\end{theorem}

The $\Delta$-wings and grim reaper surfaces are related as follows. 
The function 
\[
f_b(x + t,y) -f_b(t,0)
\]
converges to $u_b(x,y)$ as $t\to -\infty$  and to  $w_b(x,y)$ as $t\to +\infty$.

As $b\rightarrow\infty$, the tilted grim reaper surfaces converge to the vertical plane $\{x=0\}$ and the $\Delta$-wings  converge to the bowl soliton.

\section{Bounds on Area and on Curvature}

The surfaces constructed in this paper will all be either compact translators bounded by a pair of nested convex curves
in a horizontal plane or limits of such compact translators.  The following theorem
gives area and curvature bounds that hold for all such surfaces.
Here $Z$ is the $z$-axis, and $\rho_Z$ is
  rotation by $\pi$ around $Z$.

The following lemma for translators is an analog of the convex hull property for
minimal surfaces in Euclidean space.

\begin{lemma}
Suppose $M$ is a compact translator.
If $\vv$ is a horizontal vector and if $F_\vv(p):=\vv\cdot p$, then
\begin{align*}
\min_M F_\vv &= \min_{\partial M}F_\vv, \\
\max_M F_\vv &= \max_{\partial M}F_\vv.
\end{align*}
If $z=\phi(x,y)$ is the equation of a bowl soliton
and if $\Phi(x,y,z)= z - \phi(x,y)$, then
\begin{align*}
\min_M \Phi &= \min_{\partial M}\Phi, \\
\max_M \Phi &= \max_{\partial M}\Phi.
\end{align*}
\end{lemma}
\index{$F_\vv$}

\begin{proof}
The lemma follows immediately from the maximum principle.
\end{proof}

\begin{corollary}\label{bounded-corollary}
Suppose that $M_i$ is a sequence of compact translators.
\begin{enumerate}
\item If the $\partial M_i$ lie in a bounded subset of $\RR^3$, then the $M_i$
lie in a bounded subset of $\RR^3$.
\item If $\lambda_i\to \infty$ and if the dilated boundaries $\lambda_i(\partial M_i)$ lie in a bounded
subset of $\RR^3$, then the surfaces $\lambda_iM_i$ lie in a bounded subdset of $\RR^3$.
\end{enumerate}
\end{corollary}

\begin{proposition}\label{simply-connected-proposition}
Suppose that $M_i$ is a sequence of simply connected translators.
\begin{enumerate}
\item If $M_i$ converges smoothly to an embedded surface $M$,
then $M$ is simply connected.
\item If $\lambda_i\to\infty$ and if $\lambda_iM_i$ converges to an
embedded surface $M$, then $M$ is simply connected.
\end{enumerate}
\end{proposition}

\begin{proof}
To prove (1), let $C$ be a simple closed curve in $M$. Then $C$ is the smooth limit
of simple closed curves $C_i$ in $M_i$.  Since $M_i$ is simply connected,
$C_i$ bounds a disk $D_i$ in $M_i$.  By Corollary~\ref{bounded-corollary}, the $D_i$ lie in a bounded
subset of $\RR^3$. Thus the $D_i$ converge smoothly to a disk $D$ in $M$ with
$\partial D=C$.  The proof of (2) is essentially the same.
\end{proof}

\begin{theorem}\label{curvature-bound-theorem} 
There are finite constants $c_1$ and $c_2$ with the following properties.
Suppose that $M$  is a compact, embedded translator bounded by a pair of convex curves in horizontal planes.
More generally, suppose that $M$ is a smooth limit of such surfaces.
For $p\in M$ and $r\ge 0$, let
\[
    M(p,r):= \{q\in M: \dist_M(q,p)\le r\}
\]
and let $R(M,p)$ be the infimum of $r>0$ such 
that $M(p,r)$ contains a homotopically non-trivial curve in $M$. 
Then
\begin{enumerate}[\upshape(1)]
\item\label{area-item} $\area(M\cap \BB(p,r))\le c_1r^2$ for all balls $\BB(p,r)$.
\item\label{curvature-item} The 2nd fundamental form satisfies
\[
   |A(M,p)|\, \min \{1, R(M,p), \dist_M(p, \partial M)\} \le c_2,
\]
where $\dist_M$ denotes geodesic distance in $M$.
\item\label{in-addition-item} If, in addition, $M$ is an annulus disjoint from $Z$ and invariant under $\rho_Z$,
\begin{align*}
   R(M,p) &\ge \dist(p,Z)\text{, and} \\
  R(M,p) &\ge  x(M):=\dist(Z, M\cap\{y=0\}),
\end{align*}
and thus
\begin{align*}
  |A(M,p)|\, \min \{1, \dist(p, Z \cup \partial M)\} &\le c_2, \\
  |A(M,p)|\, \min \{1, x(M), \dist_M(p, \partial M)\} &\le c_2.
\end{align*}
\end{enumerate}
\end{theorem}
\index{$A(M,p)$}
\index{$R(M,p)$}
\index{$M(p,r)$}

\begin{proof} 
It suffices to prove Assertions~\eqref{area-item} and~\eqref{curvature-item} for compact $M$, since the general case follows trivially.

The existence of $c_1$ is proved in \cite{white-entropy}.

Suppose that Assertion~\eqref{curvature-item} fails. Then there exist compact examples $M_i$, points $p_i\in M_i$,
and radii $r_i$ 
such that
\[
    0 < r_i < \min\{1, R(M_i,p_i), \dist(p_i, \partial M_i) \} 
\]
and
\[
  |A(M_i, p_i)| \,  r_i \to \infty
\]
Note that $M_i(p_i,r_i)$ is a compact subset of $M_i\setminus \partial M_i$.
Let $q_i$ be a point in $M_i(p_i,r_i)$ that maximizes
\begin{equation*}
  f_i(q):=|A(M_i,q)|\, (r_i - \dist(q,p_i)).
\end{equation*}
Then $f_i(q_i)\to\infty$ since $f_i(p_i)\to\infty$.

Let $\rho_i=r_i-\dist(p_i,q_i)$.  Then
\[
   |A(M_i,\cdot)| \le 2 |A(M_i,q_i)|
\]
on $M_i(q_i,\rho_i/2)$.

Now dilate $M_i-q_i$ by $|A(M_i,q_i)|$ to get $M_i'$.
After passing to a subsequence, the $M_i'$ converge smoothly to a $M'$ such
that $M'$ is a smoothly embedded minimal surface and
\begin{align*}
|A(M',0)| &=1 , \\
\dist(0,\partial M') &=\infty, \\
R(M',0) &= \infty.
\end{align*}
Thus $M'$ is complete, properly embedded (it is the limit of embedded surfaces)
 and simply connected (by Proposition~\ref{simply-connected-proposition}).
  By  Assertion~\eqref{area-item} of this theorem, $M'$ has quadratic area growth. Thus it is a plane, a contradiction,
 because $|A(M',0)|=1$.   Hence Assertion~\eqref{curvature-item} is proved.

To prove Assertion~\eqref{in-addition-item}, let $r>R(M,p)$.  Then there is a simple closed curve $C$ in $M(p,r)$ that is homotopically nontrivial in $M$.
By elementary topology (see Lemma~\ref{topology-lemma-1}), $C$ is homotopically nontrivial in $\RR^3\setminus Z$.
Let $q^+$ be a point in $C\cap \{y=0,\, x>0\}$ and $q^-$ be a point in $C\cap \{y=0,\, x< 0\}$.
Then $x(q^+)\ge x(M)$ and $x(q^-)\le -x(M)$, so
\begin{align*}
2 x(M)
&\le
|q^+ - q^-|  \\
&\le
\dist_M(q^+,p) + \dist_M(p,q^-) \\
&\le 2r.
\end{align*}
Thus,
\begin{equation}\label{r-x-of-M}
    x(M) \le r.
\end{equation}

Since
\[
   C \subset M(p,r) \subset \overline{\BB(p,r)},
\]
and since $C$ is homotopically nontrivial in $\RR^3\setminus \ZZ$, 
it follows that $\BB(p,r)\cap Z$ is nonempty, so
\begin{equation}\label{Z-r}
    \dist(p,Z) \le r.
\end{equation}
\
The inequalities~\eqref{r-x-of-M} and~\eqref{Z-r} hold for all $r>R(M,p)$.
Hence they hold for $r=R(M,p)$.
\end{proof}

\begin{lemma}[Topology Lemma]\label{topology-lemma-1}
Suppose that  $M$ is an annulus in $\RR^3\setminus Z$ that is invariant under $\rho_Z$.
Then the inclusion of $M$ into $\RR^3\setminus Z$ induces a monomorphism of first homology.

If, in addition, the  annulus $M$  is a translator and if $V$ is a vertical plane, then $M\cap V$ does not contain a closed curve.
\end{lemma}

\begin{proof}
Since $H_1(M;\ZZ)= H_1(\RR^3\setminus Z; \ZZ)=\ZZ$, It suffices to prove that $M$ contains a closed curve that is homotopically nontrivial in $\RR^3\setminus Z$.
Let $p\in M$ and $\gamma$ be a oriented curve in $M$ from $p$ to $\rho_Zp$.
Then
\[
  \int_\gamma\,d\theta
\]
is an odd multiple of $\pi$, where
\[
  d\theta:= \frac{x\,dy - y\,dx}{x^2 + y^2}.
\]
If $C=\gamma\cup \rho_Z\gamma$,
then
\[
  \int_C d\theta = \int_\gamma d\theta+\int_{\rho_Z}d\theta=2\int_\gamma d\theta.
\]
Thus $\int_C\theta$ is an odd multiple of $2\pi$, and therefore $C$ is homotopically
 nontrivial in $\RR^3\setminus \ZZ$.
 
To prove the second statement, note that a closed curve in $V$ disjoint from $Z$ is homotopically trivial in $\RR^3\setminus Z$.
(Indeed, it is homotopically trivial in $V$ if $Z\not\subset V$ and in $V\setminus Z$ if $Z\subset V$.)
If $M\cap V$ contained a closed curve $S$, it would be homotopically trival in $\RR \setminus Z$, and therefore $S$ would bound
a disk in $M$.  By the maximum principle, that disk would lie in $V$, which is impossible.  (By unique continuation, all of $M$ would lie in $V$.)
\end{proof}

\section{Morse-Rad\'{o} Theory}\label{morse-rado-section}

Let $M$ be a translator.  There are a number of standard foliations $\Ff$ of $\RR^3$ or of open subsets
 $W$ of $\RR^3$ by translators.  For example, $\Ff$ could be a family of parallel vertical planes.
Bounding the number of  points of tangency of $M$ with the leaves of $\Ff$ is the subject of Morse-Rad\'{o} Theory,
 a powerful tool
that plays a major role in this paper.  
In this section, we recall the basic facts in Morse-Rad\'{o} Theory.
More details about the results included in this section can be found in~\cite{morse-rado}.

Recall \cite{ilmanen_1994} that $M\subset \RR^3$ is a translator if and only if it is minimal with respect to the translator metric
\begin{equation}\label{translator-metric}
  g = e^{-z}(dx^2+dy^2+dz^2).
\end{equation}

\index{translator metric}

\begin{definition}\label{critical-point-definition}
Let $M$ be an embedded or immersed minimal surface in a Riemannian $3$-manifold $N$.
Let $\Ff$ be a foliation of an open subset $W$ of $N$ by minimal surfaces.
Let $\hat M$ be the union of the components of $M\cap W$ that are not contained in leaves of $\Ff$.
A {\bf critical point} of $M$ with respect to $\Ff$ is an interior point $p$ of $\hat M$ at which $M$ is tangent
to the leaf of $\Ff$ through $p$.  The {\bf multiplicity} of the critical point
is the order of contact of $M$ and the leaf.  We let
\[
   \mathsf{N}(\Ff|M)
\]
be the total number of interior critical points, counting multiplicity.
\end{definition}

\newcommand{\Nthing}{\mathsf{N}(\mathcal{F}|M)}
 \index{$\Nthing$}

Of course the case of interest in this paper is when $N$ is $\RR^3$ with the translator metric.

The first important fact about $\mathsf{N}(\Ff|M)$ is that it depends lower semicontinuously on $\Ff$ and on $M$:

\begin{theorem}[\cite{morse-rado}, Corollary 40]\label{semicontinuity-theorem}
Suppose that $g_i$ are Riemannian metrics on a $3$-manifold $N$ that converge smoothly to a Riemannian metric $g$.
Suppose that $M_i$ are $g_i$-minimal surfaces that converges smoothly to a $g$-minimal surface $M$.
Suppose $\Ff_i$ are $g_i$-minimal foliations of open subsets $W_i$ of $N$ such that the leaves of $\Ff_i$
converge smoothly to the leaves of a $g$-minimal foliation $\Ff$ of an open subset $W$ of $N$.  Then
\[
   \mathsf{N}(\Ff|M) \le \liminf \mathsf{N}(\Ff_i|M_i).
\]
In particular, if $p$ is a critical point of $(\Ff,M)$, then $p$ is a limit of critical points $p_i$ of $(\Ff_i,M_i)$.
\end{theorem}

A {\bf minimal foliation function} on $N$ is a continuous function $F$
from an open subset of $W$ of $N$ to an open interval $I\subset\RR$ such that for
each $t\in I$,
\begin{enumerate}[\upshape {}]
\item $F^{-1}(t)$ is a minimal surface, and
\item $F^{-1}(t)$ is in the closures of $\{F>t\}$ and  of $\{F<t\}$.
\end{enumerate}
If $F$ is a minimal foliation function on $N$, we let $\mathsf{N}(F|M)= \mathsf{N}(\Ff|M)$, where $\Ff$ is the foliation whose leaves
are the level sets of $F$.
\index{minimal foliation function}

\newcommand{\nboogm}{$N(F|M)$ (number of critical points, counting multiplicity of $F|M$)}
\index{\nboogm}

In the following theorem, if $S$ is a set, then $|S|$ denotes the number of elements in the set.

\newcommand{\numberofelements}{$|S|$ (number of elements of the set $S$)}
\index{\numberofelements}

\begin{theorem}[\cite{morse-rado}, Theorem 4]\label{morse-rado-theorem}
Let $F:W\subset N\to I$ be a minimal foliation function on an open subset $W$ of a Riemannian $3$-manifold $N$.
Let $M$ be a minimal surface in $N$ homeomorphic to a $2$-manifold-with-boundary.
Suppose that $F|M\cap W$ is proper,
 that the set $Q$ of local minima of $F|\partial M$ is finite, that $M\cap W$
has finite genus, and that $(\partial M)\cap \{F<t\}$ is empty for some $t\in I$.
Then
\[
\mathsf{N}(F|M) \le   |Q| - |A| - \chi(M\cap W),
\]
where  $A$ is the set of local maxima or minima of $F|\partial M$
that are not local maxima or minima of $F|M$, and where $\chi(\cdot)$ denotes Euler Characteristic.

Equivalently, 
\[
\mathsf{N}(F|M) \le  |S| - |T| - \chi(M\cap W),
\]
where $S$ is the set of local minima of $F|\partial M$ that are also local minima of $F|M$, and where
$T$ is the set of local maxima of $F|\partial M$ that are not local maxima of $F|M$.
\end{theorem}

\begin{remark} \label{re-one}
 In practice, one sometimes encounters $F$ and $M$ that satisfy all but one of the hypotheses of Theorem~\ref{morse-rado-theorem}, namely the hypothesis that the set of local minima of $F|\partial M$ is finite. 
 In particular, that hypothesis will fail if $F$ is constant on one or more arcs of $\partial M$. One can handle such examples as follows. Suppose $F$ is not constant on any connected component of $\partial M$.
  Let $\tilde M$ be obtained from $M$ by identifying each arc of $\partial M$ on which $F$ is constant to a point. 
  Let $\tilde F$ be the function on $\tilde M$ corresponding to $F$  on $M$.
   If the set $\tilde Q$ of local minima of $\tilde F|\partial \tilde M$ is finite,
then
\begin{align*}
\mathsf{N}(F|M)
&\le   |\tilde Q| - |\tilde A| - \chi(M\cap W),
\end{align*}
where $\tilde A$ is the set of local minima and local maxima of $\tilde F| \partial \tilde M$ that are not local minima or local maxima     
      of $\tilde F|M$. 
Equivalently, 
\[
\mathsf{N}(F|M) \le   |\tilde S| - |\tilde T| - \chi(M\cap W),
\]
where $\tilde S$ is the set of local minima of $\tilde F|\partial \tilde M$ that are also
local minima of $\tilde F|\tilde M$, and $\tilde T$ is the set of local maxima
of $\tilde F| \partial \tilde M$ that are not local maxima of $\tilde F|\tilde M$.
\end{remark}

We now describe the main examples of $g$-minimal foliation functions that we will use (where $g$ is the translator metric).
First, if $\vv$ is a horizontal unit vector in $\RR^3$, then
 the function 
\begin{equation}\label{definition-F_v}
\begin{aligned}
&F_\vv: \RR^3\to \RR, \\
&F_\vv(p) = \vv\cdot p
\end{aligned}
\end{equation}
is a $g$-minimal foliation function.
Second, suppose that $U$ is $\RR^2$ or an open strip in $\RR^2$ and that $h:U\to\RR$ is a function
whose graph is a complete translator.  Then
\begin{equation}\label{general-H}
\begin{aligned}
   &H: U\times \RR \to \RR, \\
   &H(x,y,z) = z - h(x,y)
\end{aligned}
\end{equation}
is a $g$-minimal foliation function.

The complete translators $M$ that we construct in this paper all have the following properties
 (see Theorem~\ref{complete-existence-theorem}):
\begin{enumerate}[\upshape (i)]
\item\label{2-bound-item} For each horizontal unit vector $\vv$, $\mathsf{N}(F_\vv|M)\le 2$.
\item\label{8-bound-item} If $H$ is as in~\eqref{general-H}, 
  then $\mathsf{N}(H|M)\le 8$.  
\end{enumerate}
(It is not hard to show that $\mathsf{N}(H|M)\le 4$ in~\eqref{8-bound-item}, but we do not need that fact.)

In the next section, we present a few useful general facts about translators that satisfy bounds such 
as~\eqref{2-bound-item} and~\eqref{8-bound-item}.

\section{Translators of Finite Type}\label{finite-type-section}
As mentioned in Section~\ref{morse-rado-section},
 translators can be seen as minimal surfaces in $\R^3$ endowed with the metric  
$g={\rm e}^{-z} (dx^2+dy^2+dz^2).$  For minimal surfaces in Euclidean $3$-space, there is a special class of surfaces that have very 
interesting properties not shared by general minimal surfaces: the surfaces with finite total curvature. 
In particular, for each such surface $M$,
 there is a positive integer $k$ such that $$\mathsf{N}(\Ff|M) \leq k,$$ for any foliation $\Ff$ of $\R^3$ by minimal 
surfaces (i.e., for any foliation by parallel planes).
 This motivates our definition of translator of finite type, which plays a similar role in this setting.
\begin{definition}\label{finite-type-definition}
We say that a translator $M$ in $\RR^3$ is {\bf of finite type}
provided there are finite numbers $c$, $K$, and $k$ such that
\begin{enumerate}
\item\label{area-item-def} $\area(M\cap\BB)\le cr^2$ for every ball $\BB$ of radius $r$.
\item\label{curvature-item-def} For every $p\in M$, 
\[
    |A(M,p)| \, \min\{1, \dist(p,\partial M)\} \le K.
\] 
\item\label{Fv-item} $\mathsf{N}(F_\vv|M)\le k$ for each horizontal unit vector $\vv$.
\item\label{H-item} $\mathsf{N}(H|M)\le k$ for every function $H(x,y,z)=z-h(x,y)$ whose level sets
are grim reaper surfaces (tilted or untilted).
\end{enumerate}
\end{definition}
We do not require any regularity at $\partial M$.

(Actually, condition~\eqref{Fv-item} is redundant: it is implied by condition~\eqref{H-item}. See Remark~\ref{redundancy}.)

This paper studies translators that are compact minimal annuli bounded by pairs of nested rectangles, as well as
translating annuli without boundary obtained as limits of such compact examples.
  All such translators are of finite type.
  (See Theorems~\ref{easy-bound-R}, \ref{x-critical-theorem}, and~\ref{complete-existence-theorem}.)

\begin{remark}\label{entropy-remark}
Condition~\eqref{area-item-def} in Definition~\ref{finite-type-definition} is equivalent (by a simple calculation) 
to the condition that $M$ have finite entropy.
See~\cite{white21}*{Theorem~9.1}.
\end{remark}

\begin{remark}\label{general-position-remark}
Let $\Hh$ be the collection of all functions $H$ as in condition~\eqref{H-item} of Definition~\ref{finite-type-definition}.  
By lower semicontinuity (Theorem~\ref{semicontinuity-theorem}), to prove
$\mathsf{N}(H|M)\le k$ for all $H\in \Hh$, it suffices to prove it for a dense set of $H\in \Hh$.  Likewise,
to prove $\mathsf{N}(F_\vv|M)\le k$ for all horizontal unit vectors $\vv$, it suffices to prove it for a dense
set of such $\vv$.
\end{remark}

We are mainly interested in translators without boundary. 
For a surface without boundary,
$\dist(p,\partial M)=\infty$, so the curvature bound in~\eqref{curvature-item-def} of Definition~\ref{finite-type-definition} simplifies to
\[
  |A(M,p)| \le K
\]
for all $p\in M$.

Translators of finite type have a nice compactness property.
If $M_n$ is a sequence of such surfaces satisfying \eqref{area-item-def}--\eqref{H-item} in Definition~\ref{finite-type-definition}, with $c$, $k$, and $K$ independent of $n$,
then, after passing to a subsequence, the $M_n$ will converge smoothly to a limit translator $M$
satisfying the bounds~\eqref{area-item-def} and~\eqref{curvature-item-def}.
(The convergence is smooth away from $\partial M$.)
 By lower semicontinuity (Theorem~\ref{semicontinuity-theorem}), $M$ will also
satisfy the bounds~\eqref{Fv-item} and~\eqref{H-item}.

\begin{theorem}\label{translates-theorem}
Suppose that $M\subset \RR^3$ is a translator 
of finite type.  If $p_i$ is a divergent sequence and if $\dist(p_i, \partial M)\to \infty$, then,
after passing to a subsequence, $M-p_i$ converges smoothly to
 a limit surface $M'$.  Furthermore, any such limit $M'$ is a union of 
vertical planes and translating graphs (grim reaper surfaces, $\Delta$-wings, and bowl solitons).

If the $p_i$ lie in a vertical plane containing $Z$, then $M'$ is a union of vertical planes and grim reaper surfaces.

If the $p_i=(0,0,z_i)$ are in $Z$, then $M'$ is a union of vertical planes.
\end{theorem}

Of course, for a surface without boundary, the condition $\dist(p_i,\partial M)\to \infty$ is vacuously true.

\begin{proof}
We may assume that $M$ is connected.  We may also assume that $M$ is not a vertical plane
(as the theorem is trivial in that case).
The curvature and area bounds
 imply smooth subsequential convergence to a limit translator $M'$
  that is properly immersed and without boundary.
Let $\vv$ be a horizontal unit vector.  By hypothesis, $F_\vv|M$ has a finite set $V$ of critical points.
Since $p_i$ diverges in $\RR^3$, the surfaces 
\[
  M_i:= (M\setminus V) - p_i
\]
converge smoothly to the same limit surface $M'$.  
By lower semicontinuity, 
\[
\mathsf{N}(F_\vv|M')
\le
\liminf \mathsf{N}(F_\vv| M_i).
\]
But
\[
 \mathsf{N}(F_\vv|M_i) = \mathsf{N}(F_\vv|(M\setminus V)) = 0.
\]
Thus
\begin{equation}\label{ZERO}
 \mathsf{N}(F_\vv|M')=0.
\end{equation}
Let $\Sigma$ be a component of $M'$ that is not a vertical plane.  By~\eqref{ZERO} (which holds
for all horizontal unit vectors $\vv$), the surface $\Sigma$ has no
points at which the tangent plane is vertical.
According to \cite{spruck-xiao}*{Corollary 1.2}, any complete connected translator with no vertical tangent 
planes is a translating graph. By Theorem~\ref{classification-theorem}, a translating graph is a grim reaper surface, a $\Delta$-wing,
or a bowl soliton.  

Now suppose that the $p_i$ lie in a vertical plane containing $Z$.
By rotating, we may assume that the plane is the plane $\{y=0\}$, and thus that $p_i=(x_i,0,z_i)$.

For $c\in \RR$, let
\[
h_c(x,y) = \log(\cos(y-c))
\]
be the untilted grim reaper surface over $\RR\times (c-\pi/2,c+\pi/2)$, and let
\begin{align*}
&H_c: \RR\times (c-\pi/2,c+\pi/2) \times \RR \to \RR, \\
&H_c(x,y,z) = z - h_c(x,y).
\end{align*}

Note that $H_c|M$ has a finite set $C$ of critical points.  Then
  $M_i':=(M_i\setminus C) - p_i$
converges to $M'$ and $\mathsf{N}(H_c|M_i')=0$, so 
\begin{equation}\label{zero}
  \mathsf{N}(H_c|M')=0.
\end{equation}

Now suppose that $\Sigma$ is a bowl soliton or a $\Delta$-wing.  Then $z(\cdot)|\Sigma$ attains its maximum
at a single point $p:=(\bar x, \bar y, \bar z)$.  Note that $p$ is a critical point of $H_{\bar y}|\Sigma$.
Therefore $\Sigma$ cannot be a component of $M'$ by~\eqref{zero}.

Now suppose that $p_i=(0,0,z_i)\in Z$.
We must show that if $\Sigma$ is a translating graph, then $\Sigma$ is not a component of $M'$.
We have already proved it when $\Sigma$ is a $\Delta$-wing or bowl soliton. 
Thus suppose that $\Sigma$ is a grim reaper surface.  By rotating, we can assume that $\Sigma$
is a grim reaper surface over $I\times \RR$ for some interval $I$.
Note that $H_0|\Sigma$ has a critical point.  Thus, by~\eqref{zero}, $\Sigma$ cannot be a component of $M'$.
\end{proof}

{
\begin{corollary}\label{translates-corollary}
If $M$ is a complete translator of finite type 
that lies in a slab $\{|y|\le B\}$ and if $p_i=(x_i,y_i,z_i)$ is a sequence of points in $M$
with $(x_i,y_i)$ bounded and with $|z_i|\to \infty$, then, after passing to a subsequence,
$\nu(M,p_i)$ converges to $\ee_2$ or to $-\ee_2$.
\end{corollary}

\begin{proof}
By Theorem~\ref{translates-theorem}, $M-(0,0,z_i)$
 converges smoothly (after passing to a subsequence) to a union
of vertical planes. Since those planes are contained in the slab $\{|y|\le B\}$, they are all normal
to $\ee_2$. The assertion follows immediately.
\end{proof}

\begin{theorem}\label{new-finite-type-theorem}
Suppose $M$ is a translator.  Let $W$ be the set of horizontal unit vectors $\vv$
such that $\mathsf{N}(F_\vv|M)>0$.  Then $W$ is an open subset of the equator $E$.  

Now suppose that $M$ has finite type, lies in a slab $\{|y|<B\}$, and has no boundary.  Then
\begin{enumerate}
\item\label{new-finite-1} If $\vv\in \partial W$, then $\vv=\pm \ee_2$.
\item\label{new-finite-2} If $\mathsf{N}(x(\cdot)|M)=0$, then each component of $M$ is either  a plane parallel to $\{y=0\}$,
or a $\Delta$-wing, or grim reaper surface.
\end{enumerate}
\end{theorem}

\begin{proof}
The openness of $W$ is because $\mathsf{N}(F_\vv|M)$ is a lower-semicontinuous function of $\vv$
   (Theorem~\ref{semicontinuity-theorem}.)
To prove Assertion~\eqref{new-finite-1}, suppose $\vv\in \partial W$. 
 Then $\vv\notin W$ but there exists a sequence of points $\vv_n\in W$
converging to $\vv$.  Let $p_n\in M$ with $\nu(M,p_n)=\vv_n$.
Then, after passing to a subsequence, $M_n'=M-p_n$ converges to a translator $M'$.
Since $\vv\notin W$, $\mathsf{N}(F_\vv|M_n')=\mathsf{N}(F_\vv|M)=0$ and therefore
 $\mathsf{N}(F_\vv|M')=0$ by lower-semicontinuity,  But $\nu(M',0)=\vv$, so the component of $M'$ containing $0$
 is a plane. (Otherwise $0$ would be a critical point of $F_\vv$ with positive multiplicity.)  Since $M'$ lies
 in the slab $\{|y|\le 2B\}$, $\vv=\pm \ee_2$.  This completes the proof of Assertion~\eqref{new-finite-1}.
 
Note that $\mathsf{N}(F_\vv|M)=\mathsf{N}(F_{-\vv}|M)$, so $W$ is invariant under $\vv\mapsto -\vv$.
Thus, by Assertion~\eqref{new-finite-1}, $\partial W$ is either the empty set or $\{\ee_2, -\ee_2\}$.   Hence, $W$ is one of the following:
$\emptyset$, $E$, or $E\setminus \{\ee_2, -\ee_2\}$. 
If $\mathsf{N}(x(\cdot)|M)=0$, i.e., if $\mathsf{N}(F_{\ee_1}|M)=0$, then $\ee_1\notin W$ so $W$ is empty.
By  the Spruck-Xiao Theorem~\cite{spruck-xiao}*{Corollary~1.2}, $M$ consists of vertical planes and graphs.
By Theorem~\ref{classification-theorem}, every translating graph is a $\Delta$-wing, a grim reaper surface, or bowl soliton.
Since $M$ lies in a vertical slab, it cannot contain a bowl soliton.
\end{proof}

\begin{theorem}\label{product-theorem}
Suppose that $M$ is a complete translator of finite type contained in a slab $\{|y|\le B\}$.
Let $\vv$ be a horizontal unit vector that is not $\pm \ee_2$.
For $t\in \RR$ and $I\subset \RR$, let $M(t)=M\cap F_\vv^{-1}(t)$ and $M(I)=M\cap F_\vv^{-1}(I)$.
Let $I$ be a connected subset of $\RR$ that does not include any critical values of $F_\vv|M$, 
and let $t_0\in I$.  Then
\[
    M(I)
\]
is diffeomorphic to
\[
   I \times M(t_0).
\]
Indeed, there is a diffeomorphism of the form
\begin{equation}\label{the-diffeomorphism}
\Phi: p \in M(I)  \mapsto (F_\vv(p), \phi(p)) \in I \times M(t_0).
\end{equation}
\end{theorem}

\begin{proof}
To simplify notation, we write $F$ in place of $F_\vv$.
Consider the tangent vector field 
\[
    V: = \frac{\nabla (F|M)}{|\nabla (F|M)|^2}
\]
on $M(I)$.

\begin{claim}\label{V-bound-claim}
If $J$ is a compact interval contained in $I$,
then $|V(\cdot)|$ is bounded above on $M(J)$: 
\[
   c_J := \sup_{M(J)} |V(\cdot)| < \infty.
\]
\end{claim}

To prove the claim, suppose it fails for some $J$.
Then there is a sequence $p_i=(x_i,y_i,z_i)$ in $M(J)$
such that $|V(p_i)|\to\infty$.  Therefore
\begin{equation}\label{to-zero}
   \nabla(F|M)(p_i) \to 0.
\end{equation}
Since $I$ contains no critical values of $F$, the sequence $p_i$
diverges. Note that $x_i$ and $y_i$ are bounded, and thus that $|z_i|\to \infty$.
By Corollary~\ref{translates-corollary}, $\nu(M,p_i)$ converges (after passing to a subsequence)
to $\ee_2$ or to $-\ee_2$.  We may choose the orientation on $M$ so 
 that $\nu(M,p_i)$ converges to $\ee_2$. 
Thus
\begin{align*}
\nabla (F|M)(p_i) 
&= \vv - (\vv\cdot \nu(M,p_i)) \nu(M,p_i)  \\
&\to \vv - (\vv\cdot \ee_2) \ee_2,
\end{align*}
which is nonzero since $\vv\ne \pm\ee_2$.
But that contradicts~\eqref{to-zero}.
Thus we have proved Claim~\ref{V-bound-claim}.

For $p\in M(I)$, let $t\in I \mapsto q_p(t)$ be the solution 
to the following initial value problem.  The ODE is 
$q_p'(t) = V(q_p(t))$, and the initial condition is that $q_p(t)=p$
at time $t=F(p)$.

Note that
\[
   (d/dt) F(q_p(t))  = \nabla F \cdot V  = 1,
\]
so 
\[
   F(q_p(t)) \equiv t + c
\]
for some constant $c$.  Putting $t=F(p)$, we see that $F(p)=F(p)+c$, so $c=0$.  Thus
\[
   F(q_p(t)) \equiv t.
\]
That is, at each time $t$, the point $q_p(t)$ is in the level set $M(t)$.

The bound on  $|V(\cdot)|$ in Claim~\ref{V-bound-claim}  imply that during a compact time interval $J\subset I$,
$q_p(t)$ traces out a curve of length at most $c_J|J|$.
  Thus the solution $t\mapsto q_p(t)$ exists for the entire interval $t\in I$.

Now we define the diffeomorphism $\Phi$ in~\eqref{the-diffeomorphism} by letting
\[
   \phi(p) = q_{p}(t_0).
\]
The inverse $\Psi$ of $\Phi$ is given by
\begin{align*}
   &\Psi: M(t_0) \times I \to M(I), \\
   &\Psi(p,t) = q_p(t).
\end{align*}

\end{proof}

\begin{theorem}\label{first-symmetry-theorem}
Suppose that $M$ is a translator of finite type and  that $M$ has no boundary.
Suppose also that $M$ lies in a vertical slab $\{|y|\le B\}$
and that $M$ is invariant under $(x,y,z)\mapsto (-x,y,z)$.
Then as $z\to \infty$ (or as $z\to -\infty$), the surfaces $M+(0,0,z)$
converge smoothly to a finite union of vertical planes.
\end{theorem}

Note that we get convergence, not just subsequential convergence.
The limit as $z\to\infty$ will, in general, be different from the limit as $z\to -\infty$.
For example, if $M$ is a complete translating graph in a slab, then $M+(0,0,z)$
converges as $z\to\infty$ to a pair of parallel planes, and $M+(0,0,z)$
converges as $z\to -\infty$ to the empty set.

\begin{proof}
Note that $M$ intersects the plane $\{x=0\}$ orthogonally, and thus
\[
  \Gamma:= M\cap \{x=0\}
\]
is a smooth curve.  Let $\vv=\ee_2$. Any critical point of $F_\vv|\Gamma$ is also a critical 
point of $F_\vv|M$.  Thus $F_\vv|\Gamma$ has only finitely many critical points.
Consequently, if $t\in [0,\infty)\mapsto \gamma(t)=(0,y(t),z(t))$ is a parametrization of an end of $\Gamma$,
then $y(t) =  F_\vv(\gamma(t))$ is eventually mononotic and thus has a well-defined limit as $t\to\infty$.

It follows that $\Gamma+(0,0,z)$ converges as $z\to\infty$ to a union $\cup_iL_i$ of vertical lines that are contained in the plane $\{x=0\}$ and in the slab $\{|y|\le B\}$.  Let $P_i$ be the plane in the slab $\{|y|\le B\}$
such that $P_i\cap \{x=0\}=L_i$.

Now let $M'$ be any subsequential limit of $M+(0,0,z)$ as $z\to\infty$.  
   By Theorem~\ref{translates-theorem}, $M'$ is a union of vertical planes.  We know that
\begin{equation}\label{the-lines}
  M'\cap \{x=0\} = \cup_i L_i.
\end{equation}
It follows that 
\begin{equation}\label{the-planes}
  M'=\cup_i P_i.
\end{equation}
Since the limit~\eqref{the-planes} does not depend on choice of subsequence, we get convergence (and not just subsequential 
convergence) of $M+(0,0,z)$ to $M'$.

The same proof works for convergence of $M+(0,0,z)$ as $z\to -\infty$.
\end{proof}
 
\begin{theorem}\label{second-symmetry-theorem}
Let $M$ be a connected translator of finite type. Assume that $M$ is invariant under reflection in the
plane $\{y=0\}$, and that $M$ is not contained in $\{y=0\}$.  Then
\begin{enumerate}
\item $\Gamma:=(M\setminus \partial M)\cap \{y=0\}$ is a smooth $1$-manifold.
\item\label{2nd-symmetry-slope-item} If 
\[
   t\in [0,\infty)\mapsto \gamma(t)=(x(t),0,z(t))
\]
is an arclength parametrization of an end of 
   $\Gamma$, then $\gamma'(t)$ converges as $t\to\infty$ to a limit $\uu=(x'(\infty),0,z'(\infty))$.
\end{enumerate}
Furthermore, if $M$ lies in the slab $\{|y|\le B\}$ and if
$
  \dist(\gamma(t), \partial M)\to \infty
$,
then 
\begin{enumerate}\setcounter{enumi}{2}
\item\label{second-symmetry-3}
   $M-\gamma(t)$ converges subsequentially to a limit $M'$, and the component $\Sigma$
of $M'$ containing the origin is the grim reaper surface that contains the line $L:=\{s\uu: s\in \RR\}$
and that is symmetric about the plane $\{y=0\}$.
\item\label{second-symmetry-4}
    $|z'(\infty)| \le s(B)\, |x'(\infty)|$.  
\item\label{second-symmetry-5}
  $x(t)$ tends to $\infty$ (if $x'(\infty)>0$) or to $-\infty$ (if $x'(\infty)<0$).
\end{enumerate}
\end{theorem}

\begin{proof}[Proof of Theorem~\ref{second-symmetry-theorem}]
Because $M$ intersects the plane $\{y=0\}$ orthogonally, $\Gamma$ is a smooth $1$-manifold.

Let $m\in \RR$.  Let $\beta\ge\pi/2$ be such that there is a grim reaper function
\[
  h: \RR\times (-\beta,\beta)\to \RR
\]
with $\pdf{}xh\equiv m$.  Then 
\[
  \mathsf{N}(H|M)\le k < \infty
\]
where $H(x,y,z):=z-h(x,y)$ (as in Definition~\ref{general-H})
and where $k$ is as in Definition~\ref{finite-type-definition}.
Each point on $\Gamma$ where the tangent line is parallel to the line $\{z=mx, \, y=0\}$ is a critical
point of $H|M$.  Hence there are at most $k$ such points.

It follows immediately that $\gamma'(t)$ converges to a limit $\uu$.

Now suppose that $M$ lies in the slab $\{|y|\le B\}$.
By Theorem~\ref{translates-theorem},
 $M-\gamma(t)$ converges smoothly, perhaps after passing to subsequence, to a limit $M'$
consisting of vertical planes and grim reaper surfaces. 
The component $\Sigma$ containing the origin is orthogonal to the plane $\{y=0\}$.  
Thus if it were a vertical plane, it would be the plane $\{x=0\}$, which is impossible
since $\Sigma$ is contained in the slab $\{|y|\le B\}$.
Thus $\Sigma$ is a grim reaper surface.
Now $\Sigma$ contains the line $L:=\{t\uu: t\in \RR\}$ and intersects the plane $\{y=0\}$ orthogonally
along $L$.
Thus it is
the unique grim reaper surface with those properties.

Since the grim reaper surface $\Sigma$ is contained in the slab $\{|y|\le B\}$, the slope
of the line $L$ is at most $s(B)$ in absolute value.  Thus Assertion~\eqref{second-symmetry-4} holds.

By Assertion~\eqref{second-symmetry-4}, $x'(\infty)\ne 0$.  Thus Assertion~\eqref{second-symmetry-5} holds.
\end{proof}

\begin{remark}\label{redundancy}
As mentioned earlier, 
the condition~\eqref{Fv-item} on $F_\vv$ in the definition of finite type is redundant: it is implied
by condition~\eqref{H-item}, as we now explain.
By rotating, it suffices to consider the case $\vv=\ee_1$.  For $b\ge \pi/2$, let
\[
   h_b: \RR\times (-b,b) \to \RR
\]
be the grim reaper surface with $h_b(0,0)=0$ and $\pdf{}xh_b\ge 0$.  
Let $\Ff_b$ be the foliation of the slab $\{|y|<b\}$ by surfaces $h_b=\textnormal{constant}$, 
i.e., by the level sets of $H_b(x,y,z):= z- h_b(x,y)$.
Since we are assuming condition~\eqref{H-item} Definition~\ref{finite-type-definition}, 
\[
    \mathsf{N}(\Ff_b|M) = \mathsf{N}(H_b|M) \le k.
\]
As $b\to \infty$, the foliation $\Ff_b$ converges to the foliation $\Ff$ of $\RR^3$ by the level sets of $F_{\ee_1}$.  
Thus, by lower semicontinuity,
\[
   \mathsf{N}(F_{\ee_1}|M) = \mathsf{N}(\Ff|M) \le \liminf_{b\to\infty} \mathsf{N}(\Ff_b|M) \le k.
\]
\end{remark}

\section{The Space $\Rr$ of Annuli with Rectangular Boundaries}

\begin{definition}\label{Cc-definition}
We define $\Cc$ to be the space of compact,  embedded, translating annuli $M$ such that
\begin{enumerate}[\upshape(1)]
\item The boundary of $M$ is a pair of disjoint, nested, convex closed curves in the plane $\{z=0\}$:
\item $M$ is invariant under reflection in the planes $\{x=0\}$ and $\{y=0\}$.
\item\label{cc-disjoint-item} $M$ is disjoint from the $z$-axis.
\end{enumerate}
We let $\partialin M$ and $\partialout M$ denote the inner and outer components of $\partial M$.
We define $a(M)$, $b(M)$, $A(M)$, and $B(M)$ to be the positive numbers such that
\begin{align*}
(a(M),0,0) &\in \partialin M, \\
(A(M),0,0) &\in \partialout M, \\
(0,b(M),0) &\in \partialin M, \\
(0,B(M),0) &\in \partialout M.
\end{align*}
\end{definition}
See Figures~\ref{fig:x(M)} and~\ref{curve-1}.

\begin{figure}[htbp]
\begin{center}
\includegraphics[width=.66\textwidth]{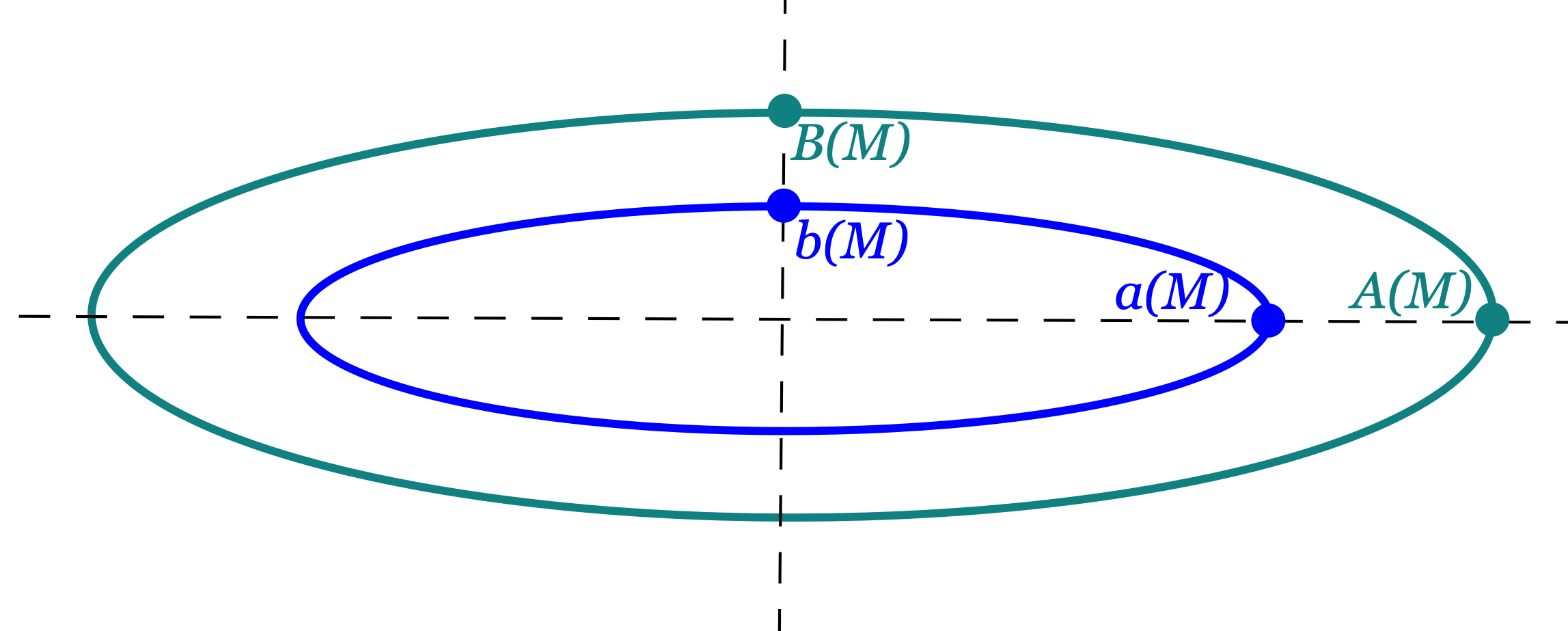}
\caption{\small $\partialin M$ and $\partialout M.$}
\label{curve-1}
\end{center}
\end{figure}

Condition~\eqref{cc-disjoint-item} in Definition~\ref{Cc-definition} is redundant, but we include it here for convenience.

\begin{definition}\label{Rr-definition}
We define $\Rr$ to be the space of $M\in \Cc$ such that
\begin{enumerate}[\upshape(1)]
\item\label{rectangle-condition} $\partialin M$ are $\partialout M$ are rectangles whose sides are
  parallel to the coordinate axes, and
\item\label{approximation-condition} $M$ is the limit of a sequence of tranlators $M_n\in \Cc$ 
  such that for each $n$, $\partialin M_n$ and $\partialout M_n$ are smooth with nowhere vanishing curvature.
\end{enumerate}
\end{definition}

(Condition~\eqref{approximation-condition} might be redundant: perhaps every $M\in \Cc$ with property~\eqref{rectangle-condition}
 also has property~\eqref{approximation-condition}.)

When discussing a surface $M$ in $\Rr$, 
we will often write $a$, $b$, $A$, and $B$
in place of the more cumbersome $a(M)$, $b(M)$, $A(M)$, and $B(M)$.
Likewise, if $M_n$ is a sequence of surfaces in $\Rr$, we will often
write $a_n$, $b_n$, $A_n$, and $B_n$ in place of $a(M_n)$, $b(M_n)$, $A(M_n)$,
and $B(M_n)$.
For $M\in \Rr$, the boundary is completely determined by
the four numbers $a:=a(M)$, $b:=b(M)$, $A:=A(M)$, and $B:=B(M)$:
\begin{align*}
\partialin M &= \partial( [-a,a]\times [-b,b] ), \\
\partialout M &= \partial( [-A,A] \times [-B,B] ).
\end{align*}

\begin{theorem}\label{easy-bound-C}
Suppose $M\in \Cc$.
Let $U\subset \RR^2$ be an open strip of width $\pi$ and let $h:U\to \RR$ be an untilted grim reaper.
Let
\begin{align*}
&H: U\times\RR \to \RR, \\
&H(x,y,z)=z-h(x,y).
\end{align*}
Then
\[
  \mathsf{N}(H|M) \le 4.
\]
\end{theorem}

\begin{proof}
There are at most countably many lines that contain a segment in $\partial M$.
It suffices to prove the Theorem for strips $U$ that are not parallel to any of those lines;
the general case follows by lower semicontinuity.

Let $L$ be the line in $U$ that bisects $U$.  Let $C$ be a closed convex curve in the plane $\{z=0\}$.
If $L$ passes through the interior of $C$, then $H|C$ has exactly two local minima, 
namely the two points in $L\cap C$.  Otherwise, $L$ has at most one local minimum,
namely the point in $C$ closest to $L$ (if that point is in $U$.)

Thus $|S|\le 4$ in the formula
\[
 \mathsf{N}(H|M) \le |S|-|T|-\chi(M\cap (U\times\RR)).
\]
Also, since $M$ is an annulus, $\chi(M\cap (U\times\RR))\ge 0$.
 (If this is not clear, see Lemma~\ref{topology-lemma} below.)
Thus $\mathsf{N}(H|M)\le 4-0=4$.
\end{proof}

\begin{theorem}\label{easy-bound-R}
Suppose $M\in \Rr$.
Let $h:U\to \RR$ be a complete translator.
(Thus $h$ is a grim reaper surface, a $\Delta$-wing, or a bowl soliton.)
Let
\begin{align*}
&H: U\times\RR \to \RR, \\
&H(x,y,z)=z-h(x,y).
\end{align*}
Then
\[
  \mathsf{N}(H|M) \le 8.
\]
\end{theorem}

\begin{proof}
In case $U$ is a strip, we may assume that $U$ is not parallel to either coordinate
axis; the general case then follows by lower semicontinuity (Theorem~\ref{semicontinuity-theorem}).  

Note that the restriction of $h$ to any line segment not parallel to $U$ is strictly concave.
  Thus $h$ has at most one local maximum on each of the eight edges of $\partial M$,
so $h|\partial M$ has at most $8$ local maxima.  The local maxima of $h|\partial M$ are
the local minima of $H|\partial M$, so $|S|\le 8$ in the formula
\[
  \mathsf{N}(H|M) \le |S|-|T|-\chi(M\cap (U\times\RR))
\]
from Theorem~\ref{morse-rado-theorem}.
Also, since $M$ is an annulus, $\chi(M\cap (U\times\RR))\ge 0$.
 (If this is not clear, see Lemma~\ref{topology-lemma} below.)
Thus $\mathsf{N}(H|M)\le 8-0=8$.
\end{proof}

\begin{lemma}\label{topology-lemma}
Suppose that $M$ is a translator and that
 $U\subset \RR^2$ is a convex set.
Then the inclusion of $M\cap (U\times\RR)$ into $M$ induces a monomorphism of first homology.
In particular, if $M$ is an annulus, then at most one component of $M\cap (U\times\RR)$ is an annulus,
and the non-annular components are disks.
   If $M\in \Cc$ and if $0\notin U$, then each component of $M\cap (U\times \RR)$ is a disk.
\end{lemma}

\begin{proof}
Let $C$ be a closed embedded curve in $M\cap (U\times\RR)$ that is homologically trivial
in $M$. Then $C$ bounds a $2$-chain in $M$.  Let $\Sigma$ be the support of that $2$-chain. By the strong maximum principle, for each horizontal unit vector
$\vv$, the maximum of $F_\vv|\Sigma$ occurs on $C$. It follows that $\Sigma$ lies in $U\times\RR$, 
and thus that $C$ is homologically trivial in $M\cap(U\times\RR)$.

Now suppose that $M\in \Cc$ and that $0\notin U$.  Each closed curve in $M\times (U\times\RR)$ is homotopically trivial 
in $\RR^3\setminus \ZZ$, and therefore homotopically trivial in $M$ by Lemma~\ref{topology-lemma-1}.
Hence it is homotopically trivial in $M\cap (U\times \RR)$.  Thus each component of $M\times (U\times\RR)$ is a disk.
\end{proof}

\begin{theorem}\label{x-critical-theorem}
Suppose that $M\in \Cc$ and that $\vv$ is a horizontal unit vector.
Then 
\begin{enumerate}
\item\label{either-item} $\mathsf{N}(F_\vv|M)$ is either $0$ or $2$.  If $\mathsf{N}(F_\vv|M)=2$, there are two critical points, each with multiplicity $1$ 
(i.e., with Gauss curvature $<0$.)
\item\label{at-most-item} There is at most one interior point of $M$ at which $\nu(p)=\vv$.
   If $\vv=\pm \ee_1$ then $y(p)=0$, and if $\vv=\pm \ee_2$, then $x(p)=0$.
\end{enumerate}
\end{theorem}

\begin{proof}
First we claim that
\[
  \mathsf{N}(F_\vv|M) \le 2.
\]
Except for a countable set $\mathcal{V}$ of unit vectors $\vv$, the function $F_\vv|\partial M$ has exactly
two local minima, so by the counting formula (Theorem~\ref{morse-rado-theorem}),
\begin{equation}\label{twosome}
\mathsf{N}(F_\vv|M) \le 2 - 0 -\chi(M) = 2
\end{equation}
if $\vv\notin \mathcal{V}$.  
 By lower-semicontinuity (Theorem~\ref{semicontinuity-theorem}),
  \eqref{twosome} holds for all horizontal unit vectors $\vv$.

If $F_\vv|M$ has an interior critical point $p$, then $\rho_Zp$ is also an interior critical point, so each has multiplicity $1$.  Since $p$ has multiplicity one, the order of contact at $p$ between $M$ and the plane $\{F_\vv=F_\vv(p)\}$ 
is $2$.  Thus the Gauss curvature of $M$ at $p$ is nonzero.  Since the mean curvature at $p$
is $(-\ee_3)^\perp=0$, the Gauss curvature is negative.
Thus we have proved Assertion~\eqref{either-item}.

To prove Assertion~\eqref{at-most-item}, note that the interior critical points of $F_\vv|M$ are the points where $\nu=\pm\vv$.
If $\nu(p)=\vv$, then $\nu(\rho_Zp)= -\vv$, and, by Assertion~\eqref{either-item}, there are no other points where $\nu=\pm \vv$.
If $\nu(p)=\ee_1$ (or if $\nu(p)=-\ee_1$), then $y(p)=0$, since otherwise the image of $p$ under $(x,y,z)\mapsto (x,-y,z)$ would be another point
at which $\nu=\ee_1$ (or $\nu= -\ee_1$).  The analogous argument applies when $\nu(p)=\pm \ee_2$.
\end{proof}

\begin{corollary}\label{x-critical-corollary}
Suppose that $M_n\in \Cc$ and that $M_n':=M_n-p_n$ converges to a limit $M$.
Let $\tilde M$ be the set of points $p\in M\setminus \partial M$ such that the convergence is smooth at $p$
and such that the component of $\tilde M$ containing $p$ is not contained in a vertical plane.
Suppose $\vv$ is a horizontal unit vector.  
Then there is at most one point $q\in \tilde M$ such that $\nu(\tilde M,q)=\vv$.
Furthermore, the Gauss curvature of $M$ at such a point $q$ is $<0$,
and
if the vertical plane $V_\vv$ containing $\vv$ is a plane of symmetry of $M$, then $q\in V$.
\end{corollary}

\begin{proof}
   Let $M_n^+= \{p\in M_n: \nu(p) \cdot \vv>0\}$ 
 and
  $\tilde M^+ = \{p\in \tilde M: \nu(p) \cdot\vv>0\}$.
By Theorem~\ref{x-critical-theorem}, $\mathsf{N}(F_\vv|M_n^+)\le 1$,
 so (by lower semicontinuity) $\mathsf{N}(F_\vv|\tilde M^+)\le 1$.
Thus there is at most one point in $\tilde M$ at which $\nu=\vv$, and if $q$ is such a point, it is a critical point of multiplicity $1$
and thus the Gauss curvature at $q$ is negative. 

Note that if $V_\vv$ is a plane of symmetry of $M$, then $q\in V_\vv$ since otherwise the image of $q$ under reflection in $V_\vv$
would be a second point at which $\nu=\vv$. 
\end{proof}

\section{The Space $\Aa$ of Annuloids}\label{space-A-section}

We will show the following by a path-lifting argument:

\begin{theorem}\label{existence-theorem}
Suppose that $b \ge \pi/2$ and that $0<\hat x < a$.
Then there exists a translator $M\in \Rr$ (not necessarily unique) such that
\begin{align*}
a(M) &= a, \\
b(M) &= b, \\
x(M) &= \hat x.
\end{align*}
\end{theorem}

We postpone the proof to \S\ref{gap-section}--\ref{connected-section};
 see Lemma~\ref{connected-lemma}.
  However, interested readers may skip to \S\ref{gap-section}
   after finishing this section (\S\ref{space-A-section}); 
  the proof of Theorem~\ref{existence-theorem} (i.e., of Lemma~\ref{connected-lemma}) 
  does not depend on anything in the intervening sections.
  In any case, there is no risk of circularity in postponing that proof.  In the intervening sections, we prove various properties
of complete $M$ that arise as limits of $M_n\in \Rr$, but nowhere do we assume existence in those proofs.
(If existence did not hold, then the assertions would be vacuously true.)

We now wish to see what happens to the surface $M$ in Theorem~\ref{existence-theorem}
 if we fix $b$ and $\hat x$ and let $a\to\infty$.

\begin{definition}\label{z-of-M-definition}
If $M\in \Cc$, we let $z(M)$ be the largest value of $z$ such that $(x(M),0,z)\in M$.
\end{definition}

(In fact, there is only one $z$ for which $(x(M),0,z)$ is in $M$;
see Corollary~\ref{unique-z-corollary}.)

\begin{theorem}\label{complete-existence-theorem}
Let $M_i$ be a sequence in $\Rr$ such that 
\begin{gather*}
b_i:=b(M_i)\to b\in [\pi/2,\infty), \\
x(M_i) \to\hat x \in (0,\infty), \\
a_i:=a(M_i)\to \infty.
\end{gather*}
Then, after passing to a subsequence, the surfaces
\[
  M_i':= M_i - (0, 0, z(M_i))
\]
converge smoothly to a complete translator $M$.  
Furthermore,
\begin{enumerate}
\item $M$ is symmetric with respect to the planes $\{x=0\}$ and $\{y=0\}$.
\item $x(M)=\hat x$ and $(\hat x, 0,0)\in M$.
\item $\area(M\cap \BB(p,r)) \le c_1r^2$ for every $p\in \RR^3$ and $r\ge 0$.
\item $|A(M,\cdot)| \min\{1,x(M)\} \le c_2$.
\item $|A(M,p)|\, \min\{1, \dist(p,Z)\} \le c_2$.
\item\label{complete-existence-twosome} $\mathsf{N}(F_\vv|M)\le 2$ for every horizontal unit vector $\vv$.
\item $\mathsf{N}(H|M)\le 8$ for every function $H(x,y,z)=z-h(x,y)$ such that the graph of $h$ is a complete
translator.
\item\label{complete-8} The connected component of $M$ containing $(\hat x, 0, 0)$ is an annulus.
\end{enumerate}
Here $c_1$ and $c_2$ are the constants in Theorem~\ref{curvature-bound-theorem}.
\end{theorem}

Concerning Assertion~\eqref{complete-8}, we will show later (Theorem~\ref{some-properties-theorem})
  that $M$ is connected and therefore that it is an annulus.

\begin{proof}
First, we claim that
\begin{equation}\label{pushed-up}
  z(M_i)\to \infty.
\end{equation}
To see this, let $f_{a,b}: [-a,a]\times [-b,b]\to \RR$ be the translator with boundary values $0$.
By the maximum principle,
\[
   M_i\cap ([-a_i,a_i]\times [-b_i,b_i]\times \RR)
\]
lies in $\{z\ge u_{a_i,b_i}(x,y)\}$.  As $i\to \infty$, $f_{a_i,b_i}(0,0)\to \infty$ and 
\[
     f_{a_i,b_i}(x,y) - f_{a_i,b_i}(0,0)
\]
converges smoothly to a translator 
\[
   f_b: \RR\times (-b,b)\to \RR.
\]
(See~\cite{graphs}*{Theorem~4.1}).  Hence if $K$ is a compact subset of $\RR\times (-b,b)$, then 
\[
  \min_K f_{a_i,b_i} \to \infty.
\]
Thus~\eqref{pushed-up} holds.

By~\eqref{pushed-up},
\begin{equation}\label{boundary-away}
  \dist( 0, \partial M_i') \to \infty.
\end{equation}
Consequently, the curvature and area bounds in Theorem~\ref{curvature-bound-theorem} give smooth convergence (after passing to a subsequence)
of $M_i'$ to a limit translator $M$.  By~\eqref{boundary-away}, $M$ has no boundary.
  From the construction, $(\hat x, 0,0)\in M$.
Also, $M$ is disjoint from the strip $(-\hat x, \hat x)\times\{0\}\times \RR$.  Thus $x(M)=\hat x$.
 All the assertions other than Assertion~\eqref{complete-8} follow trivially from the corresponding properties of the $M_i'$.
See Theorems~ \ref{semicontinuity-theorem} and~\ref{curvature-bound-theorem}.

Let $\Sigma$ be the connected component of $M$ containing $(\hat x,0,0)$.
Let $f:\RR^2\to \RR$ be the bowl solition.  Then, for any $\lambda$, each component 
of 
\begin{equation}\label{goofy}
    M_n'\cap \{-\lambda < z < f(x,y) + \lambda\} \tag{*}
\end{equation}
is a disk or an annulus.  

(If that is not clear, note that if one or more closed curves in~\eqref{goofy} bound a region $K$ in $M_n'$,
then $f|K$ must attain its maximum on $\partial K$ and $z(\cdot)|K$ must attain its minimum on $\partial K$,
and thus $K$ is contained in~\eqref{goofy}. Hence the inclusion of~\eqref{goofy} into $M$ induces
a monomorphism of first homology.)

By smooth convergence, each component of
\[
   M\cap \{ - \lambda < z < f(x,y)+\lambda\}
\]
is a disk or an annulus.  Letting $\lambda\to\infty$, we see that $\Sigma$ is a disk or an annulus.

Thus it suffices to show that $\Sigma$ is not simply connected.
No complete translator is contained in a half-slab of the form $\{|y|\le B\}\cap \{x< a\}$.
   (See Corollary~\ref{half-slab-corollary}.)
Thus $\Sigma\cap\{x=0\}$ is nonempty.  Let $C$ be a shortest path in $\Sigma\cap\{x\ge 0\}$ from $(x(M),0,0)$ 
to a point in $\Sigma\cap\{x=0\}$.  Extending by reflection in $\{x=0\}$ and then in $\{y=0\}$ gives
a simple closed curve $C'$ in $M$ that winds once around the $z$-axis.  Since $C'$ is homologically nontrivial
in $\RR^3\setminus Z$, it is homologically nontrivial in $\Sigma$. 
Thus $\Sigma$ is an annulus.  
\end{proof}

\begin{definition}\label{Aa-definition}
Let $\pi/2\le b \le B \le b+\pi$ and $\hat x\in (0,\infty)$.
We define $\Aa(b,B,\hat x)$ to be the space of all limits of sequences 
\[
    M_i - (0,0,z(M_i))
\]
such that 
\begin{gather*}
   M_i\in \Rr, \\
   a(M_i)\to\infty,
   \\
   b(M_i)\to b, 
   \\
   B(M_i) \to B, 
   \\
   x(M_i)\to \hat x.
\end{gather*}
We let   
\begin{align*}
  \Aa(b,\hat x) &:= \bigcup_B \Aa(b,B, \hat x), 
  \\
  \Aa(b) &:= \bigcup_{\hat x} \Aa(b, \hat x) 
  \\
  \Aa  &:= \bigcup_b\Aa(b)  \\
     &=\bigcup_{b,B,\hat x} \Aa(b,B,\hat x).
\end{align*}
\end{definition}

\begin{remark}\label{existence-remark}
Let $b\ge \pi/2$ and $\hat x\in (0,\infty)$.
By Theorems~\ref{existence-theorem} and~\ref{complete-existence-theorem}, $\Aa(b,\hat x)$ is nonempty, and the surfaces in
 $\Aa(b,\hat x)$ have the properties
listed in Theorem~\ref{complete-existence-theorem}.
In particular, the surfaces in $\Aa$ are translators of finite type (as defined in Section~\ref{finite-type-section}). 
\end{remark}

For surfaces in $M$ in $\Rr$ or in $\Aa$, there is a curvature bound that gives more information than the bounds
in Theorem~\ref{curvature-bound-theorem} and~\ref{complete-existence-theorem}
when $x(M)$ is close to $0$:

\begin{theorem}\label{better-curvature-theorem}
Suppose that $M\in \Rr$ or $M\in \Aa$. 
If $\vv$ is a horizontal unit vector, let 
\[
  \delta(M,\vv,p) = \min\{ |p-q|: q\in M, \, \nu(M,q)=\vv\}.
\]
(Here $\delta(M,\vv,p)=\infty$ if there is no such point $q$.)
Then
\[
      |A(M,p)| \, \min \{1,  \delta(M,\vv,p), \dist(p,\partial M) \}  \le C.
\]
\end{theorem}

\begin{proof}
It suffices to prove the theorem for $M\in \Cc$ with smooth boundaries, as any surface in $\Rr$ is a limit of such surfaces.

Thus, if the theorem is false, there is a sequence $M_n\in \Cc$ with smooth boundaries such that
\[
    \sup_{p\in M_n} |A(M_n,p)|\, \min\{1, \delta(M_n,\vv, p), \dist(p, \partial M_n) \} \to \infty.
\]
For each $n$, the supremum is attained at some point $p_n$.

Translate $M_n$ by $-p_n$ and then dilate by $|A(M_n,p_n)|$ to get $M_n'$.
Thus $|A(M_n',0|=1$ and 
\begin{gather}
   \delta(M_n', \vv, 0) \to \infty, 
   \label{far-away}
   \\
   \dist(0,\partial M') \to \infty.
\end{gather}

Then $M_n'$ converges smoothly (after passing to a subsequence) to a complete, embedded minimal surface $M'$ with 
$|A(M',0)|=1$.  Since $M'$ has genus $0$ and quadratic area growth, it is a catenoid.

If the axis of $M'$ is not parallet to $\vv$, then there is a point $q\in M'$ such that $\nu(M',q)=\vv$.  The curvature
of $M'$ at $q$ is non zero, so $q$ is the limit of points $q_n\in M_n'$ such that $\nu(M_n',q_n')=\vv$.
But that is excluded by~\eqref{far-away}.

Thus the axis of $M'$ is parallel to $\vv$, so the plane $\vv^\perp$ intersects $M'$ transversely in a circle.
Consequently, for large $n$, $\vv^\perp \cap M_n'$ contains a simple closed curve.  But that is impossible 
  by Lemma~\ref{topology-lemma-1}.
\end{proof}

\begin{remark}\label{blowup-remark}
Suppose that $M_n$ is a sequence  of surfaces in $\Rr$ or in $\Aa$, that $M_n'=M_n-q_n$ converges to a limit $M$,
and that the convergence is not smooth at a point $p\in M\setminus\partial M$. 
By Theorem~\ref{better-curvature-theorem}, the two points in $M_n'$ where $\nu=\pm \ee_1$ both have to converge to $p$.
Hence there can be at most one point in $M\setminus\partial M$ where the convergence $M_n'\to M$ is not smooth.
\end{remark}

\begin{theorem}\label{better-compactness-theorem}
Suppose that $M_n\in \Rr$ or $M_n \in \Aa$,
 and that the $M_n$ lie in a slab $\{|y|\le \lambda\}$.
Let $p_n^\pm$ be the points where $\nu(M_n,p_n^\pm)= \pm \ee_1$.
Suppose that $q_n$ are points in the plane $\{y=0\}$ such that
\[
  \dist(q_n, \{p_n^+, p_n^-\}\cup \partial M_n) \to \infty.
\]
Then, after passing to a subsequence, $M_n'=M_n - q_n$ converges smoothly to a limit $M'$, each component
of which is a plane parallel to $\{y=0\}$, or a $\Delta$-wing, or a grim reaper surface.

Now suppose also that each $q_n=(x_n,0,z_n)$ is in $M_n$, and that $x_n\ge 0$.
Then the component $\Sigma$ of $M'$ containing $0$ is either a $\Delta$-wing
or a grim reaper surface.  Furthermore, in this case,
\[
   x_n - x(M_n) \to \infty.
\]
\end{theorem}  

\begin{proof}
The curvature bound in Theorem~\ref{better-curvature-theorem} 
implies that we get smooth subsequential convergence, and that the limit surface has bounded 
principal curvatures.  For each bounded open set $U$, $\mathsf{N}(x(\cdot)|M_n'\cap U)=0$ for all sufficiently large $n$,
so $\mathsf{N}(x(\cdot)|M'\cap U)=0$.  Since $U$ is arbitrary, $\mathsf{N}(x(\cdot)|M')=0$.
By Theorem~\ref{new-finite-type-theorem}, the components of $M'$ are planes parallel to $\{y=0\}$, $\Delta$-wings,
and grim reaper surfaces.

Now suppose also that each $q_n$ is in $M_n\cap\{y=0\}$ and that $x_n\ge 0$.  Then $x_n\ge x(M_n)$
be definition of $x(M_n)$.  The component $\Sigma$ of $M'$ containing $0$ is perpendicular to $\{y=0\}$ at $0$
and is contained in the slab $\{|y|\le \lambda\}$, so it cannot be a plane.
Hence $\Sigma$ must be a $\Delta$-wing or a grim reaper surface.   In particular $\Sigma$ is the graph of
a function
\[
  w: \RR\times(-\beta,\beta) \to \RR
\]
for some $\beta$.

Now the component $\Gamma_n$ of $M_n'\cap \{y=0\}$ containing $0$ converges to $\Sigma\cap\{y=0\}$
as $n\to\infty$. 

Note that the minimum value of $x(\cdot)$ on $\Gamma_n$ is $x(M_n) - x_n$.
If $x(M_n)-x_n$ were bounded below by $\alpha$, then $x(\cdot)$ would be bounded below by $\alpha$ on $\Sigma\cap\{y=0\}$
(that is, on $\{(x,0,w(x,0)): x\in \RR\}$), which is not the case.  Thus $x(M_n)-x_n\to -\infty$,
so $x_n - x(M_n)\to \infty$.
\end{proof}

\TOCstop
\section*{What are the next steps?}
\TOCstart

According to Remark~\ref{existence-remark}, 
 for each $b\ge \pi/2$ and $\hat x\in (0,\infty)$, there is a complete
smooth translator $M$ in $\Aa(b,\hat x)$.  We know that it has an annular component
 (Theorem~\ref{complete-existence-theorem}),
but it is not obvious that $M$ is connected.  Also, although we have examples for every $\hat x\in (0,\infty)$ and
$b\ge \pi/2$, it is not obvious that these examples are distinct.  We will eventually prove
that $\Aa(b,\hat x)$ and $\Aa(b',\hat x')$ are disjoint unless $b=b'$ and $\hat x=\hat x'$.
But a priori there is the possibility that $\Aa(b,\hat x)=\Aa(\pi/2,\hat x)$ for all $b\ge \pi/2$.
(We do already know that $x(M)=\hat x$ for $M\in \Aa(b,\hat x)$.  Thus if $\hat x\ne \hat x'$, then
$\Aa(b,\hat x)$ and $\Aa(b',\hat x')$ are disjoint.)

  Our goal now is to prove that every surface $M$ in $\Aa(b,B,\hat x)$ has the properties
described in Theorem~\ref{theorem-1}.  In particular, we wish to show that
\begin{enumerate}
\item\label{assertion-1} $M$ is connected (and therefore is an annulus).
\item\label{assertion-2} $M+(0,0,z)$ converges as $z\to -\infty$ to the
   empty set.
\item\label{assertion-3} There are nonnegative numbers $b(M)$ and $B(M)$ with $b(M)\le B(M)$
   such that $M+(0,0,z)$ converges smoothly at $z\to \infty$ to the
   planes  $\{y=\pm b(M)\}$ and $\{y=\pm B(M)\}$.
\item\label{assertion-4} $B(M)= B$.
\item\label{assertion-5}  $b(M)=b$.
\end{enumerate}

We will show~\eqref{assertion-1}, \eqref{assertion-2}, and~\eqref{assertion-3}
 in the next three sections; see Theorems~\ref{connected-theorem} and~\ref{some-properties-theorem}.

Assertions~\eqref{assertion-4} and~\eqref{assertion-5} seem to be much more subtle.
We deduce them from other properties of $M$, 
 which we describe and prove in Sections~\ref{vertical-section}, \ref{inner-outer-section},
 and~\ref{down-behavior-section}.
Assertions~\eqref{assertion-4} and~\eqref{assertion-5} 
 are proved in Theorem~\ref{b-B-theorem}.
\begin{remark}[Entropy of annuloids]
Assertion \ref{assertion-3} above implies that the entropy of $M$ is 4. This is a consequence
of Corollary 8.5 in \cite{GMM}.
\end{remark}

\section{The Slice $M\cap\{y=0\}$.} \label{sec:slice}

Let $M\in \Aa$.  To analyze $M$, it is helpful to analyze the slice $M\cap\{y=0\}$.
To analyze that slice, it suffices by symmetry to analyze
\[
 \Gamma := M\cap \{y=0\}\cap \{x>0\}.
\]

\begin{theorem}\label{y-slice-theorem}
Let $M\in \Aa$.  Then $\Gamma$ is the union of two graphs
\[
    \{ (x,0,\uup(x)): x\in [x(M),\infty)\}
\]
and
\[
   \{ (x,0,\ulow(x)): x\in [x(M),\infty) \}
\]
where $\uup$ and $\ulow$ are continuous on $[x(M),\infty)$,
smooth on $(x(M),\infty)$, and where
\begin{enumerate}[\upshape (1)]
\item\label{y-slice-1} $\ulow(x(M))=\uup(x(M))=0$.
\item\label{y-slice-2} $\ulow(x)< \uup(x)$ for all $x>x(M)$.
\item\label{y-slice-3} The limits
\[
    (\uup)'(\infty):=\lim_{x\to\infty}(\uup)'(x)
\]
and 
\[
   (\ulow)'(\infty):=\lim_{x\to\infty} (\ulow)'(x)
\]
exist and are finite.
\end{enumerate}
\end{theorem}

Later (Theorem~\ref{main-theorem-concluded}) we will prove that $(\ulow)'(\infty)=-s(b(M))$ and that $(\uup)'(\infty) =\pm S(B(M))$,
where $s(\beta)$ is the slope $\partial u_\beta/\partial x$ of the grim 
 reaper surface 
   $u_\beta: \RR\times (-\beta,\beta)\to \RR$ in \S\ref{graphs-section}.

\begin{proof}
Since $\{y=0\}$ intersects $M$ orthogonally, $\Gamma$ is a smooth $1$-manifold.

By Theorem~\ref{complete-existence-theorem}~\eqref{complete-existence-twosome}
   (see also Remark~\ref{existence-remark}),
\[
  2 \ge \mathsf{N}(F_{\ee_1}| M).
\]
Now each critical point $(x,0,z)$ of $F_{\ee_1}|\Gamma$ is a critical point of $F_{\ee_1}|M$, 
as is its mirror image $(-x,0,z)$.
Thus $F_{\ee_1}|\Gamma$ has at most one critical point. Since $(x(M),0,0)$ is a critical point,
$F_{\ee_1}|\Gamma$ has exactly one critical point.
  (Note that the point $(x(M),0,0)$ is in $M$ by Definition~\ref{Aa-definition}.) 
Thus no component of $\Gamma$ is a closed curve.
(We also know from Lemma~\ref{topology-lemma-1} that no component
 of $\Gamma$ is a closed curve.) 

By Theorem~\ref{second-symmetry-theorem},  $x(\cdot)$ tends to $\infty$ or $-\infty$ on each end of $\Gamma$.
Since $x>0$ on $\Gamma$, we see that $x\to\infty$ on each end of $\Gamma$.
Thus each component of $\Gamma$ has at least one local minimum of $F_{\ee_1}|\Gamma$.
Since $F_{\ee_1}|\Gamma$ has only one critical point, the curve 
 $\Gamma$ has only one component. (It is nonempty since it contains the
point $(x(M),0,0)$.)

Since $F_{\ee_1}|\Gamma$ has its minimum at $(x(M),0,0)$ and has no other critical points, and since
$x\to\infty$ on each end of $\Gamma$, we see that $\Gamma$ is the union of two graphs
satisfying~\eqref{y-slice-1} and~\eqref{y-slice-2}.
By Theorem~\ref{second-symmetry-theorem}, it also satisfies~\eqref{y-slice-3}.
\end{proof}

Theorem~\ref{y-slice-theorem} has an analog for surfaces in $\Rr$, or more generally, for
surfaces in $\Cc$:

\begin{theorem}\label{finite-y-slice-theorem}
Suppose that $M\in \Cc$.  Write $a=a(M)$ and $A=A(M)$.
If $x(M)<a$, 
then $M\cap \{y=0\}\cap \{x>0\}$
is the union of
\[
   \{ (x,0,\ulow(x)): x(M)\le x \le a\}
\]
and
\[
   \{ (x,0,\uup(x)): x(M)\le x \le A\},
\]
where 
\begin{enumerate}
\item $\uup$ and $\ulow$ are continuous and are smooth for $x>x(M)$.
\item $\ulow(x(M))=\uup(x(M))=z(M)$ and $\ulow(x)<\uup(x)$ for $x>x(M)$.
\item $\ulow(a)=\uup(A)=0$.
\end{enumerate}
If $x(M)=a$, then $M\cap\{y=0\}\cap\{x>0\}$ is
\[
  \{ (x,0,\phi(x)): a\le x \le A\}.
\]
for a function $\phi$ that is continuous on $[a,A]$ and smooth on $(a,A)$.
\end{theorem}

We omit the proof, since it is very similar to the proof of
 Theorem~\ref{y-slice-theorem} (but simpler because behavior as $x\to\infty$ does not arise.)

Of course, the same theorem holds with the roles of $x$ and $y$ reversed, that is, for $M\cap\{x=0\}\cap\{y>0\}$.

\begin{corollary}\label{unique-z-corollary}
If $M\in \Cc$, then there is a unique $z$ such that $(x(M),0,z)\in M$.
\end{corollary}

   Recall (Definition~\ref{z-of-M-definition}) that $z(M)$ is equal to the $z$ in 
Corollary~\ref{unique-z-corollary}.

\begin{corollary}\label{just-1-or-2-corollary}
Suppose $M\in \Cc$.
For each $x$, the vertical line $\{(x,0)\}\times\RR$ intersects $M$ in at most $2$ points.
For each $y$, the vertical line $\{(0,y)\}\times\RR$ intersects $M$ in at most $2$ points.
For each $y\in (b,B]$ ,the vertical line $\{(0,y)\times\RR\}$ interects $M$ in exactly one point.
\end{corollary}

\begin{theorem}\label{components-theorem}
Suppose that $M$ is a surface in $\Cc$ or in $\Aa$. If $M\in\Cc$, suppose also that $x(M)<a(M)$.
Let  $\Mlow$ be the component of $M\cap \{x>x(M)\}$ containing
\[
  \graph(\ulow) \cap \{x>x(M)\}
\]
and let $\Mup$ be the component of $M\cap\{x>x(M)\}$ containing
\[
  \graph(\uup)\cap \{x>x(M)\}.
\]
Then $\Mlow\ne \Mup$.

Furthermore, if $M\in \Cc$, then 
\begin{equation}\label{2-components}
  M\cap \{x>x(M)\} = \Mlow\cup\Mup.
\end{equation}
\end{theorem}

We will show later (Theorem~\ref{connected-theorem}) that \eqref{2-components} also holds for $M\in \Aa$.
(Knowing~\eqref{2-components} for $M\in \Rr$ does not immediately imply it for $M\in \Aa$
because, in general, a limit of connected sets need not be connected.)

\begin{proof}
It suffices to prove that $\Mlow\ne \Mup$ for $M\in \Cc$; the result for $M\in \Aa$ follows easily.
Suppose, contrary to the theorem, that $\Mup=\Mlow$.
Then there is a path $C$ in $M\cap\{x>x(M)\}$ joining a point in $\graph(\ulow)$ to a point in $\graph(\uup)$.
We can choose the path so that $C$ is embedded and so that $C\cap\{y=0\} = \partial C$.
Let $C_0$ be the arc in $M\cap\{y=0\}$ joining the endpoints of $C$.
Since $C\cup C_0$ is homologically trivial in $\RR^3\setminus Z$, it is homologically trivial in $M$. 
(See Lemma~\ref{topology-lemma-1}.)  Thus, since $M$ is an annulus, $C\cup C_0$ bounds a simply-connected region $D$ in $M$.
Now the union $\Dd$ of $D$ and its image under $(x,y,z)\mapsto (x,-y,z)$ will be region whose boundary
lies in $\{x>x(M)\}$.  By the maximum principle,
\[
  \min_{\Dd}x(\cdot) = \min_{\partial \Dd} x(\cdot).
\]
The left-side is $\le x(M)$ since $(x(M),0,z(M))\in \Dd$, whereas the right-side is $>x(M)$,
a contradiction.  Thus $\Mup\ne \Mlow$.

Now we show (assuming $M\in \Cc$) that~\eqref{2-components} holds.
 Note that $\Mup$ contains $\partialout M\cap \{x>x(M)\}$
and $\Mlow$ contains $\partialin M\cap\{x>x(M)\}$.
If there were another component $\Sigma$ of $M\cap\{x>x(M)\}$, it would have no boundary in $\{x>x(M)\}$.
But then $x(\cdot)|\Sigma$ would attain its maximum at an interior point, violating the strong maximum
principle. Thus there is no such component, so~\eqref{2-components} holds.  
\end{proof}

\section{A Slope Bound}

For $b\ge \pi/2$, let
\[
f_b: \RR\times (-b,b) \to \RR
\]
be the translator such that $f(0,0)=0$ and $Df(0,0)=0$.  Thus $f_b$ is an untilted grim reaper if $b=\pi/2$
and a $\Delta$-wing if $b>\pi/2$.

The goal of this section is to prove

\begin{theorem}\label{slopes-theorem}
If $M\in \Rr$, then 
\begin{equation}\label{slopes-inequality}
   (\ulow)'(x) < \pdf{}x f_b(x,0)   \le 0
\end{equation}
for all $x\in [x(M),a]$, where $b=b(M)$.

If $M\in \Aa(b)$, then the inequality~\eqref{slopes-inequality} holds for all $x\ge x(M)$.

In particular, $\ulow$ is a strictly decreasing function.
\end{theorem}

(In Theorem~\ref{slopes-theorem}, we require that $M\in \Rr$ and not merely that $M\in \Cc$.
In fact, the theorem is false if we replace $\Rr$ by $\Cc$.)

\begin{corollary}\label{slopes-corollary}
If $M\in \Aa(b)$, then
\[
   \lim_{x\to\infty} (\ulow)'(x) \le -s(b).
\]
\end{corollary}

\begin{proof}[Proof of Corollary~\ref{slopes-corollary}]
Let $x\to \infty$ in~\eqref{slopes-inequality}.  (The limit of $(\ulow)'(x)$ exists by
Theorem~\ref{y-slice-theorem}).
\end{proof}

Later (Theorem~\ref{main-theorem-concluded}) we will show that equality holds in Corollary~\ref{slopes-corollary}.

\begin{proof}[Proof of Theorem~\ref{slopes-theorem}]
It suffices to prove the assertion for $M\in \Rr$, as the result for $M\in \Aa(b)$ follows trivially.

For $\alpha,\beta>0$, let 
\[
f_{\alpha,\beta}: [-\alpha,\alpha]\times[-\beta,\beta]\to\RR
\]
be the translator with boundary values $0$.

Now suppose that $M\in \Rr$. 

Let
\begin{align*}
&\alpha:[-\infty,\infty]\to [a,A] \\
&\beta:[-\infty,\infty] \to [b,B]
\end{align*}
be surjections that are continuous and strictly increasing.
Let $U$ be the region in the upper halfspace $\{z>0\}$ between the graphs 
of $f_{a,b}$ and $f_{A,B}$.
By the maximum principle, the interior of $M$ lies in $U$.

For $t\in [-\infty,\infty]$, let $f^t=f_{\alpha(t),\beta(t)}$.

Note that the graphs of $f^t$ with $t\in \RR$ form a foliation $\Ff$ of $U$.
Let $T:U\to \RR$ be the function such that
\[
   T(p) = t   \quad \text{for $p\in \graph(f^t).$}
\]
Let $M'=M\setminus \partial M$.
Since $T|M'$ is proper and since $M'$ includes no boundary points,
\[
\mathsf{N}(T|M') \le \chi(M') = 0
\]
by Theorem~\ref{morse-rado-theorem}.
That is, $M'$ is everywhere transverse to the foliation $\Ff$.

For $p\in U$, note that $\graph(f^{T(p)})$ is the leaf that passes through $p$.
For $x\in [x(M),a)$, let
\[
  u(x) = \ddx f^T(x,0) - (\ulow)'(x), \quad\text{where $T=T(x,0,\ulow(x))$}.
\]
Now $M$ and $\graph(f^T)$ are not tangent at $p=(x,0,\ulow(x))$.  Since $\ee_2$ is tangent to $M$ and to
$\graph(f^T)$ at $p$, we see that $M\cap \{y=0\}$ and $\graph(f^T)\cap\{y=0\}$ are not tangent at $p$.
In other words, $u(x)\ne 0$.
We have shown that $u(\cdot)$ never vanishes on $[x(M),a)$.
Now $u(x)= \infty$ for $x=x(M)$.  Thus $u(x)>0$ for all $x\in [x(M),a)$.
Consequently, $u(x)\ge 0$ for all $x\in [x(M),a]$.
Thus if $x\in [x(M),a]$ and if $T=T(x,0,\ulow(x))$, then
\begin{align*}
(\ulow)'(x) 
&\le \ddx f^T(x,0)  \\
&= \ddx f_{\alpha(T),\beta(T)}(x,0) \\
&<  \ddx f_b(x,0),
\end{align*}
where the last inequality is by Lemma~\ref{slopes-lemma} below.
This completes the proof of Theorem~\ref{slopes-theorem}.
\end{proof}

\begin{lemma}\label{slopes-lemma}
If $\pi/2 \le b \le \beta$, then
\[
   \pdf{}x f_{\alpha,\beta}(x,0) < \pdf{}x f_b(x,0)
\]
for all $x\in (0,\alpha]$.
\end{lemma}

To prove the lemma, it suffices to prove that
\begin{equation}\label{thing-1}
\text{$\ddx f_{\alpha,\beta}(x,0)\ne \ddx f_b(x,0)$ for each $x\in (0,\alpha)$}, \tag{i}
\end{equation}
and that
\begin{equation}\label{thing-2}
\text{$\ddx f_{\alpha,\beta}(\alpha,0) < \ddx f_b(\alpha,0)$.} \tag{ii}
\end{equation}

\begin{proof}[Proof of~\eqref{thing-1}]
Let $\Sigma=\graph(f_{\alpha,\beta})$.  

Define $F_b$ on the slab $W:=\{|y|<b\}$ by
\[
F_b(x,y,z)=z - f_b(x,y).
\]
Now $F_b|\partial \Sigma$ has exactly two critical points, namely $(\alpha,0,0)$ and $(-\alpha,0,0)$.
Thus
\begin{align*}
\mathsf{N}(F_b|\Sigma)
&\le
2 - \chi(\Sigma\cap\{|y|<\beta\} \\
&= 1
\end{align*}
Thus $F_b|\Sigma$ has at most one critical point.  
Equivalently, the function
\[
  w:=f_{\alpha,\beta} - f_b
\]
has at most one critical point in the rectangle $R:=(-\alpha,\alpha)\times(-\beta,\beta)$.
Since $(0,0)$ is a critical point, it is the only one.

Thus if $x\in (0,\alpha)$, then $Dw(x,0)\ne 0$.
Now $\partial_y w(x,0) = 0$ since $w(x,y)\equiv w(x,-y)$.
Thus
\[
  \partial_x w(x,0)\ne 0.
\]
This completes the proof of~\eqref{thing-1}.
\end{proof}

\begin{proof}[Proof of~\eqref{thing-2}]
Let $u= f_b - f_{\alpha,\beta}$ on $[-\alpha,\alpha] \times (-\beta,\beta)$.
Now $u$ is proper and bounded above, so it attains a maximum.
By the maximum principle, the maximum is attained at the boundary,
that is at a point $(\pm \alpha, y)$. By symmetry, it is attained at $(\alpha,y)$
for some $y$.  In fact, the maximum of $u(\alpha,y)$ occurs at $y=0$.

Thus we have shown
\[
 \max u = u(A,0).
\]
By the strong maximum principle, $u_x(A,0)>0$, i.e., that $\ddx f_b(A,0)> \ddx f_{\alpha,\beta}(A,0)$.
This completes the proof of~\eqref{thing-2}, and thus the proof of Lemma~\ref{slopes-lemma}.
\end{proof}

\section{Sideways Graphicality}\label{sideways-section}

We now show that large portions of $M\in \Cc$ and $M\in \Aa$ are graphical sideways in that they can be expressed
as $y=y(x,z)$.

In this section, we let
\begin{align*}
&h_c(x,y) = \log(\cos(x-c)) \qquad\text{for $x\in I_c:=(c-\pi/2, c+\pi/2)$}, \\
&H_c(x,y,z) = z - h_c(x,y) \qquad\text{for $(x,y,z)\in W_c:=I_c\times \RR^2$}. 
\end{align*}
Thus the graph of $h_c$ is an untilted grim reaper over the strip $I_c\times\RR$.

\begin{proposition}\label{first-y-graph-proposition}
Suppose that $M$ is a surface in $\Cc$ or in $\Aa$,
and that
\[
   c\ge \hat c:= x(M)+\pi/2.
\]
Then
\[
  \mathsf{N}(H_c| M\cap\{y\ne 0\})=0.
\]
\end{proposition}

\begin{proof}
By lower semicontinuity (Theorem~\ref{semicontinuity-theorem}), it suffices to prove it for $M\in \Cc$.
By Theorem~\ref{components-theorem}, $M\cap \{x>x(M)\}=\Mup \cup \Mlow$,
so it suffices to prove that
\[
 \mathsf{N}(H_c|\Sigma\cap\{y\ne 0\})=0
\]
for $\Sigma=\Mup$ and for $\Sigma=\Mlow$.
The proofs for the two cases are essentially the same, so we give the proof  for $\Mlow$.

We may assume that $c< a(M)+(\pi/2)$, as otherwise $\Mlow\cap W_c$ is empty.

If $c< a(M)$, then $H_c|\partial \Mlow$ has exactly two local minima, namely the two
points on $\{x=c\}\cap\partialin M$.

If $c\in [a(M), a(M)+\pi)$, then $H_c|\partial \Mlow$ has exactly one
local minimum, namely $\{x=c\}\cap \partialin M$, which is either a point or an interval.
(If it is an interval, we identify that interval to a point 
as in Remark~\ref{re-one}.)
Thus 
\[
\mathsf{N}(H_c|\Mlow) \le |S|-|T|-\chi(\Mlow\cap W) \le 2 - 0 -1 =1.
\]
(Note that $\chi\ge 1$ by Lemma~\ref{topology-lemma}.)
Consequently, 
\begin{align*}
1 
&\ge \mathsf{N}(H_c|\Mlow) \\
&\ge \mathsf{N}(H_c|\Mlow\cap \{y\ne 0\}) \\
&=  2 \mathsf{N}(H_c| \Mlow\cap \{y>0\}).
\end{align*}
It follows immediately that $\mathsf{N}(H_c|\Mlow\cap\{y\ne 0\})=0$.
\end{proof}

\begin{corollary}
Suppose that $p=(x_0,y_0,z_0)$ is an interior point of $M$ and that $x_0 \ge \hat c:=x(M)+\pi$.
Then $\ee_2 \in \Tan(M,p)$ if and only if $y_0=0$.
\end{corollary}

\begin{proof}
The ``if'' assertion follows from the fact that $\{y=0\}$ is a plane of symmetry of $M$.  Suppose that $\ee_2\in \Tan(M,p)$.
Then there is a grim reaper surface $G$
that is tangent to $M$ at $p$ and that contains the line $\{p+s\ee_2: s\in \RR\}$.
Note that $G$ is the level set $\{H_c=H_c(p)\}$ for some $H_c$ as in Proposition~\ref{first-y-graph-proposition}.
Now $p$ is a critical point of $H_c|M$, so $y_0=0$ by Proposition~\ref{first-y-graph-proposition}.
\end{proof}

\begin{theorem}\label{y-graph-theorem}
Suppose $M\in \Cc$ or $M\in \Aa$. Let $\hat c= x(M)+\pi$.
If $M\in \Aa$, let me know 
\begin{align*}
\Omegain(M) &= \{(x,z):  \text{$x\ge \hat c$ and $z\le \ulow(x)$} \}, \\
\Omegaout(M) &= \{(x,z): \text{$x \ge \hat c$ and $z \le \uup(x)$}  \}.
\end{align*}
If $M\in \Cc$, let
\begin{align*}
\Omegain(M) &= \{(x,z):  \text{$x\in [\hat c, a)$ and $0\le z\le \ulow(x)$} \}, \\
\Omegaout(M) &= \{(x,z): \text{$x \in  [\hat c, A)$ and $0 \le z \le \uup(x)$}  \}.
\end{align*}
There exist nonnegative, continuous functions $\yin$ on $\Omegain$ and $\yout$ on $\Omegaout$ whose 
graphs $y=\yin(x,z)$ and $y=\yout(x,z)$ are contained in $M$.
Furthermore, $M\cap\{x\ge \hat c\}\cap\{y\ge 0\}$ has exactly two components,
namely $\graph(\yin)$ and $\graph(\yout)$ if $M\in \Aa$, or
\[
\graph(\yin) \cup (\partialin M\cap\{x=a\})
\]
and
\[
\graph(\yout) \cup (\partialout M\cap\{x=A\})
\]
if $M\in \Cc$.

The functions $\yin$ and $\yout$ are strictly positive except 
where $z=\ulow(x)$ (for $\yin$) and $z=\uup(x)$ (for $\yout$).
The functions are smooth on the interior of their domains.
\end{theorem}

\begin{proof}
It suffices to prove it for $M\in \Cc$, since the $M\in \Aa$ case follows by taking limits.

For each $x_0\in [\hat c, a)$, the height function $z(\cdot)$ on the
 smooth curve $\Mlow\cap \{x=x_0\}\cap \{y\ge 0\}$
has only one critical point, namely the endpoint  $(x_0,0,\ulow(x_0))$.
Thus $z(\cdot)$ maps the curve homeomorphically onto $[0,\ulow(x)]$. 
Hence there is a continuous, nonnegative function $\yin$ on $\Omegain$
such that $\yin(x,z)\in \Mlow$ for each $(x,z)\in \Omegain$.
The graph $\{(x,\yin(0,z),z): (x,z)\in \Omegain\}$ of $\yin$ is a smooth manifold (namely
$\Mlow\cap\{ \hat c\le x < a\}$)
and $\ee_2$ is not tangent to that manifold except where $y=\ulow(x)$.
Thus $\yin$ is a smooth function except where $y=\ulow(x)$. 
This completes the proof in the case of $\yin$.

The same proof works for $\yout$.
\end{proof}

\begin{corollary}\label{y-graph-corollary}
Suppose $p\in M \cap \{x\ge \hat c\} \cap \{y>0\}$.  Then
\begin{align*}
\nu(M,p)\cdot \ee_2>0 &\qquad\text{if $p\in \Mup$}, \\
\nu(M,p)\cdot \ee_2<0 &\qquad\text{if $p \in \Mlow$}.
\end{align*}
\end{corollary}

\begin{theorem}\label{connected-theorem}
Suppose $M\in \Aa$.
Then 
\begin{enumerate}
\item\label{connected-item-1} $M$ is connected.
\item\label{connected-item-2} $M\cap \{x>x(M)\}$ has exactly two connected components, $\Mlow$ and $\Mup$.
\end{enumerate}
\end{theorem}

\begin{proof}
Let $M'$ be the connected component of $M$ containing 
\[
M\cap \{y=0\}\cap \{x>0\}.
\]
By Theorem~\ref{y-graph-theorem}, $M\setminus M'$ is contained in the half-slab $\{|y|\le B\}\cap \{x< x(M)+\pi\}$.
But there are no nonempty translators (without boundary) in such a half-slab (Corollary~\ref{half-slab-corollary}), so
$M\setminus M'$ is empty and therefore
 $M$ is connected.  Thus we have proved
Assertion~\eqref{connected-item-1}.

Now let 
\[
 \Sigma = (M\cap \{x>x(M)\})\setminus (\Mup\cup \Mlow).
\]
Suppose, contrary to Assertion~\eqref{connected-item-2}, that $\Sigma$ is nonempty.
Then 
\begin{equation}\label{henry-oliver}
 x(M) <   \sup_{\Sigma}x(\cdot).
\end{equation}
Now $\partial \Sigma$ is nonempty by Assertion~\eqref{connected-item-1}, and $x(\cdot)\equiv x(M)$ on $\partial \Sigma$, so
\begin{equation}\label{wish}
   \sup_{\partial \Sigma} x(\cdot) = x(M).
\end{equation}

By Theorem~\ref{y-graph-theorem},
\[
   \sup_\Sigma x(\cdot) \le x(M)+\pi.
\]
Thus, by a version of the maximum principle in slabs (Corollary~\ref{half-slab-corollary}),
\begin{align*}
\sup_{\Sigma}x(\cdot) 
&= \sup_{\partial \Sigma}x(\cdot)  \\
&= x(M)
\end{align*}
by~\eqref{wish}.
 But this contradicts~\eqref{henry-oliver}.
  Hence Assertion~\eqref{connected-item-2} holds.
\end{proof}

\begin{theorem}\label{some-properties-theorem}
Suppose $M\in \Aa$.  Then
\begin{enumerate}[\upshape (1)]
\item $M$ is an annulus.
\item\label{up-planes-item} $M-(0,0,z)$ converges as $z\to \infty$ to the empty set.
\item\label{down-planes-item} There are numbers $b(M)$ and $B(M)$ with $0\le b(M)\le B(M)$ such that 
  $M+(0,0,z)$ converges smoothly as $z\to\infty$ to the planes $\{y=\pm b(M)\}$ a
 and $\{y=\pm B(M)\}$.
\item For each $x\ge \hat c:=x(M)+\pi$,
\begin{align*}
\lim_{z\to -\infty} \yin(x,z) &= b(M), \\
\lim_{z\to -\infty} \yout(x,z) &= B(M).
\end{align*}
\end{enumerate}
\end{theorem}

\begin{proof}
Since $M$ is connected and has an annular component (Theorem~\ref{complete-existence-theorem}), it is an annulus.

By Theorem~\ref{first-symmetry-theorem}, $M-(0,0,\zeta)$ converges smoothly as $\zeta\to\infty$ to a limit $M_\infty$ that is a finite
union of planes in the slab $\{|y|\le B\}$.
By Theorem~\ref{y-graph-theorem}, $M_\infty$ lies in the region $\{|y|\le x(M)+\pi\}$.  Thus $M_\infty$ is empty.

By Theorem~\ref{first-symmetry-theorem}, $M+(0,0,\zeta)$ converges smoothly as $\zeta\to\infty$ to a finite
union $M_{-\infty}$ of planes.    From Theorem~\ref{y-graph-theorem}, we see that the number of planes
is four (counting multiplicity).  Thus the planes have the form $\{y=\pm b(M)\}$ and $\{y=\pm B(M)\}$ for some $0\le b(M)\le B(M)$.
It follows trivially that $\yin(x,z)\to b(M)$ and $\yout(x,z)\to B(M)$ for every $x\ge \hat c$.
\end{proof}

\begin{theorem}\label{alexandrov-y-theorem}
Suppose that $M$ is a surface in $\Cc$ (see Definition~\ref{Cc-definition}) with smooth, uniformly convex
boundary curves, or that $M$ is a surface in $\Rr$.
Then 
the surface $S:=M\cap\{z>0\}\cap \{y\ge (b+B)/2\}$
projects diffeomorphically  onto its image in the $xz$-plane, and $M$ is disjoint from the open
region in $\{z>0\}$ between $S$ and its image under reflection in  the plane $\{y=(b+B)/2\}$.
\end{theorem}

\begin{proof}
For $M$ with smooth boundary, the proof is by the standard 
standard Alexandrov moving planes argument, so we omit it.  The theorem for $M\in \Rr$  follows
immediately by taking limits.
\end{proof}

\begin{corollary}\label{in-out-corollary}
Suppose $M\in \Rr$.  Then
\[
   \yin + \yout \le b + B
\]
at all points of $\domain(\yin)$.
\end{corollary}

\begin{theorem}\label{x-alexandrov-theorem}
Suppose that $M\in \Cc$ with smooth, strictly convex boundaries, or that
 $M\in \Rr$.
Then $\Mlow\cap\{z>0\}$ projects diffeomorphically onto its image in the $yz$
plane.
In particular, in the domain of $\yin$, 
for each $z$, $\yin(x,z)$ is a decreasing function of $x$.
\end{theorem}

\begin{proof}
First consider the case that $M\in \Cc$ with smooth, strictly convex boundaries.
The proof in this case is the following standard Alexandrov argument. 
For $t\in (x(M),a(M)]$, let $\Sigma(t)$ be the image of $\Mlow\cap\{x\ge t\}$
under reflection in the plane $\{x=t\}$.
Let $\Tt$ be the set of $t\in [x(M),a(M)]$ such that $\Sigma(t)$ is disjoint
from $M\cap\{x<t\}$ and such that $\nu\cdot \ee_1<0$ at all points of $\Sigma(t)$.
Now $\Tt$ is nonempty since $t\in \Tt$ for all $t$ sufficiently close to $a(M)$.
Clearly $\Tt$ is an open subset of $(x(M),a(M))$.  By the strong maximum principle,
it is also a relatively closed subset.
The statement for $\Rr$ follows by taking the limit.
\end{proof}

%%%
%%%
%%%

%%%
%%%
%%%

\section{Vertical Graphicality}\label{vertical-section}

In this section, we show that if $M\in \Rr$, then a large portion of $M$ can be written
as a graph $z=u_M(x,y)$.

\begin{definition}[Waist]\label{waist-definition}
If $M$ is a translator and if $\tilde M$ is the set of non-flat components of $M$ (i.e., the set of components that are
not contained in vertical planes), then $\waist(M)$ is the set of points in $\tilde M$ at which the tangent
plane is vertical.
\end{definition}

For $M\in \Cc$, then no component of $M$ is flat, and therefore $\waist(M)$ is the preimage
of the equator
 \[
    E:= \{(x,y,z)\in \SS^2: z=0\}
 \]
under the Gauss map.  If $p$ is in the waist, then the Gauss Curvature of $M$ at $p$
is negative (by Theorem~\ref{x-critical-theorem}),
 so $\nu$ maps a neighborhood of $p$ in $M$ diffeomorphically onto an open
set in $\SS^2$.  Thus $\waist(M)$ is a smooth curve, and $\nu$ is a smooth
immersion from $\waist(M)$ onto its
 image.
  By Theorem~\ref{x-critical-theorem}, the map 
$$\nu|_{\waist(M)}: \waist(M) \longrightarrow E$$ is one-to-one,
so $\nu$ is a diffeomorphism from $\waist(M)$ onto an open subset
 of $E$.

\begin{theorem}\label{waist-theorem}
Suppose that $M$ is a surface in $\Cc$ with smooth boundary curves, or a surface in $\Rr$.
Suppose that $p$ is a point in $M\setminus\partial M$ where $\Tan(M,p)$ is vertical.
If $x(p)\ne 0$, then $\nu\cdot\ee_1$ and $x(p)$ have opposite signs, 
and if $y(p)\ne 0$, then $\nu\cdot\ee_2$ and $y(p)$ have opposite signs.
\end{theorem}

(Theorem~\ref{waist-theorem} is true, with essentially the same proof for any $M\in \Cc$, but the lack
of regularity of $\partial M$ makes the proof a bit more involved.)

\begin{proof}
If $y(p)=0$, then $p$ must be one of the points $(\pm x(M),0,z(M))$,
and the assertion holds at those points.
Likewise, it holds if $x(p)=0$.

Thus we may assume that neither $x(p)$ nor $y(p)$ is $0$.  By symmetry, 
it suffices to consider the case when $x(p)>0$ and $y(p)>0$.
Let $J$ be the connected component of $\waist(M)\cap\{x> 0, \, y> 0\}$ containing $p$.

Let $q$ be an endpoint of $J$.  We claim
that
\begin{equation}\label{waist-claim}
 \text{$\nu(q)\cdot \ee_i\le 0$ for $i=1,2$}. \tag{*}
\end{equation}
To prove~\eqref{waist-claim}, note that if $q$ is an endpoint of $J$, then one
of the following must hold:
\begin{enumerate}
\item\label{case-1}  $q$ is a point in $\partial M$ where $\Tan(M,q)$ is vertical, or
\item\label{case-2} $q$ is a  point in $M\setminus \partial M$ with $y(q)=0$, or
\item\label{case-3} $q$ is a point in $M\setminus \partial M$ with $x(q)=0$.
\end{enumerate}
In case~\eqref{case-1}, $q$ is in $\partialin M$ since, by the maximum principle, $\Tan(M,\cdot)$ is never
vertical on $\partialout M$.  Thus~\eqref{waist-claim} holds in this case.

In case~\eqref{case-2}, by Theorem~\ref{y-slice-theorem}, $x=x(M)$ and $q=(x(M),0,z(M))$, so $\nu(q)= -\ee_1$
and therefore~\eqref{waist-claim} holds.

Likewise, in case~\eqref{case-3}, $\nu(q)=-\ee_2$, so~\eqref{waist-claim} holds.

Now $\nu$ is a homeomorphism of $\overline{J}$ onto its image.  
Note that if $(x,y,z)\in J$, then $\nu(x,y,z)$ cannot be $\pm \ee_1$, since if it were,
 $\nu(x,y,z)$ and $\nu(x,-y,z)$ would be equal, contrary to Theorem~\ref{x-critical-theorem}.
 For the same reason, $\nu(x,y,z)$ cannot be $\pm \ee_2$.
Thus $\nu(J)$ lies in one of the four components of 
\[
  E \setminus \{\ee_1,\ee_2, -\ee_1,-\ee_2\}.
\]
By~\eqref{waist-claim}, 
  the endpoints of $\nu(\overline{J})$ lie in $\{v\in E: v\cdot\ee_1\le 0, \, v\cdot \ee_2\le 0\}$.
Thus $J$ lies in $\{v\in E: v\cdot \ee_1<0, \, v\cdot \ee_2<0\}$.
\end{proof}

\begin{definition}\label{u-definition}
Suppose $M$ is in $\Cc$ or $\Aa$.
We define a function $u=u_M$ as follows.
The domain of $u_M$ is the set of $(x,y)$ such that $\{z:(x,y,z)\in M\}$ is nonempty and
has a greatest element $\zeta$ and such that the tangent plane to $M$ at $(x,y,\zeta)$ is not vertical.
For $(x,y)\in \domain(u_M)$, we let
\[
u_M(x,y) = \max\{ z: (x,y,z)\in M\}.
\]
\end{definition}

If $M$ is smooth and $\partialout M$ is smooth, then $u_M$ is smooth.
For $M\in \Rr$, $u_M$ is $C^1$, and it is smooth except at the corners $(\pm A, \pm B)$.

Recall that for $M\in \Cc$ or $M\in \Aa$, 
\[
   x(M):= \min \{x>0: \text{$(x,0,z)\in M$ for some $z$}\}.
\]
Likewise, we let
\begin{equation}\label{yM-definition}
   y(M):= \min \{y>0: \text{$(0,y,z)\in M$ for some $z$}\}.
\end{equation}

\begin{theorem}\label{wide-graph}
For $M\in \Rr$, 
\[
  [\hat c, A]\times [-B,B] \subset \domain(u_M)
\]
and 
\[
  \pdf{}yu_M(x,y)< 0 \quad\text{on $[\hat c, A)\times (0,B]$},
\]
where $\hat c:=x(M)+\pi$.

Likewise, if $B>y(M)+\pi$, then
\[
   [-A,A]\times [y(M)+\pi, B] \subset \domain(u_M),
\]
and
\[
   \pdf{}x u_M(x,y) <0  \quad \text{on $(0,A]\times [y(M)+\pi,B)$.}
\]
\end{theorem}

\begin{proof}
Let $\Sigma = \Mup\cap \{x\ge \hat c\}$.
By Corollary~\ref{y-graph-corollary}, 
\begin{equation*}\label{dottie}
\text{$\nu\cdot \ee_2>0$ on $\Sigma\cap\{y>0\}$.}
\end{equation*}
If $p\in M\cap\{y>0\}$ and $\nu\cdot\ee_3=0$, then $\nu\cdot\ee_2<0$ by Theorem~\ref{waist-theorem}, and thus $p\notin \Sigma$.
Hence $\nu\cdot\ee_3$ is never $0$ on $\Sigma\cap\{y>0\}$.
By Theorem~\ref{finite-y-slice-theorem}, $\nu\cdot\ee_3>0$ on $\Sigma\cap\{y=0\}$.  Thus $\nu\cdot\ee_3>0$ on $\Sigma\cap\{y>0\}$ since
$\Sigma\cap\{y>0\}$ is connected.

Note that on $\Omegaout$, 
\[
  \pdf{\yout}{z} = - \frac{\nu\cdot\ee_3}{\nu\cdot\ee_2} < 0,
\]
where $\yout(x,z)$ is as in Theorem~\ref{y-graph-theorem}.
Thus for each $\hat x\in [\hat c, A)$, $\yout(\hat x,z)$ is a strictly decreasing function of $z$. Then, we have that $z$ is a 
strictly decreasing function of $y$ on $\Sigma\cap\{x=\hat x\}$.   Notice that $\yout(\hat x,0)=B$ and $\yout(\hat x, \uup(\hat x))=0$. 
Since this holds for each $\hat x\in [\hat c, A)$, we see that
\[
   \Sigma \cap \{y\ge 0\}
\]
is the graph of a function on $[\hat c, A]\times [0,B]$. By symmetry, $\Sigma$ is the graph of a function $u$ on
$[\hat c,A]\times [-B, B]$.  
By Theorem~\ref{y-graph-theorem} and Theorem~\ref{connected-theorem},
    $(M\cap\{x\ge \hat c\})\setminus \Sigma$ lies below the graph of $u$. Thus $u=u_M$.

Furthermore, 
\[
  \pdf{u_M}{y} = - \frac{\nu\cdot\ee_2}{\nu\cdot\ee_3} < 0  \quad\text{on $[\hat c, A)\times (0,B)$}.
\]
The last assertion of the theorem follows by reversing the roles of $x$ and $y$.
\end{proof}

\section{The Inner and Outer Portions of an Annuloid}\label{inner-outer-section}

In this section, we describe a way of dividing $M\in \Rr$ (or $M\in\Aa$) 
into an inner and an outer portion, and we show that the outer portion is well-behaved.
In particular, we show that $\nu\cdot\ee_3$ is everywhere positive on the outer portion, and we
get good control on the slopes of the tangent planes in the outer portion.  
This control is the key to showing that 
 if $M\in \Aa(b,B,\hat x)$, then $B(M)=B$.
See Theorem~\ref{planes-theorem}.

The idea is the following.  For points in the outer portion that are not in the plane $\{y=0\}$, 
the unit normal should point away from that plane, whereas on the inner portion, the unit normal should point
toward the plane.  Thus, except on the plane $\{y=0\}$, the outer portion should be the points
where $(\nu\cdot \ee_2)/y>0$, and the inner portion should be the points where $(\nu\cdot\ee_2)/y<0$.
What about on $M\cap\{y=0\}$?  The function $(\nu\cdot\ee_2)/y$ extends to a smooth function $f_M$
on all of $M$, so we extend the notion of outer and inner to all of $M$ by defining a point to be outer
or inner according to whether $f_M$ is $>0$ or $<0$.

\begin{definition} \label{definition-myin-myout}
Suppose that $M$ is an oriented surface in $\RR^3$ such that reflection in the plane $\{y=0\}$
is an orientation-reversing
isometry of $M$. We let
$f_M: M\to \RR$ be the smooth function such that 
\[
    f_M = \frac{\nu\cdot \ee_2}y     \quad\text{on $M\cap\{y\ne 0\}$.}
\]
We let
\begin{align*}
  &\Myout: = \{p\in M: f_M(p)>0\},  \\
  &\Myin:= \{p\in M: f_M(p)<0\}.
\end{align*}
\end{definition}

(The function $f_M$ exists because for any smooth function $u$ on $M$ with $u(x,y,z)\equiv -u(x,-y,z)$,
the function $u(x,y,z)/y$ on $M\cap\{y\ne 0\}$ extends smoothly to $M$.)

Note that  
\begin{equation}\label{crest}
\text{$f_M(p) = \nabla_{\ee_2}(\nu\cdot \ee_2) = (\nabla_{\ee_2}\nu)\cdot \ee_2$ for $p\in M\cap\{y=0\}$}.
\end{equation}
If $p\in M\cap\{y=0\}$, then $\ee_2$ is one of the principal directions of $M$ at $p$.
Thus, by~\eqref{crest},
\begin{equation}\label{crest2}
\text{If $p\in M\cap\{y=0\}$ and if the Gauss curvature is $\ne 0$, then $f_M(p)\ne 0$.}
\end{equation}

For $M\in \Rr$, note that $f_M$ has the following behavior on $\partial M$:
\begin{equation}\label{boundary-segments}
\begin{aligned}
f_M > 0  &\quad\text{on $(-A,A)\times\{\pm B\}\times\{0\}$}, \\
f_M< 0   &\quad\text{on $(-a,a)\times\{\pm b\}\times\{0\}$},\\
f_M=0   &\quad\text{ on $\{\pm a\}\times [-b,b]\times\{0\}$}, \\
f_M=0  &\quad\text{ on $\{\pm A\}\times [-B,B]\times\{0\}$}.
\end{aligned}
\end{equation}

The following is trivial, but useful:

\begin{proposition}\label{trivial-proposition}
Suppose that $M$ is in $\Rr$ or in $\Aa$.
Suppose $L$ is a line parallel to the $y$-axis.  Let $p$ be the point in $L\cap M$
that maximizes $y(p)$.  Then $\nu(M,p)\cdot \ee_2\ge 0$.

In particular, if $M$ and $L$ cross transversely at $p$, then $p$ is in $\Myout$.
\end{proposition}

\begin{proof}
It suffices to prove it for $M\in \Rr$; the other case follows by taking limits.
Let $K$ be the closure of the bounded component of $\{z\ge 0\}\setminus M$.
By definition, $\nu(M,\cdot)$ is the unit normal vector field to  $M$ that points out of $K$.
Proposition~\ref{trivial-proposition} follows immediately.
\end{proof}

\begin{lemma}\label{wingnuts}
If $M\in \Rr$, then
\begin{enumerate}
\item\label{wingnuts-item-1} the surface $\Mup\cap\{x \in (x(M)+\pi, A)\}$  is contained in $\Myout$, and
\item\label{wingnuts-item-2} the surface $\Mlow\cap\{x \in (x(M)+\pi, a)\}$ is contained in $\Myin$.
\end{enumerate}
\end{lemma}

\begin{proof}
Let $M'$ be the surface in Assertion~\eqref{wingnuts-item-1}.
By Theorem~\ref{y-graph-theorem}, $\ee_2\cdot \nu$ is $>0$ on all the points on $M'\cap\{y>0\}$.
Thus those points are in $\Myout$. By symmetry, $M'\cap\{y< 0\}$ is also contained in $\Myout$.
Now $\nu\cdot\ee_2$ is a Jacobi field that is $>0$ in $M'\cap\{y>0\}$ and that vanishes along
$M'\cap\{y=0\}$.  Thus, by the boundary maximum principle~\cite{hopf}, 
\[
\text{$\nabla_{\ee_2}(\nu\cdot\ee_2)>0$ on $M'\cap\{y=0\}$.} 
\]
Thus $M'\cap\{y=0\}$ is contained in $\Myout$ (by~\eqref{crest}).

The proof of Assertion~\eqref{wingnuts-item-2} is essentially the same.
\end{proof}

\begin{definition}\label{R-lambda-definition}
Given $\lambda\in (0,\infty)$, let $\Rr(\lambda)$ be the collection of $M\in \Rr$ such that
\begin{enumerate}
\item $M$ lies in the slab $\{|y|\le \lambda\}$.
\item $a(M)\ge x(M)+2\pi$.
\end{enumerate}
\end{definition}

The following theorem shows that the tangent planes to $M\in \Rr(\lambda)$ at points where $f_M=0$ have uniformly
bounded slopes.

\begin{theorem}\label{slopey-theorem}
There is an $\eta=\eta(\lambda)<\infty$ with the following property.
If $p\in M\in \Rr(\lambda)$ and if $f_M(p)=0$, then
\[
   |\nu(M,p)\cdot \ee_3|^{-1} \le \eta.
\]
\end{theorem}

\begin{proof}
First we give a slope bound on the points of the boundary where $f_M=0$, i.e., 
on the edges $\{\pm A\}\times [-B,B]\times\{0\}$ and $\{\pm a\}\times [-b,b]\times\{0\}$.
Let $u: [0,2\pi)\times [-\lambda,\lambda]\to [0,\infty)$ be the translator that is $\infty$ on the edge
\[
    \{2\pi\}\times [-\lambda,\lambda]
\]
and that is $0$ on the other three sides of the boundary.
(For existence of the function $u$, see Theorem~\ref{short-barrier-theorem}.)

Let $s=\max_{y\in [-\lambda, \lambda]} (\partial/\partial x)u(0,y)$.
(By the boundary maximum principle, $s<\infty$.)
By the maximum principle,
\[
    \Mup\cap\{ x\in [-A, -A +\pi)\}
\]
lies under the graph of $(x,y)\mapsto u(x+A,y)$.
Thus the slope of the tangent plane to $M$ at each point of of $\{-A\}\times [-B,B]\times \{0\}$
is bounded above by $s$.  By symmetry, the same bound holds on $\{A\}\times [-B,B]\times \{0\}$.
The same argument gives the same upper bound on the edges $\{\pm a\}\times [-b,b]\times \{0\}$.

Thus we have shown
\begin{equation}\label{boundary-control}
   |\nu\cdot \ee_3|^{-1} \le (1+s^2)^{1/2}  \quad\text{on $(\partial M)\cap \{f_M=0\}$.}
\end{equation}

Now suppose that Theorem~\ref{slopey-theorem} is false.  Then there exist
 $M_n\in \Rr(\lambda)$ and $p_n\in M_n$  such that
\[
   f_{M_n}(p_n)=0
\]
and such that 
\[
   |\nu(M_n,p_n)\cdot \ee_3| \to 0.
\]
Since $\nu(M_n,p_n)\cdot\ee_2=0$, it follows that (after passing to a subsequence) 
\begin{equation}\label{limit-normal}
\text{$\nu(M_n,p_n)\to \sigma\ee_1$, where $\sigma$ is $1$ or $-1$.}
\end{equation}
By~\eqref{boundary-control}, we may suppose that $p_n$ is in the interior of $M_n$ for all $n$.  
By symmetry, we can assume that $x(p_n)\ge 0$ for all $n$.  By Lemma~\ref{wingnuts}, we see that 
$x(p_n)\le x(M_n) +\pi$, and thus
\[
    0 \le  x(p_n) \le a_n -\pi
\]
for all $n$.

By passing to a subsequence, we can assume that
\begin{enumerate}[{Case} 1:]
\item $\dist(p_n,\partial M_n)$ is bounded away from $0$ and $\dist(p_n,Z)$ is bounded away from $0$,  or
\item $\dist(p_n, \partial M_n)$ is bounded away from $0$ and $\delta_n:=\dist(p_n,Z)\to 0$, or
\item $\dist(p_n, \partial M_n)\to 0$ (and therefore $z(p_n)\to 0$.)
\end{enumerate}

In case 1, we let
\[
   M_n' = M_n - (x(p_n),0, z(p_n))
\]
and
\[
   p_n' = p_n -  (x(p_n),0,z(p_n))= (0,y(p_n),0).
\]
By passing to a subsequence, we can assume (by the fundamental compactness theorem for minimal surfaces of locally bounded area and genus~\cite{white18}) that the
  $M_n'$ converge to a smooth embedded limit $M'$ (possibly with multiplicity),
and  that $p_n'$ converges to a limit $p'=(0,y(p'),0)$.  By the curvature estimates in Theorem~\ref{curvature-bound-theorem},
 the convergence is smooth in a neighborhood of $p'$.
 Thus, 
 \begin{equation}\label{lucas}
   f_{M'}(p') = \lim f_{M_n'}(p_n') = \lim f_{M_n}(p_n) = 0
\end{equation}
and $\nu(M',p')=  \ee_1$ by~\eqref{limit-normal}.
The component of $M'$ containing $p'$ cannot be contained in a plane, since if were, that component
would be the plane $\{x=0\}$, which is impossible since $M'$ is contained in the slab $\{|y|\le \lambda\}$.
By Corollary~\ref{x-critical-corollary} (with $\vv=  \ee_1$), $y(q_n)=0$ and the Gauss curvature of $M_n$ at $q_n$ is nonzero.
Thus 
\[
 f_{M'}(p')\ne 0
\]
by~\eqref{crest2}, contrary to~\eqref{lucas}.
The contradiction proves the Theorem in Case 1.

In Case 2, let $q_n$ be the point in $Z$ closest to $p_n$.  
Let $\tilde M_n = (M_n-q_n)/\delta_n$ and $\tilde p_n:=(p_n-q_n)/\delta_n$.
By~\eqref{limit-normal} and Lemma~\ref{cat-lemma} (at the end of this section), $\tilde M_n$ converges smoothly
to the standard catenoid $\tilde M$ whose waist is the unit circle in the the plane $\{z=0\}$
and $\tilde p_n$ converges to the point $\tilde p= (1,0,0)$.
Now $f_{M_n}(p_n)=0$, so $f_{\tilde M_n}(\tilde p_n)=0$ and therefore, letting $n\to\infty$,
\[
  f_{\tilde M}(\tilde p)=0.
\]
But the Gauss curvature of $\tilde M$ at $\tilde p$ is nonzero, so $f_{\tilde M}(\tilde p)\ne 0$ by~\eqref{crest2}.
a contradiction.  Thus Case 2 cannot occur.

Now we turn to Case 3: $\dist(p_n,\partial M_n)\to 0$.
The boundary of $\partial M_n$ lies outside
\[
     (-a_n,a_n)\times (-b_n,b_n)\times \RR,
\]
so if $\rho_n$ is the minimum distance
 from one of the  points $(\pm x(M_n),0,z(M_n))$ to $\partial M_n$,
then
\begin{align*}
\rho_n &\ge \min\{a_n - x(M_n), b_n\}  \\
&\ge \min\{2\pi, \pi/2\} \\
&= \pi/2.
\end{align*}
Thus if $0<r<\pi/2$, then (for all sufficiently large $n$),
\begin{equation}\label{boundary-ball}
  \text{$\BB(p_n,r)$ contains no critical points of $x(\cdot)|M_n$}.
\end{equation}
Now let $M_n' = (M_n-p_n)/z(p_n)$. By passing to a subsequence, we can assume that
$M_n'$ converges smoothly in $\{z>-1\}$ to a surface $M'$
 that is minimal with respect to the Euclidean metric.
  Note that $M'$ is contained in $\{z\ge -1\}$
and that the boundary of $M'$ is (if nonempty) a set of lines in $\{z= -1\}$ parallel the $x$-axis.
Hence the function $x(\cdot)$ (i.e.,  the function $F_{\ee_1}$) is not constant on any component of $M$.

Thus by Theorem~\ref{semicontinuity-theorem},
 the critical point $p'$ of $x(\cdot)|M'$ is a limit of critical points of $x(\cdot)|M_n'$.
But that is impossible by~\eqref{boundary-ball}.
\end{proof}

\begin{theorem}\label{R-lambda-theorem}
Let $M\in \Rr(\lambda)$.  
Then:
\begin{enumerate}
\item\label{R-lambda-item-1} The tangent plane is never vertical at any point of $\overline{\Myout}$.
\item\label{R-lambda-connected-item} $\Myout$ is connected.
\item\label{R-lambda-uppity-item} $\nu\cdot \ee_3>0$ at all points of $\overline{\Myout}$.
\item\label{R-lambda-slope-item} At all points $(x,y,z)\in \overline{\Myout}$,
\[
  0\le   (B-|y|) \le \eta B(\nu\cdot\ee_3),
\]
where $\eta=\eta(\lambda)$ is given by Theorem~\ref{slopey-theorem}.
\item\label{R-lambda-graph-item} $\Myout \cap\{x\ge x(M)+\pi\}$ is the graph of a function
\[
    u: [x(M)+\pi,A(M)]\times [-B,B]
\]
such that
\begin{equation*}
   (B-|y|)\sqrt{1+|Du|^2} \le \eta B,
\end{equation*} 
\end{enumerate}
\end{theorem}

The statement of Theorem~\ref{R-lambda-theorem}
is somewhat involved.  However, for the remainder of the paper, we only need two consequences of the theorem:
Assertion~\eqref{R-lambda-graph-item} of the theorem, and Theorem~\ref{planes-theorem} below.

\begin{proof}
By Theorem~\ref{waist-theorem}, $\ee_2\cdot\nu$ and $y(\cdot)$ have opposite signs on $\waist(M)\cap\{y\neq 0\}$. Hence $f_M(p)<0$  on 
 $\waist(M)\cap\{y\neq 0\}$. By Theorem~\ref{slopey-theorem}, $f_M(p)$ is never equal to $0$ on $\waist(M)$. 
 Thus $f_M<0$ at all points of $\waist(M)$, so $\waist(M) \subset \Myin$. This is Assertion~\eqref{R-lambda-item-1}.

\begin{claim}\label{component-claim}
Suppose $W$ is a connected component of $M\cap\{f_M\ne 0\}$ such that $\nu\cdot \ee_2=0$
at all points of $\partial W$.  Then there must be points in $W$ with $\nu\cdot\ee_3=0$.
\end{claim}

\begin{proof}[Proof of claim]
Note that $\nu\cdot\ee_2$ is a Jacobi field on $\overline{W}$ that vanishes at the boundary.
Thus $\overline{W}$ is not strictly stable.   Now $\nu\cdot\ee_3$ is a Jacobi field that never
vanishes on $\partial W$ (by~Theorem~\ref{slopey-theorem}).  
Thus if it never vanished on $W$, then it would never vanish on $\overline{W}$,
and thus $\overline{W}$ would be strictly stable, a contradiction.  Thus we have proved the claim.
\end{proof}

Recall (see~\eqref{boundary-segments}) that 
\begin{equation}\label{boundary-portion}
 (\partial M)\cap\{f_M>0\} = \Ee^+\cup \Ee^-,
\end{equation}
where 
\[
\Ee^\pm =  (-A,A)\times \{B^\pm\}\times \{0\}.
\]

Let $W$ be the connected component of $\Myout$ that contains
\begin{equation}\label{porky}
\Mup\cap \{x\in (x(M)+\pi, A)\}.  \tag{*}
\end{equation}
(The set~\eqref{porky} is in $\Myout$ by Lemma~\ref{wingnuts}.)
Thus $W$ contains both of the edges $\Ee^\pm$, since it contains the
portions of those edges with $x>x(M)+\pi$.  
If $W'$ were another connected component of $\Myout$, it would not contain any of the edges $\Ee^\pm$, and
thus $\nu\cdot\ee_2$ would vanish everywhere on $\partial W'$ (by~\eqref{boundary-portion}).
Thus by Claim~\ref{component-claim}, $\nu\cdot\ee_3$ would vanish at some points of $W'$.
But by Assertion~\eqref{R-lambda-item-1}, $\nu\cdot\ee_3$ does not vanish at any point of $\overline{W'}$.
The contradiction proves that  $\Myout$ is connected, which is Assertion~\eqref{R-lambda-connected-item}.

Now $\Myout$ is connected, $\nu\cdot \ee_3$ never vanishes on $\overline{\Myout}$,
and $\nu\cdot\ee_3>0$ on $\Ee^+\subset \Myout$.
Thus $\nu\cdot \ee_3>0$ at all points of $\overline{\Myout}$, which is Assertion~\eqref{R-lambda-uppity-item}.

By Spruck-Xiao \cite{spruck-xiao}*{Lemma 5.7}
 (see also \cite{graphs}*{Theorem~2.6}), the maximum of $(\nu\cdot\ee_3)^{-1}(B - |y|)$ on $\overline{\Myout}$
occurs at a point in the boundary, and thus a point where $f_M=0$.  Thus Assertion~\eqref{R-lambda-slope-item} follows
from Theorem~\ref{slopey-theorem}.

Assertion~\eqref{R-lambda-graph-item} follows immediately from Assertion~\eqref{R-lambda-slope-item}
and Theorem~\ref{wide-graph}.
\end{proof}

\begin{definition}\label{LL-definition}
Let $\Ll(b,B)$ be the set of all limits of sequences $M_n'=M_n - q_n$
where $M_n\in \Rr$ and $q_n=(x_n,0,z_n)\in \{y=0\}$, and where
\begin{align*}
&a(M_n)-x(M_n)\to\infty, \\
&\text{$b_n:=b_n(M)$ converges to $b$}, \\
&\text{$B_n:=B_n(M)$ converges to $B$},  \, \text{and}\\
&\dist(q_n,\partial M_n)\to \infty.
\end{align*}
\end{definition}

Note that $\Aa(b,B,\hat x)\subset \Ll(b,B)$ for every $\hat x$, 
and that any limit of surfaces in $\Ll(b,B)$ is also in $\Ll(b,B)$.

\begin{theorem}\label{planes-theorem}
Let $\lambda>B$.
If $M\in \Ll(b,B)$ and if $p\in M^\textnormal{outer}$, then 
\begin{equation}\label{Ll-inequality}
   (B - |y(p)|) \le \eta(\lambda) B (\nu(M,p)\cdot \ee_3).
\end{equation}
Now suppose $M\in \Ll(b,B)$ is a non-empty union of planes.
Then $M^\textnormal{outer}$ contains the planes $\{y=\pm B\}$, possibly with multiplicity,
and the planes of $M^\textnormal{inner}$ are contained in the slab $\{|y|\le b\}$.
If $B>b$, then the plane $\{y=B\}$ occurs with multiplicity~$1$.
\end{theorem}

\begin{proof}
Let $\lambda>B$.
Note that $M_n\in \Rr(\lambda)$ for all sufficiently large $n$.
We may assume that $M_n\in \Rr(\lambda)$ for all $n$.
Thus for $p\in (M_n')^\textnormal{outer}$, we have the estimate (Theorem \ref{slopey-theorem})
\[
    (B_n - |y(p)|) \le \eta B_n (\nu(M_n',p)\cdot \ee_3).
\]
Passing to the limit gives~\eqref{Ll-inequality}.

Now suppose that $M$ is a nonempty union of planes.
For large $n$, the $y$-axis, $Y$, intersects $M_n'$ transversely in $k>0$ points, where $k$ is the number of planes.
By Proposition~\ref{trivial-proposition}, the point in $Y\cap M_n'$ for which $y$ is greatest is in $(M_n')^\textnormal{outer}$. 
Thus (passing to the limit) $M^\textnormal{outer}$ is nonempty.

Since $M$ consists of vertical planes, $\nu\cdot\ee_3\equiv 0$ on $M$, and thus $|y|\equiv B$ on $M^\textnormal{outer}$
by~\eqref{Ll-inequality}  Hence $M^\textnormal{outer}$ consists of the planes $\{y=\pm B\}$.

If $B=b$, then trivially $M^\textnormal{inner}$ is contained in $\{|y|\le b\}$.
Thus suppose that $B>b$.  Then, by Theorem~\ref{alexandrov-y-theorem} (applied to the $M_n'$), the planes $\{y=B\}$ and $\{y=-B\}$ each occur with multiplicity $1$ in $M$, and the remainder of $M$ lies in the slab $\{|y|\le b\}$.
\end{proof}

\begin{corollary}\label{planes-corollary}
Suppose that $M\in \Ll(b,B)$, that $\Sigma$ is a $\Delta$-wing or grim reaper surface definied
on the strip $\RR\times (-\beta,\beta)$, and that $\Sigma$ is contained in $M^\textnormal{inner}$.
Then $\beta\le b$.
\end{corollary}

\begin{proof}
Let $M'$ be a subsequential limit of $M+(0,0,z)$ as $z\to\infty$.  By Theorem~\ref{translates-theorem}, such a limit exists
and is a union of planes.   Note that $\Sigma+(0,0,z)$ converges to the planes $\{y=\pm\beta\}$,
and thus those planes are in the inner portion of $M'$. 
  Thus $\beta\le b$ by Theorem~\ref{planes-theorem}.
\end{proof}

We conclude this section with the Lemma that was used in the proof of 
  Theorem~\ref{slopey-theorem}.

\begin{lemma}\label{cat-lemma}
Suppose that $M_i$ are surfaces in $\Cc$ or $\Aa$ that lie in a slab $\{|y|\le \Lambda\}$.
Suppose also that $p_i\in M_i$, that $\dist(p_i,Z)\to 0$
and that $\nu(M_i,p_i)\to \vv$, where $\vv$ is a horizontal vector not equal to $\ee_2$ or $-\ee_2$.
Let $q_i$ be the point in $Z$ closest to $p_i$.
Then $M_i':=(M_i-q_i)/\dist(p_i,Z)$ converges smoothly and with multiplicity $1$ to the catenoid whose
waist is the unit circle in the plane $\{z=0\}$.
\end{lemma}

(The lemma is also true when $\vv=\pm \ee_2$, but a different proof is required.  We only need
the case when $\vv=\pm \ee_1$.)

\begin{proof}
By the compactness theorem~\cite{white18}*{Theorem~1.1} for minimal surfaces, 
we can assume (by passing to a subsequence)
  that $M_i-q_i$ converges to a limit $M$,
where $M\setminus \partial M$ is smooth and embedded (possibly with multiplicity),
and where the convergence is smooth away from the boundary except at a locally finite subset $S$ of $M\setminus \partial M$.
(Indeed, by Remark~\ref{blowup-remark}, $S$ contains at most one point.)

The symmetries imply that $\Tan(M,0)$ is one of the coordinate planes.
We claim that $\Tan(M,0)$ cannot be $\{x=0\}$.  For suppose it were.  Since $M$ is embedded
and invariant under $(x,y,z)\mapsto (-x,y,z)$, if follows that $M$ and $\{x=0\}$ coincide in a neighborhood of $0$.
But then by unique continuation, all of $\{x=0\}$ would lie in $M$, which is impossible since $M$ is contained
in the slab $\{|y|\le \Lambda\}$.

Thus $\Tan(M,0)$ is either $\{y=0\}$ or $\{z=0\}$.  Consequently, since $\vv\notin\{\pm\ee_2,\pm\ee_3\}$,
$\vv$ is not perpendicular to $\Tan(M,0)$.

By passing to a subsequence, we can assume that $M_i'$ converges to a limit $M'$
and that 
\[
   p_i' = (p_i-q_i)/\dist(p_i,Z) = (p_i-q_i)/|p_i-q_i|
\]
 converges to a point $p'$.  
The convergence is smooth at $p'$ (and at $-p'$), so
\begin{equation}\label{both}
  \nu(M',p')=\vv, \quad \nu(M',-p')= -\vv.
\end{equation}
By \cite{white18}*{Theorems~2.2 and~2.3}, $M'$ is one of the following:
\begin{enumerate}
\item A multiplicity $1$ plane.  
\item Two or more planes  (counting multiplicity) parallel to $\Tan(M,0)$.
\item A complete, embedded, non-flat minimal surface of finite total curvature whose ends
  are parallel to $\Tan(M,0)$. In this case, the convergence is smooth and with multiplicity $1$.
\end{enumerate}
If $M'$ consists of planes, then (by~\eqref{both}) there are at least two planes, counting multiplicity,
and thus those planes are parallel to $\Tan(M,0)$.
In that case, by~\eqref{both}, $\vv$ would be perpendicular to $\Tan(M,0)$. But we have already
seen that $\vv$ is not perpendicular to $\Tan(M,0)$.
Thus $M$ is a complete, nonflat, embedded, minimal surface of finite total curvature.
Since $M$ has genus $0$, it is a catenoid with ends parallel to $\Tan(M,0)$, and therefore parallel
to $\{z=0\}$ or to $\{y=0\}$.  If the ends were parallel to $\{y=0\}$, then (by symmetry) $M'$ would intersect
$Z$ orthogonally in a pair of points.  But that is impossible since the $M_i'$ are disjoint from $Z$.

Hence the ends of $M'$ are parallel to $\{z=0\}$ (and therefore $\Tan(M,0)$ is horizontal.)
\end{proof}

\begin{remark}\label{cat-lemma-remark}
In the proof of Lemma~\ref{cat-lemma}, we showed that the $M_i$ converged (after passing to a subsequence)
to a smooth limit $M$ with $\Tan(M,0)$ parallel to the ends of $M'$, i.e., with $\Tan(M,0)$ horizontal.
Of course, the convergence of $M_i$ to $M$ is not smooth at $0$, and thus the multiplicity of the component of $M$
containing $0$ is $\ge 2$.  In fact, since every vertical line in the plane $\{y=0\}$ intersects $M_i$ at most twice 
 (by~Theorems~\ref{y-slice-theorem} and~\ref{finite-y-slice-theorem}), that multiplicity is exactly $2$.
\end{remark}

\section{Behavior as $z\to -\infty$}\label{down-behavior-section}

Suppose $M\in \Aa(b,B,\hat x)$.  
Recall from \S\ref{sideways-section}
  that $\Omegain(M)$ consists of the region of the $(x,z)$-plane bounded above 
 by the curve $z=\ulow(x)$, and on the left by the vertical line $x=\hat c:=x(M)+\pi$.
In this section we show (Theorem~\ref{fundamental-theorem}) that if $(x,z)\in \Omegain(M)$ and if $(x,z)$ is far from the top edge of $\Omegain(M)$, then $\yin(x,z)$
and $\yout(x,z)$ are very near $b$ and $B$, respectively.
It follows that $b(M)=b$ and $B(M)=B$. 

Theorem~\ref{fundamental-theorem} is a direct consequence of the analogous result (Lemma~\ref{fundamental-lemma}) for $M\in \Rr$.

\begin{definition}\label{reaper-definition}
For $b\ge \pi/2$, let
\[
  w_b: \RR\times (-b,b) \to \RR
\]
be the grim reaper surface with $w_b(0,0)=0$ and
\[
             \pdf{w_b}x \equiv -s(b).  
\]
Let 
\[
  y^b: \{(x,z): z \le -s(b) x\} \longrightarrow [0,b)
\]
be the function such that $(x,y,z)\in \graph(w_b)$ if and only if $y=\pm y^b(x,z)$.
\end{definition}

\begin{lemma}\label{reaper-lemma}
Suppose that $M_n\in \Rr$, that $b_n:=b(M_n)$ converges to $b$, and
that $a_n:=a(M_n)$ tends to infinity. Let $p_n=(x(M_n),0,z(M_n))$ be the point
where $\nu(M_n,p_n)= -\ee_1$.   
Let
    $x_n\in [x(M_n) ,A_n]$ and
      $q_n:=(x_n,0,\unlow(x_n))$
 be such that
\begin{align*}
\dist(q_n, \{p_n\} \cup \partial M_n) \to \infty.
\end{align*}
Then 
\[
    \Sigma_n:= \Mlow_n - q_n
\]
 converges smoothly to the grim reaper surface $\graph(w_b)$.
\end{lemma}
    
\begin{proof}
Let $M_n'=M_n - q_n$.  
Theorem~\ref{better-compactness-theorem}  gives smooth subsequential convergence of $M_n'$ to a limit $M'$.
(Since $x_n\ge x(M_n)$, the distance from $q_n$ to $(-x(M_n),0,z(M_n))$ is greater than the distance from $q_n$
to $p_n$.)
Let $\Sigma$ be the component of $M'$ containing $0$.  
By Theorem~\ref{better-compactness-theorem}, 
\[
x_n - x(M_n) \to \infty
\]
and $\Sigma$ is a $\Delta$-wing or grim reaper surface, and thus the graph
of a function
\[
w: \RR\times (-\beta,\beta)\to \RR
\]
with $w(0,0)=0$, for some $\beta\ge \pi/2$.
Note that $\Sigma=\graph(w)$ is the limit of $\Mlow_n - q_n$.

Let $x\in\RR$.
By Theorem~\ref{slopes-theorem}, 
\begin{equation}\label{reaper-inequality}
    (\unlow)'(x_n + x) \le \pdf{}x f_{b_n}(x_n+x,0).
\end{equation}
As $n\to\infty$, the left side of~\eqref{reaper-inequality} converges to $\pdf{}x w(t,0)$. 
Since $x_n-x(M_n)\to\infty$, $x_n\to \infty$.  Thus the right side of~\eqref{reaper-inequality}
converges to $-s(b)$.  Hence
\[
   \pdf{}x w(x,0) \le -s(b)
\]
for all $x$.  Consequently $\Sigma=\graph(w_\beta)$ for some $\beta\ge b$.
On the other hand, $\beta\le b$ by Corollary~\ref{planes-corollary}.
\end{proof}

Recall that for $M\in \Rr$, 
\[
   \Omegain(M):=\{(x,z) \in [\hat c,a) \times \RR: 0\le z \le \ulow(x)  \},
\]
where $\hat c = x(M)+\pi$ and $a=a(M)$.
The
 following lemma says that for $M\in \Rr$, if $(x,z)\in \Omegain$ is far from the upper
and lower edges of $\Omegain$, then $\yin(y,z)$ is close to $b(M)$ and $\yout(y,z)$ is close
to $B(M)$.    The analogous assertion for $M\in \Aa(b,B,\hat x)$ (which is the main result of
this section) follows readily.

\begin{lemma}\label{fundamental-lemma}
For every pair of positive numbers $\eps$ and $\lambda$, there is an $R$ with the following property.
Suppose that $M\in \Rr(\lambda)$ with $b=b(M)$.
Then 
\[
   \text{$\yin(x,z)\in [b-\eps,b+\eps]$ and $\yout(x,z)\in [B-\eps, B]$}
\]
for all $(x,z)\in \Omegain(M)$ such that 
\[
   z\in [ R, \ulow(x) - R],
\]
\end{lemma}

\begin{proof}
Suppose it is false.  Then there exist $M_n\in \Rr(\lambda)$
and $(x_n,z_n)\in \Omegain(M_n)$ such that
\begin{equation}\label{zns-trapped}
n \le z_n \le \unlow(x_n) - n
\end{equation}
and such that
\begin{equation}\label{mishit}
\text{$\yin(x_n,z_n)\notin [b_n-\eps,b_n+\eps]$ or $\yout(x_n,z_n)\notin [B_n-\eps,B_n]$}.
\end{equation}
By passing to subsequence, we can assume that $b_n$ and $B_n$ converge:
\begin{equation}\label{bees-converge}
\begin{aligned}
&b_n \to b, \\
&B_n \to B.
\end{aligned}
\end{equation}

We will prove Lemma~\ref{fundamental-lemma} by showing that~\eqref{zns-trapped}
and~\eqref{bees-converge} imply (contrary to~\eqref{mishit}) that 
\begin{align*}
\lim_n \ynin(x_n,z_n) &= b, \,\text{and} \\
\lim_n\ynout(x_n,z_n) &= B.
\end{align*}

By Theorem~\ref{better-compactness-theorem}, we can assume (passing to a subsequence)
 that $M_n':= M_n - (x_n,0,z_n)$ converges smoothly to a limit $M'$, each component of which is one of the following:
 a plane parallel to $\{y=0\}$,
  or a $\Delta$-wing, or a grim reaper surface.

Note that if $x>0$ and if $Y + (x,0,z)$ intersects $M'$ transversely, then it does so in exactly four points (counting multiplicity).

Thus $M'$ is the union of $4$ planes (counting multiplicity). 
 By~Theorem~\ref{planes-theorem}, the planes are $\{y=\pm B\}$ and $\{y=b^*\}$
for some $b^*\in [0,b]$.

Thus $\lim_n \ynout(x_n,z_n)=B$ and $\lim \ynin(x_n,z_n)= b^*$, 
so it remains only to show that $b^*\ge b$.

Fix (for the moment) a $k\in (0,\infty)$.
For $n>k$, using \eqref{zns-trapped} and the fact that $\unlow$ is decreasing (Theorem \ref{slopes-theorem}), 
there is a $x_n'> x_n$ such that 
\[
   \unlow(x_n') = z_n + k.
\]
Let
\[
  q_n = (x_n', 0, \unlow(x_n')) = (x_n',0,z_n+k).
\]
Now
\[
\dist(q_n,\partial M_n) \ge z(q_n) = z_n + k \ge z_n\to\infty.
\]
Also,
\begin{align*}
|p_n - q_n|
&\ge 
z(p_n) - z(q_n) \\
&= \unlow(x(M_n)) - (z_n + k) \\
&\ge  \unlow(x_n) - z_n - k \\
&\ge n-k \to\infty.
\end{align*}
By Lemma~\ref{reaper-lemma},  $\Mlow_n-q_n$ converges smoothly
to the grim reaper surface $\graph(w_b)$.
Thus for every $t>0$,
\[
 \lim_n  \ynin(x_n',0, \unlow(x_n') - t) = y_b(0,-t).
\]
In particular, for $t=k$, $\unlow(x_n')-k = z_n$, so
\[
  \lim_n \ynin(x_n',z_n)  = y^b(0,-t).
\]
Since $\ynin(x_n,z_n) \ge \ynin(x_n',z_n)$ (by Theorem~\ref{x-alexandrov-theorem}),
\[
 b^* = \lim_n \ynin(x_n,z_n) \ge y^b(0,-t).
\]
This holds for every $t\in (0,\infty)$. Thus letting $t\to\infty$ gives
\[
 b^*\ge b.
\]
\end{proof}

\begin{theorem}\label{fundamental-theorem}
For every pair of positive numbers $\eps$ and $\lambda$, there is an $R$ with the following property.
Suppose that $M\in \Aa(b,B,\hat x)$.
If $B < \lambda$,
then
\[
   \text{$\yin(x,z)\in [b-\eps,b+\eps]$ and $\yout(x,z)\in [B-\eps, B]$}
\]
for all $(x,z)\in \Omegain(M)$ such that 
\[
    z < \ulow(x) - R
\]
\end{theorem}

\begin{proof}
Let $M_n\in \Rr$ be such that: $b_n:=b(M_n)$,  $B_n:=B(M_n)$, $x(M_n)$, and $a(M_n)$  converge to $b$, $B$, $\hat x$,
and $\infty$, 
 and such that $\tilde M_n=M_n - (0,0,z(M_n))$ converges
smoothly to $M$.   We may assume that $B_n<\lambda$ and that $a_n > x(M_n)+ 2\pi$ for
all $n$.

By the smooth convergence,
\begin{enumerate}[\upshape(i)]
\item  
        $\unlow(\cdot) - z(M_n)$ converges uniformly
         on compact subsets of $(x(M),\infty)$ to $\ulow(\cdot)$, and
\item\label{y-convergence-item}
    $\ynin(x, z + z(M_n))$ and $\ynout(x, z+ z(M_n))$ converge uniformly 
     on compact subsets of $\{x>x(M), \, z< \ulow(x)\}$ to $\yin(x,z)$ and $\yout(x,z)$.
\end{enumerate}

Let $R=R(\eps,\lambda)$ be as in Lemma~\ref{fundamental-lemma}.
Suppose first that $x> x(M)+\pi$ and that $z< \ulow(x)-R$.

Then for all sufficiently large $n$,
\[
     z  < \unlow(x) - z(M_n) - R
\]
or, equivalently,
\[   
    z + z(M_n) < \unlow -  R.
\]
Also, since $z(M_n)\to\infty$, 
\[
   R < z+ z(M_n)
\]
for all sufficiently large $n$.  Thus, for such $n$,
\[
   \text{$\ynin(x,z+z(M_n))\in [b-\eps,b+\eps]$ and $\ynout(x,z+z(M_n))\in [B-\eps, B]$}
\]
Then, letting $n\to\infty$, we have (by~\eqref{y-convergence-item})
\begin{equation}\label{eureka}
   \text{$\yin(x,z)\in [b-\eps,b+\eps]$ and $\yout(x,z)\in [B-\eps, B]$}.
\end{equation}
We have shown that~\eqref{eureka} holds for all $(x,z)$ with $x>x(M)$ and $z< \ulow(x)-R$.
By continuity, \eqref{eureka}~also holds for $(x,z)$ with $x\ge x(M)$ and $z\le \ulow(x)-R$.
\end{proof}
   
\begin{theorem}\label{b-B-theorem}
Suppose $M\in \Aa(b,B, \hat x)$. Then $b(M)=b$ and $B(M)=B$.
\end{theorem}

\begin{proof}
Let $x\ge x(M)+\pi$.  Then
\begin{align*}
  &\lim_{z\to -\infty} \yin(x,z) = b, \, \text{and} \\
  &\lim_{z\to -\infty} \yout(x,z)=B.
\end{align*}
by Theorem~\ref{fundamental-theorem}.   Hence $b(M)=b$ and $B(M)=B$.  (See Theorem~\ref{some-properties-theorem}
if this is not clear.)
\end{proof}

\section{Behavior as $x\to \pm \infty$} \label{sec:behavior-at-infinity}

In this section we study the asymptotic behavior of the wings, $\Mlow$ and $\Mup$, of $M \in \Aa(b,B,\hat x)$, as $x \to \infty.$ In particular, we prove that if we take limit of 
$\Mlow-(x,0,\ulow(x))$, as $x \to \infty$, then this limit is the grim reaper surface of slope $-s(b)$. On the other hand, $\Mup-(x,0,\uup(x))$ converges, as $x \to \infty$, to a grim reaper surface of slope $s=\pm s(B)$. If $B>b$, then $s=s(B).$

\begin{theorem}\label{main-theorem-concluded}
Suppose that $M\in \Aa(b,B,\hat x)$.  Let $\lambda>B$.  Then
\begin{enumerate}[\upshape(1)]
\item\label{lower-slope-item} $\displaystyle \lim_{x\to\infty} (\ulow)'(x)= - s(b)$,
 and $\Mlow - (x,0,\ulow(x))$ converges to the grim reaper surface $\graph(w_b)$. (See Definition~\ref{reaper-definition}.)
\item\label{upper-graph-item}
  The surface $\Mup \cap\{x\ge x(M)+\pi\}$ is the graph of a function
  \[
      u: [x(M)+\pi,\infty)\times (-B,B)\to \RR
  \]
  that satisfies the gradient bound
  \begin{equation*}\label{the-gradient-bound}
   (B-|y|)\sqrt{1+|Du|^2} \le \eta B, 
\end{equation*}  
 for a constant  $\eta=\eta(\lambda)$.
\item\label{upper-slope-item} 
The limit 
   $\displaystyle L:=\lim_{x\to\infty} (\uup)'(x)$
exists, and $L=\pm s(B)$.  
Moreover, $\Mup - (x,0,\uup(x))$
converges smoothly as $x\to \infty$ to the grim reaper surface $\graph(w)$ such that
\begin{gather*}
  w:\RR\times (-B,B)\to \RR, \\
  w(0,0)=0, \\
  \pdf{w}x \equiv L.
\end{gather*}
In particular, 
\[
   \lim_{x\to\infty}\pdf{}x u(x,y) = L
\]
for each $y\in (-B,B)$.
\item\label{uncapped-item}  If $B>b$, then $L=s(B)$.
\end{enumerate}
\end{theorem}

\begin{proof}
Let $\beta\ge \pi/2$ be the number such that 
\begin{equation}\label{reaper-limit}
   \lim_{x\to\infty} (\ulow)'(x)= -s(\beta).
\end{equation}
(The limit exists by Theorem~\ref{y-slice-theorem}.)
By Corollary~\ref{slopes-corollary}, the limit in~\eqref{reaper-limit} is at most $-s(b)$,
and therefore
\begin{equation}\label{beta-bigger-again}
\beta\ge b.
\end{equation}
Let $x_n\to\infty$ and 
\[
  M_n':=M - (x_n,0,\ulow(x_n))
\]
By passing to a subsequence, we can assume that $M_n'$ converges
smoothly to a limit $M'$. 
  Let $\Sigma$ be the component of $M'$ containing $0$.
By~\eqref{reaper-limit},
\[
  \Sigma\cap\{y=0\} = \{(x,0,z): z = -s(\beta)x\},
\]
Thus $\Sigma$ is the grim reaper surface $w_\beta$.
(If this is not clear, see Theorem~\ref{translates-theorem}.)
Note that $\Sigma$ is in $(M')^\textnormal{inner}$.
Thus
\[
  \beta \le b.
\]
by Corollary~\ref{planes-corollary}.
This (together with~\eqref{beta-bigger-again}) completes the proof of Assertion~\eqref{lower-slope-item}.

Assertion~\eqref{upper-graph-item} follows immediately from the corresponding assertion
 in Theorem~\ref{R-lambda-theorem}
for $M\in \Rr$. 

By Theorem~\ref{y-slice-theorem},
   the limit $\displaystyle L=\lim_{x\to\infty} (\uup)'(x)$ exists and is finite.
Suppose $x_n\to\infty$.  By the gradient bound in Assertion~\eqref{upper-graph-item} and Arzela-Ascoli,
the function
\[
   u(x+x_n,y) - u(x_n,0)
\]
converges smoothly, perhaps after passing to a subsequence,
to a function
\[
    w: \RR\times (-B,B) \to \RR.
\]
Since $w(x,0)= Lx$, we see that $w$ is the grim reaper surface
with 
\[
   \pdf{}x w \equiv L,
\]
and thus that $L=\pm s(B)$.  Since the limit is independent of the choice
of subsequence, we get convergence and not just subsequential convergence.

To prove Assertion~\eqref{uncapped-item}, suppose that $b<B$.  Then $s(b)<s(B)$, so
\begin{equation}\label{sandwich}
   -s(B) < -s(b) < s(B).
\end{equation}
Now $\uup(x)> \ulow(x)$ for all $x>x(M)$, so
\[
   L:=\lim_{x\to\infty} (\uup)'(x) \ge \lim_{x\to\infty} (\ulow)'(x) = -s(b).
\]
Since $L=\pm s(B)$, we see from~\eqref{sandwich} that $L=s(B)$.
\end{proof}

We will show later (Theorem~\ref{cutoffs-theorem}) that there are examples for which $b=B$ and other examples for which $b<B$.
The following theorem gives a condition guaranteeing that $b=B$:

\begin{theorem}\label{b-B-criterion}
Suppose that $M\in \Aa$ and that $B:=B(M)>y(M)+\pi$.   Then 
$b(M)=B(M)$ and
\[
   \lim_{x\to\infty} (\uup)'(x) = - s(B).
\]
\end{theorem}

\begin{proof}
By Theorem~\ref{wide-graph}, 
\[
   \RR\times [y(M)+\pi, B) \subset \domain(u_M)
\]
and
\[
  \pdf{}x u_M(x,y) \le 0  \quad\text{on $[0,\infty)\times [y(M+\pi, B)$.}
\]

In particular, for each $y\in [y(M)+\pi, B)$,
\[
 \limsup_{x\to\infty}\pdf{}x u_M(x,y) \le 0.
\]
In particular, since $u_M(x,0)=\uup(x)$ for $x\ge x(M)$ and since $\lim_{x\to\infty}(\uup)'(x)=\pm s(B)$, 
we see that
\[
  \lim_{x\to\infty} (\uup)'(x)= - s(B).
\]
On the other hand, if $b<B$, then
\[
  \lim_{x\to\infty} \pdf{}xu(x,y)= s(B)>0
\]
by Theorem~\ref{main-theorem-concluded}.  Thus $b=B$.
\end{proof}

\section{Large necks and Prongs}\label{large-neck-section}

In this section, we analyze the behavior of $M_n\in \Aa$ as the necksize  $x(M_n)$ tends to $\infty$.
We begin by examining the behavior of $M_ n\in \Aa$ near the point $(x(M_n),0,0)$.
Specifically,  we analyze the limit (which we call a prong) of $M_n - (x(M_n),0,0)$ as $x(M_n)\to\infty$.
Then, in Theorem~\ref{infinite-neck-theorem}, we describe the behavior of $M_n$ at bounded distances from the plane $\{x=0\}$:
in particular, we show that suitable vertical translates of the $M_n$ converge to a pair of untilted grim reaper surfaces over 
strips $\RR\times (b,b+\pi)$ and $\RR\times (-(b+\pi),-b)$.

\begin{theorem}\label{prong-existence-theorem}
Suppose that $M_n\in \Aa$ are annuloids with $b_n:=b(M_n)\to b<\infty$ and
with $x(M_n)\to\infty$.
Then, perhaps after passing to a subsequence, $B_n:=B(M_n)$ converges to a limit $B$, and the surfaces
\[
   M_n - (x(M_n), 0, 0)
\]
converge smoothly to a limit $M$.
\end{theorem}

Theorem~\ref{prong-existence-theorem} follows immediately from the curvature and area bounds in
Theorem~\ref{curvature-bound-theorem}.

\begin{definition}We define a {\bf prong} to be any surface $M$ obtained as in Theorem~\ref{prong-existence-theorem}. \end{definition}

Theorem~\ref{prong-theorem} gives some basic properties of prongs, including the behavior of a prong as $z$ tends
to $\infty$ or $-\infty$, and the fact that $B=b+\pi$ in Theorem~\ref{prong-existence-theorem}.
  Theorem~\ref{prong-ends-theorem} describes the behavior as $x$ tends
 to $\infty$ or to $-\infty$.
Another interesting property of a prong, namely that it is a sideways graph $x=x(y,z)$ over
a region in the $yz$-plane, will be proved in Section~\ref{final-graphicality-section}; see Theorem~\ref{prong-graph-theorem}.

\begin{theorem}\label{prong-theorem}
Suppose that $M$ is a prong, and that $M_n$, $b$, and $B$ are as in Theorem~\ref{prong-existence-theorem}. Let $\lambda>B$, 
then
\begin{enumerate}[\upshape (1) ]
\item\label{y-slice-item} $M\cap\{y=0\}$ is the union of two graphs
\[
  \{(x,0,\uup(x)): x\ge 0\} \quad\text{and}\quad \{(x,0,\ulow(x)): x\ge 0\},
\]
where $\uup, \ulow: [0,\infty)\to \RR$ and 
\begin{align*}
   \ulow(0)&= \uup(0) = 0, \, \text{and} \\
   \ulow(x)&< \uup(x) \quad\text{for $x>0$}.
\end{align*}
We let $\Mup$ and $\Mlow$ be the components of $M\cap\{x>0\}$
containing $\{(x,0,\uup(x)):x>0\}$ and $\{(x,0,\ulow(x)): x>0\}$.
\item\label{in-out-item} There are functions 
\begin{align*}
&\yin: \Omegain\to [0,B), \\
&\yout: \Omegaout\to [0,B),
\end{align*}
where
\begin{align*}
\Omegain&:= \{(x,z): x\ge \pi, \, z\le \ulow(x)\}, \\
\Omegaout&:= \{(x,z): x\ge \pi, \, z\le \uup(x)\},
\end{align*}
such that
\begin{align*}
\Mup\cap\{x\ge \pi\}\cap\{y\ge 0\} &= \{(x,\yout(x,z), z): (x,z)\in \Omegaout\}, \\
\Mlow\cap\{x\ge \pi\}\cap \{y\ge 0\} &= \{ (x,\yin(x,z), z): (x,z) \in \Omegain\}.
\end{align*}
\item\label{up-low-item} $M$ is connected, and $M\cap \{x>0\} = \Mup\cup \Mlow$.
\item\label{decreasing-item} For each $z$, $\yin(x,z)$ is a strictly decreasing function of $x$.
\item\label{B-graph-item}  $\Mup\cap \{x\ge \pi\}$ is the graph of a function $u(x,y)$
\[
   u:  [\pi,\infty) \times (-B,B)\longrightarrow  \RR,
\]
and
\begin{equation*}\label{gradient-bound-prong}
   (B-|y|)\sqrt{1+|Du|^2} \le \eta B,
\end{equation*}
for a constant $\eta=\eta(\lambda)$.
\item\label{upper-empty-item} As $z\to\infty$, $M-(0,0,z)$ converges to the empty set.
\item\label{planes-item} As $z\to \infty$, $M+(0,0,z)$ converges smoothly to the planes $\{y=\pm b\}$ and $\{y=\pm B\}$.
\item\label{prong-lower-slope-item} $\displaystyle \lim_{x\to \infty} (\ulow)'(x) = - s(b)$, and as $x\to\infty$,
\[
    \Mlow - (x,0,\ulow(x))
\]
converges to the grim reaper surface $w_b: \RR\times (-b,b)\to\RR$ with $w(0,0)=0$ and $\partial w/\partial x \equiv -s(b)$.
\item\label{pi-prong} $B=b+\pi$.
\end{enumerate}
\end{theorem}

\begin{proof}
Let $M_n\in \Aa$ be such that $M_n-(x(M_n),0,0)$ converges smoothly to $M$.

The proof of Assertion~\eqref{y-slice-item} is almost identical to the proof of Theorem~\ref{y-slice-theorem}.

Assertion~\eqref{in-out-item} follows from Theorem~\ref{y-graph-theorem}.

The proof of Assertion~\eqref{up-low-item} is the same as the proof of Theorem~\ref{connected-theorem}.

Concerning Assertion~\eqref{decreasing-item}, the fact that $\yin(x,z)$ is a decreasing function of the $x$ follows from the corresponding
property of the $M_n$  (Theorem~\ref{x-alexandrov-theorem}). 
 Strict inequality follows, for example, by the strong maximum principle.  (Specifically, $\nu\cdot \ee_1$
is a Jacobi field that is nonnegative on $\Mlow$.  If it vanished at a point of $\Mlow$, it would vanish identically on
 $\Mlow$ by the strong maximum principle, and then on all of $M$ by unique continuation, which is impossible since 
   $\nu\cdot\ee_1=-1$ at the origin.)

Now let $\lambda> B$.  Then $\lambda> B_n$ for all sufficiently large $n$.  We may
assume that $\lambda>B_n$ for all $n$.
Assertion~\eqref{B-graph-item} follows from the corresponding property of the $M_n$
in Theorem~\ref{main-theorem-concluded}.

To prove Assertion~\eqref{upper-empty-item},
 note (by Theorem~\ref{translates-theorem}) that every sequence $z_i\to\infty$ has a subsequence $z_{i(n)}$ such that
$M- (0,0,z_{i(n)})$ converges a limit $M'$ consisting of a union of planes parallel to the plane $\{y=0\}$.
  By Assertions~\eqref{in-out-item} and~\eqref{up-low-item},
$M'$ is contained in $\{ |x|\le \pi\}$.  Thus $M'$ is the empty set.

By Theorem~\ref{fundamental-theorem}, for each $\eps$, there is an $R=R(\eps,\lambda)$
such that, for each $n$, 
\[
   \text{$\ynin(x,z) \in [b_n-\eps,b_n+\eps]$ and $\ynout(x,z) \in [B_n-\eps,B_n]$}
\]
provided $x\ge x(M_n)+\pi$ and $z\le \unlow(x) - R$.

Passing to the limit, we see that
\[
    \text{$\yin(x,z) \in [b-\eps, b+\eps]$ and $\yout(x,z)\in [B-\eps,B]$},
\]
provided $x\ge \pi$ and $y\le \ulow(x)-R$.

Thus for each $x\ge 0$,
\begin{equation}\label{b-B-prong} 
\begin{gathered}
\lim_{z\to -\infty} \yin(x,z)=b, \\
\lim_{z\to -\infty} \yout(x,z)=B.
\end{gathered}
\end{equation}

We know (by Theorem~\ref{translates-theorem}) 
 that if $\zeta_n\to\infty$, then $M + (0,0,\zeta_n)$ converges (after passing to a subsequence)
to a limit consisting of planes.  By~\eqref{b-B-prong}, those planes are $\{y=\pm b\}$
and $\{y=\pm B\}$.  Since this is independent of the sequence and choice of subsequence, in
fact $M+(0,0,\zeta)$ converges as $\zeta\to\infty$ to those planes.

The proof of Assertion~\eqref{prong-lower-slope-item} 
  is identical to the proof of the corresponding assertion for $M\in \Aa$
   in Theorem~\ref{main-theorem-concluded}.

To prove that $B=b+\pi$ (Assertion~\eqref{pi-prong}), note that $M$ is disjoint from $\{y=0\}\cap\{x<0\}$.
Let $\Sigma$ be the portion of $M\cap \{x<0\}$ in the halfspace $\{y>0\}$.

Then $\Sigma$ lies in the half-slab $\{0< y \le B\}\cap \{x<0\}$, the boundary
of $\Sigma$ lies in $\{x=0\}$, and
\[
  \sup_{\Sigma}z(\cdot) < \infty
\]
(This last inequality follows from Assertion~\eqref{upper-empty-item}.)
Also $\Sigma+(0,0,z)$ converges smoothly as $z\to\infty$ to the halfplanes 
$\{y=b\}\cap \{x\le 0\}$ and $\{y=B\}\cap\{x\le 0\}$.

Hence, by a general theorem (Theorem~\ref{chini-theorem-2}) about translators in half-slabs,
$B\ge b+\pi$.  On the other hand, $B_n < b_n +\pi$, so $B\le b+\pi$.
Thus $B=b+\pi$.
 \end{proof}
\begin{remark}[Entropy of prongs]
Taking into account \eqref{planes-item} in Theorem \ref{prong-theorem}, we can use Corollary 8.5 in \cite{GMM} to deduce that 
a prong has entropy $4$.
\end{remark}
The following theorem describes the behavior of the prong $M$ in Theorem~\ref{prong-theorem}
as $x\to -\infty$.

\begin{theorem}\label{prong-ends-theorem}
Let $M$ be a prong as in Theorem~\ref{prong-theorem}, and let
\[
 \psi(t):=\max_{M\cap\{x=t\}}z(\cdot).
\]
Then
\[
   M - (x,0,\psi(x))  
\]
converges smoothy at $x\to-\infty$ to a pair of (untilted) grim reaper surfaces,
one over $\RR\times (b,b+\pi)$ and the other over $\RR\times (-(b+\pi),-b)$.
\end{theorem}

\begin{proof}
Let $M'$ be a subsequential limit of $M-(x,0,\psi(x))$ and let $\Sigma$ be a component of $M'$.  
By Theorem~\ref{chini-theorem-2}\eqref{chini-compact}
 (applied to each of the two components of $M\cap\{x<0\}$),
\begin{equation}\label{slab-pair}
\text{$M'$ is contained in the two slabs given by $\{b\le |y|\le b+\pi\}$.}
\end{equation}
Let $\Sigma$ be a component of $M'$.  By symmetry, it suffices to consider the case when
 $\Sigma$ is contained in the slab $\{b \le y \le b+\pi\}$.
By Theorem~\ref{translates-theorem}, $\Sigma$ is a graph or a vertical plane.
By construction,
\begin{equation}\label{max-height}
  \max_{M'\cap \{x=0\}} z(\cdot) = 0.
\end{equation}
Thus $\Sigma$ is  a graph.  By~\eqref{slab-pair}, it is an untilted grim reaper surface.

We have shown that each component of $M'$ in the halfspace $\{y>0\}$ is a grim reaper surface over
the strip $\{b < y < b+\pi\}$.

Let $k$ be the number of grim reaper surfaces in $\{y>0\}$. Then $M'+(0,0,z)$ converges as $z\to\infty$
to the planes $\{y=\pm b\}$ and $\{y=\pm (b+\pi)\}$, each with multiplicity $k$.
By Theorem~\ref{planes-theorem}, the plane $\{y=b+\pi\}$ occurs with multiplicity $1$.  Hence $k=1$.

Note that~\eqref{slab-pair} and~\eqref{max-height} determine the two grim reaper surfaces.  Thus the limit $M'$ is independent
of choice of subsequence, so we get convergence and not merely
subsequential convergence.
\end{proof}

\begin{theorem}\label{infinite-neck-theorem}
Suppose that $M_n\in \Aa$, that $b_n:=b(M_n)\to b<\infty$, $B_n:=B(M_n)\to B$, and $x(M_n) \to+\infty.$
Then $y(M_n)\to b$ and $B=b+\pi$.
Furthermore, if $\zeta_n=\max_{M\cap\{x=0\}}z(\cdot)$, then
\[
   M_n - (0,0,\zeta_n)
\]
converges smoothly to a pair of grim reaper surfaces over $\RR\times (b,b+\pi)$ and $\RR\times (-(b+\pi),-b)$.
\end{theorem}

\begin{proof}
That $B=b+\pi$ was proved in Theorem~\ref{prong-theorem}\eqref{planes-item}.

Thus, for large $n$, $B_n>b_n$, and therefore 
\[
  y(M_n) \ge B_n - \pi
\]
by Theorem~\ref{b-B-criterion}.  Thus
\[
   \liminf y(M_n) \ge B-\pi = b.
\]
On the other hand, $\limsup y(M_n)\le b$ trivially.  Thus $y(M_n)\to b$.

Let $M$ be a subsequential limit of $M_n-(0,0,\zeta_n)$.
Note that $M$ is also a subsequential limit of 
\[
   M_n' :=   M_n \cap \{|x|< x(M_n)\} - (0,0,\zeta_n)
\]
and that $\mathsf{N}(x(\cdot)| M_n')=0$.  Thus $\mathsf{N}(x(\cdot)| M) =0$.
Let $\Sigma$ be a component of $M$.
By Theorem~\ref{new-finite-type-theorem}, 
 $\Sigma$ is a plane parallel to $\{y=0\}$, or a $\Delta$-wing, or a grim reaper surface.
 Since $\max_{M\cap\{x=0\}}z(\cdot)=0$, $\Sigma$ cannot be 
a plane.
Thus $\Sigma$ is a $\Delta$-wing or a grim reaper surface.

Since $M_n$ is disjoint from $\{x=0, \, |y|<y(M_n)\}$, we see (by  Letting $n\to \infty$) that
$M$ is disjoint from $\{x=0,\, |y|<b\}$.   Thus $\Sigma\cap\{x=0\}$ lies in $\{b\le |y|\le B\}$.
Since $B=b+\pi$, $\Sigma$ is an untilted grim reaper surface.

We have shown that $M'$ is a union of untilted grim reaper surface. 
 Exactly as in the proof of Theorem~\ref{prong-ends-theorem},
$M'$ consists of exactly $2$ grim reaper surfaces, and we get convergence, not merely subsequential convergence.
\end{proof}

\section{Small Necks}

In this section, we show that if its neck is very small, then an annuloid in $\Aa$ resembles two $\Delta$-wings joined
by a small catenoidal neck.

\begin{theorem}\label{small-neck-theorem}
Suppose that $M_n\in \Aa$, that $b_n:=b(M_n)\to b<\infty$, and that $x(M_n)\to 0$.
Then
\begin{enumerate}[\upshape(1)]
\item\label{small-neck-1}  $M_n/x(M_n)$ converges to the catenoid whose waist is the unit circle in $\{z=0\}$.
\item\label{small-neck-2}  There is an $n_0$ such that if $n\ge n_0$, then $\waist(M_n)$, the set of points $p$ where $\nu(M_n,p)$ is horizontal, is a smooth closed curve that is contained in $\BB(0,2x(M))$ and 
 that projects diffeomorphically to $\partial K_n$, where $K_n$ is a compact, strictly convex
region in $\RR\times (-b_n,b_n)$.
\item\label{small-neck-3}  If $n\ge n_0$, then one of the components of $M_n\setminus \waist(M_n)$ is a graph over
\[   
     (\RR\times (-B_n,B_n)) \setminus K_n,
\]
and the other is a graph over
\[
    (\RR\times (-b_n,b_n)\setminus K_n.
\]
\item\label{small-neck-4}   $M_n$ converges to $\graph(f_b)$, where $f_b:\RR\times(-b,b)\to \RR$ is the translator
such that $f_b(0,0)=0$ and $Df_b(0,0)=0$.   The convergence is smooth with multiplicity $2$ away from the origin.
\item\label{small-neck-5} $B_n:=B(M_n)$ converges to $b$.
\end{enumerate}
\end{theorem}

Later (Theorem~\ref{cutoffs-theorem}) 
   we will show that  $B_n=b_n$ for all sufficiently large $n$, provided $b>\pi/2$.  (We do not know
whether ``provided $b=\pi/2$'' is necessary.)

\begin{proof}
Assertion~\eqref{small-neck-1} was proved in Lemma~\ref{cat-lemma}.

If $n$ is large, then, by Assertion~\eqref{small-neck-1}, $M_n$ contains
a curve $C_n$ that is a slight perturbation of the circle $\{z=0\}\cap\partial\BB(0,x(M_n))$
such that $\Tan(M_n,\cdot)$ is vertical at each point of $C_n$.  It follows that $\nu(M_n,\cdot)$
maps $C_n$ diffeomorphically onto the equator.  By Corollary~\ref{x-critical-corollary}, $C_n$ contains all the points
of $\waist(M_n)$.  Thus Assertion~\eqref{small-neck-2} holds.

Assertion~\eqref{small-neck-3} follows from Assertion~\eqref{small-neck-2} by Lemma~\ref{degree-lemma} below.

Note that $\nu\cdot\ee_3$ is $>0$ on one component $M_n^+$ of  components of $M_n\setminus \Gamma_n$ 
and that $\nu\cdot \ee_3$ is $<0$ on the other component, $M_n^-$.
Each component is stable since it is a graph.

By passing to a subsequence, we can assume that $M_n^+$ and $M_n^-$ converge as sets to limits $M^+$
and $M^-$, both containing the origin, and that $B_n$ converges to a limit $B$.
 By stability, the convergence is smooth away from the origin.
By Remark~\ref{cat-lemma-remark}, $M':=M^+\cup M^-$ is smooth and embedded (possibly with multiplicity) and $\Tan(M',0)$
is horizontal.   Let $\Sigma^+$ and $\Sigma^-$ be the components of $M^+$ and $M^-$ containing the origin.
Since $M$ is smooth and embedded, $\Sigma^+$ and $\Sigma^-$ coincide near $0$.  Thus, by unique continuation,
$\Sigma^+=\Sigma^-$.

Now $\nu\cdot \ee_3$ is a non-negative Jacobi field on $\Sigma^+$ that is $>0$ at $0$.  Hence $\nu\cdot\ee_3>0$
everywhere on $\Sigma^+$, and therefore $\Sigma^+$ is $\graph(f_\beta)$ for some $\beta$, where 
\[
   f_\beta: \RR\times (-\beta,\beta)\to\RR
\]
is the translator with $f_\beta(0,0)=0$ and $Df_\beta(0,0)=0$.    Since $\Sigma^-=\Sigma^+$, 
we see that $\Sigma^-=\Sigma^+=\graph(f_\beta)$.   
Since $M_n^-$ lies in $\{|y|\le b_n\}$, we see that $\Sigma^-$ lies in $\{|y|\le b\}$. 
Thus 
\begin{equation}\label{beta-is-less}
\beta\le b.
\end{equation}

Note that $M'+(0,0,z)$ converges as $z\to\infty$ to a limit $M''$ consisting of the planes $\{y=\pm \beta\}$, each with multiplicity $2$.
Now $M''\in \Ll(b,B)$, so by Theorem~\ref{planes-theorem}, the plane $\{y=B\}$ is in $M''$.  Thus (by~\eqref{beta-is-less}) $\beta=b=B$.
\end{proof}

\begin{lemma}\label{degree-lemma}
Suppose that $M\in \Aa$, and that the waist of $M$ is a smooth closed curve $C$ that
projects diffeomorphically onto the boundary of a compact, strictly convex set $K$ in
$\RR\times (-b,b)$, where $b=b(M)$.  
Let $M^+$  be the set of points of $M$ where $\nu\cdot\ee_3>0$ and $M^-$ be the set
of points of $M$ where $\nu\cdot \ee_3<0$.  Then $M^+$ projects diffeomorphically
onto 
\[
    (\RR^2\times (-B,B)) \setminus K,
\]
and $M^-$ projects diffeorphically onto 
\[
   (\RR^2\times (-b,b)) \setminus K,
\]
where $b=b(M)$ and $B=B(M)$.
\end{lemma}

\begin{proof}
We give the proof for $M^-$.  (The proof for $M^+$ is the same.)
Let
\begin{align*}
&\pi: M^- \to \RR^2, \\
&\pi(x,y,z) = (x,y).
\end{align*}
Let $\Omega$ be a connected component of $\RR^2\setminus (\{y=\pm b\} \cup \partial K)$.
Note that if $Q$ is a compact subset of $\Omega$, then $\pi^{-1}(Q)$ is compact.
Then (since $\nu\cdot\ee_3<0$ everwhere in $M^-$), for $q\in \Omega$, the number $d_\Omega$
of points in $\pi^{-1}(q)$ is independent of $q$.

If $\Omega$ is the interior of $K$, then $d_\Omega=0$ since $\pi^{-1}(0,0)=Z\cap M^-$ is empty.
if $\Omega$ is the region $\{y>b\}$ or the region $\{y<-b\}$, then $d_\Omega=0$
since $M$ lies in the slab $\{|y|\le B\}$.  Finally, if $\Omega$ is the region 
\[
   \RR^2\setminus K,
\]
then $d_\Omega=1$, since for $x> x(M)$, $\pi^{-1}(x,0)$ has exactly one point,
the point $(x,0, \ulow(x))$.
\end{proof}

\begin{remark}
Whether every $M\in \Aa$ satisfies the hypothesis (and therefore also the conclusion)
of Lemma~\ref{degree-lemma} is an interesting open question.
By Theorem~\ref{small-neck-theorem}, the hypothesis is satisfied when the necksize $x(M)$
is sufficiently small.
\end{remark}

The following corollary shows that several notions of necksize are very nearly the same when $x(M)$ is small.

\begin{corollary}\label{neck-notions-corollary}
Let $M_n$ be as in Theorem~\ref{small-neck-theorem}.  Then
\[
  \lim_{n\to\infty} \frac{y(M_n)}{x(M_n)} = 1
\]
and
\[
   \lim_{n\to\infty} \frac{L_n}{2\pi x(M_n)} = 1
\]
where $L_n$ is the length of the shortest homotopically nontrivial closed curve in $M_n$.
\end{corollary}

%%
%%
%%
%%
%%

%%%
%%%
%%%
\section{Continuity and Properness}\label{properness-section}

In this section, we show that $b(M)$, $B(M)$, $x(M)$, and $y(M)$ depend continuously on $M\in\Aa$.
We also show that the map $M\in \Aa\mapsto (b(M),x(M))$ is proper.

\begin{theorem}\label{properness-theorem}
Suppose that $M_n\in \Aa$, and that $b_n:=b(M_n)$ and $x_n:=x(M_n)$ converge to finite
limits $b$ and $\hat x$.  
Then, after passing to a subsequence, $M_n$ converges to an $M$ such that $b(M)=b$, $x(M)=\hat x$,
and $B(M)=\lim_nB(M_n)$.

In particular, the map
\begin{align*}
&\Phi: \Aa \to [\pi/2,\infty)\times (0,\infty)  \\
&\Phi(M) = (b(M), x(M))
\end{align*}
is proper and surjective.
\end{theorem}

\begin{proof}
After passing to a subsequence, we can assume that the $B_n:=B(M_n)$ converge to a limit $B \in [b,b+\pi]$ and (by Theorem~\ref{R3-theorem}) that the $M_n$ converge as sets to a limit $M$:
\begin{equation}\label{itamar}
   \delta_n(M_n,M) \to 0,
\end{equation}
where $\delta(\cdot,\cdot)$ is the metric on closed subsets of $\RR^3$ defined in Appendix~\ref{sets-appendix}.
Now $M_n\in \Aa(b_n,B_n,x_n)$ (by Theorem~\ref{b-B-theorem}), so, by definition of $\Aa(b_n,B_n,x_n)$, there exists a surface
$\tilde M_n\in \Rr$ such that
\begin{equation}\label{bery}
\begin{gathered}
|b(\tilde M_n) - b_n| < 2^{-n}, \\
|B(\tilde M_n) - B_n| < 2^{-n}, \\
|x(\tilde M_n) - x_n| < 2^{-n}, \\
a(\tilde M_n)> n, \\
\delta(M_n', M_n) < 2^{-n}, 
\end{gathered}
\end{equation}
where $M_n':=\tilde M_n - (0,0, z(\tilde M_n))$.

By~\eqref{itamar} and~\eqref{bery}, $\delta(M_n',M)\to 0$.
The curvature and area bounds imply that the convergence is smooth.
Again, by~\eqref{bery}, $b(\tilde M_n)\to b$ and $x(\tilde M_n)\to \hat x$.
Thus $M\in \Aa(b, B, \hat x)$ (by definition of $\Aa(b,B,\hat x)$), and therefore $b(M)=b$, $B(M)=B$,
and $x(M)=\hat x$.
\end{proof}

\begin{theorem}\label{b-B-x-continuity-theorem}
The maps $M\mapsto b(M)$, $M\mapsto B(M)$, and $M\mapsto x(M)$ are continuous maps
from $\Aa$ to $\RR$.
\end{theorem}

\begin{proof}
Suppose that $M_n\in \Aa$ converges to $M\in \Aa$.  Trivially $x(M_n)\to x(M)$.
Let $b_n=b(M_n)$, $b=b(M)$, $B_n=B(M_n)$, and $B=B(M)$.
By passing to a subsequence, we can assume that the $b_n$  and $B_n$ converge to a limits $b'$ and $B'$ in $[\pi/2,\infty]$.

First, we claim that $b'\le b$.  For if not, let $b<\beta< b'$.  For all sufficiently large $n$, $b_n>\beta$.
For such  $n$,
\[
  (\unlow)'(x) \le \pdf{}x f_{b_n}(x,0)  < \pdf{x} f_\beta(x,0).
\]
Letting $n\to\infty$ gives
\[
   (\ulow)'(x) \le \pdf{}x f_\beta(x,0)
\]
and therefore (letting $x\to \infty$), $-s(b) \le -s(\beta)$, so $b\ge \beta$, a contradiction.

Since $B_n\in [b_n, b_n+\pi)$, we see that $B'\in [b',b'+\pi]$.

By the Properness Theorem~\ref{properness-theorem}, we can assume, after passing to a subsequence, that
 $M_n$ converges to a limit $M'$ with $b(M')=b$ and $B(M')=B$.  Since $M_n$ also converges to $M$, we
 see that $M'=M$ and thus $b(M)=b(M')=b$ and $B(M')=B(M)=B$.
\end{proof}

Continuity of the map $M\in\Aa\mapsto y(M)$ requires a different argument.

\begin{lemma} \label{b-B-lemma}
Suppose $M_n\in \Aa$ converges to $M\in \Aa$, and suppose that $\zeta_n\to \infty$.
Then $M_n':=M_n+(0,0,\zeta_n)$ converges smoothly to the planes $\{y=\pm b\}$ and $\{y=\pm B\}$,
where $b=b(M)$ and $B=B(M)$.
\end{lemma}

\begin{proof}
After passing to a subsequence, we can assume that $M_n'$ converges to a limit $M'$ consisting of planes, $\Delta$-wings,
and grim reaper surfaces.  Let $x> x(M)+\pi$ and $z\in \RR$.   By the smooth convergence, $x>x(M_n)+\pi$ for large $n$ and 
\[
   \unlow(x(M_n)+x)) \to \unlow(x(M)+x).
\]
Thus 
\[
   (z - \zeta_n) - \unlow(x(M_n)+x) \to -\infty.
\]
Hence by Theorem \ref{fundamental-theorem}, 
\[
    |\ynin(x,z - \zeta_n) - b_n| \to 0
\]
where $b_n=b(M_n)$.  Since $b_n\to b$ (by Theorem~\ref{b-B-theorem}), 
\[
  \ynin(x,z  - \zeta_n)\to b.
\]
This holds for all $(x,z)$ with $x>x(M)$.  Hence
$M'$ coincides with the planes $\{y=\pm b\}$ and $\{y=\pm B\}$
in the halfspace $\{x>x(M)\}$.   Since $M'$ consists of planes, $\Delta$-wings, and grim reaper surfaces,
in fact $M'$ consists of the planes $\{y=\pm b\}$ and $\{y=\pm B\}$.
\end{proof}

\begin{theorem} The map $M\in \Aa \mapsto y(M)$ is continuous.
\end{theorem}

\begin{proof} 
Suppose that $M_n\in \Aa$ converges to $M\in \Aa$.  By Theorem~\ref{b-B-x-continuity-theorem},
$B_n:=B(M_n)$ and $b_n:=b(M_n)$ converge to $B=B(M)$ and $b=b(M)$.
By passing to a subsequence, we can assume that $y(M_n)$ converges.

Let $p\in M\cap \{y=0\}$.  Then there exist $p_n\in M\cap \{x=0\}$ such that $p_n\to p$.
Now
\[
    y(M_n) \le |y(p_n)| \to |y(p)|,
\]
so 
\[
  \lim y(M_n) \le |y(p)|.
\]
Taking the infimum over $p\in M\cap\{x=0\}$ gives
\[
 \lim y(M_n) \le y(M).
\]
Thus to prove the theorem, it suffices to prove the reverse inequality.

Choose $p_n=(0,y_n,z_n)\in M_n$ with 
\begin{equation}\label{squeeze}
0 \le |y_n| < y(M_n) + (1/n).
\end{equation}
By symmetry, we can assume $y_n\ge 0$.
We can also assume that $y_n$ converges to a limit $y_\infty$.  Thus
\[
  y_\infty \le \lim_n y(M_n)
\]
by squeeze.

{\bf Case 1}: $z_n$ is bounded below.  
Then (passing to a subsequence) we can assume that $p_n$ converges to a point $p=(0,y_\infty,z_\infty)\in M$.
Thus
\[
     y(M) \le y(p) = y_\infty \le \lim_n y(M_n).
\]

{\bf Case 2}: $z_n$ is not bounded below.  Then, passing to a subsequence, we can
assume that $z_n\to -\infty$.

By Lemma~\ref{b-B-lemma}, $M_n':= M_n + (0,0,z_n)$ converges smoothly to the planes $\{y=\pm b\}$ and $\{y=\pm B\}$.
Hence $y_\infty$ is $b$ or $B$.  In particular, $b\le y_\infty$.
Trivially, $y(M)\le b$.  Thus
\[
   y(M) \le b \le y_\infty \le \lim y(M_n).
\]
\end{proof}

%%%
%%%
%%%

%%
%%
%%
%%
%%

\section{Gap Theorems}\label{gap-section}

In the next two sections, we prove  results that will be used to establish  the existence  of  compact translating annuli bounded by pairs of nested rectangles (The existence is proved in Section 19 using a path-lifting argument). 
In this section, we establish conditions guaranteeing that a pair of curves does not bound a connected translator.

\begin{theorem}\label{gap-theorem}
Suppose that $I$ is an infinite open strip in $\RR^2$ of width $\pi$.   Suppose that $M$ is a properly immersed translator in $\RR^3\cap\{z\ge 0\}$ with no boundary in the slab $S:=I\times\RR$.  Then $M$ lies in the complement of $S$.
\end{theorem}

\begin{proof}
We may assume that the strip $I$ is $\RR\times(-\pi/2,\pi/2)$.  For $\alpha,\beta>0$, let
\[
    f_{\alpha,\beta}: [-\alpha,\alpha] \times [-\beta,\beta] \to \RR
\]
be the unique graphical translator~\cite{graphs}*{\S3} with boundary values $0$.
By the maximum principle, if $\beta<\pi/2$, then the graph of $f_{\alpha, \beta}$ lies below $M$.  That is,
for $(x,y,z)\in M$ with $(x,y)\in [-\alpha,\alpha]\times [-\beta,\beta]$, 
\[
   z \ge f_{\alpha,\beta}(x,y).
\]
Letting $\beta\to \pi/2$, we see that
\[
   z\ge f_{\alpha,\pi/2}(x,y).
\]
As $\alpha\to \infty$, $f_{\alpha,\pi/2}(x,y)-f_{\alpha,\pi/2}(0,0)$ converges smoothly to the untilted 
grim reaper surface $f_{\pi/2}(x,y):=\log(\cos y)$.
(See the proof of \cite{graphs}*{Theorem~4.1}.)
  Thus $f_{\alpha,\pi/2}$ converges to $\infty$ uniformly on compact
subsets of $\RR\times(-\pi/2,\pi/2)$.  The result follows immediately.
\end{proof}

\begin{remark}\label{gap-remark}
Note this does not require any regularity of $M$: $M$ can be any closed set that satisfies the maximum principle.
(In the language of \cite{white16}, one could state the result as follows:
Suppose $\Sigma$ is a $(2,0)$ set (with respect to the Ilmanen metric) in $I\times\RR$ that lies in $\{z\ge 0\}$.  Then $\Sigma$
is empty.)
\end{remark}

\begin{corollary}\label{convex-corollary}
If $M$ is a translator in $\{z\ge 0\}$, then $M$ lies in $C\times [0,\infty)$, where
$C$ is the convex hull of the projection of $\partial M$ to the horizontal plane.
\end{corollary}

\begin{lemma}\label{boundedness-lemma}
Let $M \subset \R^3 \cap \{z\ge 0\}$ be a translator such that $\partial M$  lies in a compact set $K$.  Then $M$ is bounded. In particular, if $K$ lies below a bowl soliton $Q$, then $M$ also lies below $Q$.
\end{lemma}

\begin{proof}
Let $R = \max\{|p| \; : \; p \in \partial M\}$. By Corollary~\ref{convex-corollary}, 
\[
    M \subset \{(x,y,z) \; : \; x^2 + y^2 \leq R^2, \, z\ge 0\}.
 \]
By \cite{CSS}, there exists a complete translating annulus $\Sigma$ such that $\Sigma$ is rotationally
invariant about the $z$-axis and such that
\[
\min_{p \in \Sigma} \dist(p,Z) = R.
\]
By translating $\Sigma$ vertically, we can assume that the neck of $\Sigma$ is at a height $h > 0$
and that $\Sigma$ is disjoint from $\{(x,y,z): x^2+y^2 + z^2\le (2R)^2\}$.
Let $\mathbf{v}$ be a horizontal unit vector. By the maximum principle, $\Sigma+s\mathbf{v}$ is disjoint
from $M$ for $0 \leq s \le R.$ It follows that M is disjoint from the plane $\{z = h\}$, and
thus that $M$ lies in $$\{x^2 + y^2 \leq R^2\} \cap \{0\leq z \leq h \}.$$
The last statement ( ``In particular...") follows immediately from the maximum
principle. (If $M$ did not lie below $Q,$ then there would be a largest $s > 0$ such that
$M$ has a nonempty intersection with $Q + s\mathbf{e}_3.$ For that s, the strong maximum
principle would be violated at the point of contact of $M$ and $Q + s\mathbf{e}_3.$)
\end{proof}

\begin{theorem} \label{th:gap}
Let $U_n\subset U_n'$ be open convex regions in $\RR^2$ such that $U_n$ converges to a bounded open 
convex set $U$ and such that $U_n'$ converges to an infinite strip $U'$.
Suppose that
\[
    \min\{ |p-q|: p\in \partial U, q\in \partial U'\} \ge \pi.
\]
Then for all sufficiently large $n$, there is no connected translator in $\{z\ge 0\}$ whose
boundary is $S_n:=((\partial U_n)\cup (\partial U_n'))\times\{0\}$.
\end{theorem}

\begin{proof}
Suppose the result is false.  Then (after passing to a subsequence) each $S_n$ bounds a connected translator $M_n$
in $\{z\ge 0\}$.  Passing to a further subsequence, 
we can assume that the $M_n$ converge as sets to a closed set $M$ in $\RR^3$. 
Note that $U'\setminus U$ contains two parallel infinite strips $I_1$ and $I_2$ each of width $\pi$.  Thus
by Theorem~\ref{gap-theorem} (and Remark~\ref{gap-remark}), $M$ is disjoint from $I_1\times\RR$ and $I_2\times \RR$.
Thus $M$ is the union of three connected components, where one component, $M^*$, has boundary $(\partial U)\times\{0\}$,
and where each of the other two components is bounded by one of the straight lines in 
  $(\partial U')\times \{0\}$.

By Lemma \ref{boundedness-lemma}, $M^*$ is compact, a contradiction.

(If the contradiction is not clear, let $K$ be a compact set such that $M^*$ is in the interior
of $K$ and such that $M\setminus M^*$ is disjoint from $K$.  For all sufficiently large $n$,
$M_n$ contains a point in $\partial K$, and therefore $M\cap \partial K$ is nonempty, 
a contradiction.)
\end{proof}

\section{Limits when $\partialin M$ approaches $\partialout M$}

\begin{lemma}\label{torus-lemma}
Let $C$ be a smooth convex curve in $\{z=0\}$ such that at each point, the radius of curvature is $\ge R>0$.
If $0<r<R$, then the set 
\[
   T(C,r)= \{p\in \RR^3: \dist(p,C)\le r\}
\]
is a solid torus whose boundary is foliated by circles of radius $r$.  At each point of the boundary, the mean curvature
with respect to the inward pointing normal $\nu$ is greater than or equal to 
\[
 \frac1r - \frac1{R-r}.
\]
\end{lemma}

The proof is a standard and straightforward, so we omit it.  (If $\uu$ and $\vv_2$ are orthornomal vector fields tangent to 
$\partial T(C,r)$ with
$\uu$ tangent to the circles, then $\nabla_{\uu}\uu\cdot\nu = 1/r$
and $\nabla_\vv\vv\cdot\nu \ge - 1/(R-r)$.)

If $r\le R/3$, then $R-r\ge 2r$, so the mean curvature is $\ge 1/r - 1/(2r)=1/(2r)$.  Hence we have:

\begin{corollary}\label{torus-corollary}
Let $h\ge 0$.
In Lemma~\ref{torus-lemma}, if $r\le \min\{R/3, 1/(2h)\}$, then the mean curvature of $T(C,r)$ with respect to the 
inward unit normal is everywhere $\ge h$.
In particular, if 
\[
   r\le \min\left\{ \frac{R}3, \frac12 \right\},
\]
then $T(C,r)$ is mean-convex with respect to the translator metric.
\end{corollary}

\begin{theorem}\label{graphs-theorem}
Suppose $R>0$, and let
\[
   r= \frac12 \min\{R/3, 1/2\}.
\]
Suppose $A$ is an annulus bounded by a pair of smooth, nested, convex curves  in $\{z=0\}$.  Suppose that the inner curve, $C_1$, has radius of curvature
everywhere $\ge R$, and that $A$ lies in $T(C_1,r)$.  Then there is exactly one translator $M$ in $T(C_1,r)$ with boundary $\partial A$,
and it is a graph over $A$.
\end{theorem}

\begin{proof}
Let $C$ be the set of points $p$ in the the unbounded component
of $\{z=0\}\setminus C_1$ such that $\dist(p,C_1)=r$.   Then $C$ is a convex curve whose radius of curvature is everywhere $\ge R$.

\begin{claim}
Suppose $M$ is a translator in $T(C_1,r)$ with boundary $\partial A$.  Then $M$ is a graph over $A$.
\end{claim}

To prove the claim, note that
Since $M$ lies in $T(C_1,r)$, it also lies in $T(C,2r)$.
  By Corollary~\ref{torus-corollary}, $T(C,\rho)$ is $g$-mean-convex for every $\rho\le 2r$.
For $\rho\in [r,2r]$, $A$ is contained in $T(C,\rho)$.  Thus, by the maximum principle, 
\begin{equation}\label{cookie}
M  \subset T(C,r).
\end{equation}
Note that the projection of $T(C,r)$ to $\{z=0\}$ (which is also $T(C,r)\cap \{z=0\}$) is an annulus $A'$ whose inner boundary is $C_1$.
Thus $M$ lies in $A'\times \RR$.  Let $D$ be the convex planar region bounded by $C_2$, the outer boundary curve of $A$.
By the maximum principal, $M$ lies in $D\times\RR$ 
and $M\setminus\partial M$ lies in the interior of $D\times\RR$.
Thus, by~\eqref{cookie}, $M$ lies in $A\times\RR$
and $M\setminus \partial M$ lies in the interior of $A\times \RR$.

It follows that $M$ is a graph over $A$.  For if not, there would be a $t>0$ such that
$M\cap (M+t\ee_3)$ is nonempty, and, at the largest such $t$, we would get a violation of the maximum principle.
This completes the proof of the claim. 

It follows from the claim that there is at most one translator in $T(C_1,r)$ with boundary $\partial A$.

Thus to complete the proof of the theorem, it suffices to prove that there is at least one such translator.
For that, we can let $M$ be a $g$-area-minimizing surface (flat chain mod $2$ or integral current) in $T(C_1,r)\cap \{z\ge 0\}$
with boundary $\partial A$.
\end{proof}

\begin{lemma}[Dichotomy Lemma]\label{dichotomy-lemma}
Let $M_i\in \Cc$ and suppose that the two components of $\partial M_i$ both converge to the same convex curve $\Gamma$.
Let $D$ be the translating graph bounded by $\Gamma$.
Then, after passing to a subsequence, one of the following occurs:
\begin{enumerate}
\item $\max_{p\in M_i} \dist(p,\Gamma)\to 0$, in which case $x(M_i)\to x(\Gamma)$, or
\item $M_i$ converges to $D$.  Away from $\Gamma\cup Z$, the convergence is smooth with some positive integer multiplicity.
   In this case, $x(M_i)\to 0$.
\end{enumerate}
\end{lemma}

\begin{proof}
Let $D_i$ and $D_i'$ be the translating graphs bounded by the inner and outer components of $\partial M_i$.
By the maximum principle (consider vertical translates of $D_i$ and of $D_i'$), $M_i$ lies in the closed
region of $\{z\ge 0\}$ between $D_i$ and $D_i'$.

Thus, after passing to a subsequence, the $M_i$ converge as sets to a closed subset $M$ of $D$.

By the curvature and area bounds, away from $Z\cup \Gamma$, the convergence is smooth with some integer multiplicity $k\ge 0$.

If $k=0$, then $\Gamma\subset M \subset \Gamma\cup (Z\cap D)$.
Since each $M_i$ is connected, $M$ is connected and therefore $M=\Gamma$.  Thus 
\[
 \max_{p\in M_i}\dist(p,\Gamma)\to 0
\]
and $x(M_i)\to x(\Gamma)$.

Now suppose that $k>0$.  Let $\eps>0$ be small and let 
\[
     U= \{p \in \RR^3: \dist(p, Z\cap D)>\eps\}.
\]
Then for large $i$, $M\cap U$ consists of $k$ components, each of which converges smoothly as $i\to\infty$ to $D\cap U$.
Hence $x(M_i) \to 0$.
\end{proof}

\TOCstop
\section*{The space $\Pairs$}
\TOCstart

Let $\Pairs$ be the space of curves $\Gamma$ in $\{z=0\}$
such that $\Gamma$ consists of a pair of disjoint, convex, simple closed curves, each of which is symmetric about the $x$ and $y$ axes.
As in Definition~\ref{Cc-definition},
   $\Cc$ is the space of embedding translating annuli $M$ such that $\partial M\in \Pairs$, $M$ is invariant under 
reflection in the planes $\{x=0\}$ and $\{y=0\}$, and $M\cap Z=\emptyset$.

\begin{comment} old version
\begin{lemma}\label{proper-lemma}
The boundary map $\partial: M\mapsto \partial M$ is a continuous, proper map from $\Cc$ to $\Pairs$.
\end{lemma}
\end{comment}

\begin{lemma}\label{proper-lemma}
Suppose that $M_i\in \Cc$ and that $\partial M_i \to \Gamma\in \Pp$. 
Then, after passing to a subsequence, the $M_i$ converge to a limit $M\in \Cc$,
and the convergence is smooth in $\{z>0\}$.
Furthermore, if $\partial M_i$ and $\Gamma$ are $C^{2,\alpha}$ and if
the $\partial M_i$ converge to $\Gamma$ in $C^{2,\alpha}$, then the $M_i$
converge to $M$ in $C^{2,\alpha}$.
\end{lemma}

The first statement of the lemma is the assertion that 
the boundary map $\partial: M\mapsto \partial M$ is
 a continuous, proper map from $\Cc$ to $\Pairs$.
 
\begin{proof}
Recall that the $M_i$ are minimal surfaces with respect to the translator metric.
By the standard compactness theory (e.g., \cite{white18}*{Theorem~1.1})
and the curvature bounds in Theorem~\ref{curvature-bound-theorem}\eqref{in-addition-item}, we can assume that the $M_i$ converge to a limit $M$,
where $M\setminus \Gamma$ is smooth and embedded (possibly with multiplicity) and where the convergence
of $M_i$ to $M$ is smooth in compact subsets of $\RR^3\setminus (\Gamma\cup S)$, for some
locally finite subset  $S$ of $Z$.
By Corollary~\ref{bounded-corollary},
 the $M_i$ (and therefore also $M$) all lie some compact subset of $\RR^3$.
Let $\eps_i$ be the length of the shortest closed geodesic $\gamma_i$ in $M_i$.  It suffices to show
that $\eps_i$ is bounded away from $0$, as then the $M_i$ converge subsequentially to an embedded
minimal annulus $M\in \Cc$ (by the curvature bound in Theorem~\ref{curvature-bound-theorem}\eqref{curvature-item}, or by standard Douglas-Rado theory.)
Suppose to the contrary that $\eps_i\to 0$.   Note that $\gamma_i$ converges to a point $p$ in $M$.
One component of $M_i\setminus \gamma_i$ converges to a disk $D_\textnormal{out}$ in $M$ with boundary $\Gammaout$
and the other component converges to a disk $D_\textnormal{in}$ in $M$ with boundary $\Gammain$; the convergence is smooth
away from $\{p\}\cup \partial\Gamma$.   A priori, $D_\textnormal{in}$ and $D_\textnormal{out}$ might not be smooth at $\{p\}$.  But 
since $M$ is smooth, it follows that $D_\textnormal{in}$ and $D_\textnormal{out}$ are smooth at $p$.  But then their
contact at $p$ violates the strong maximum principle.
This completes the proof of the first Assertion.

The second assertion follows immediately from 
  Theorem~3 of~\cite{white-curv}.
Choose $r>0$ large enough that the ball $B(0,r)$ of radius $r$ about $0$ in $\RR^3$ contains
all the $M_i$.  Let 
\[
    N = \overline{B(0,2r)}\cap \{z\ge 0\}.
\]
Assertion~(4) of~\cite{white-curv}*{Theorem~3}
  includes the hypothesis that $\partial N$ is smooth and strictly convex
with respect to the Riemannian metric.  (Strict convexity of $\partial N$ at a point
means that the principal curvatures with respect to the inward pointing unit normal are positive.
Equivalently, it means that mean curvature vector points into $N$, and that the product
of the principal curvatures is positive.)
However, the proof of that theorem only uses smoothness and strict convexity of $\partial N$ in some
open subset of $\partial N$ containing $\Gamma$.
Indeed, smoothness and strict {\em mean} convexity of $\partial N$ near $\Gamma$ suffice.
  
Actually, \cite{white-curv}*{Theorem~3} only asserts convergence of $M_i$ to $M$
in $C^{2,\beta}$ for $\beta<\alpha$.  But convergence in $C^{2,\alpha}$ then follows
 by standard Schauder estimates.
\end{proof}

\section{Path Lifting}\label{path-lifting} 

Throughout this section, we shall consider one-parameter families of closed, convex, symmetric sets in 
the plane $\{z=0\}$
\begin{eqnarray*}
t \in [0,1] & \longmapsto & \cin(t) \\
t \in [0,1] & \longmapsto & \cout(t)
\end{eqnarray*}
satisfying:
\begin{enumerate}[(a)]
\item $\cin(0)=\cout(0),$
\item $\cin(t)$ is contained in the interior of $\cout(t)$, for all $t \in (0,1],$
\item\label{nested-outers} $\cout(t)$ is contained in the interior of $\cout(t')$ for all $0\le t < t' \le 1$.
\item $\cin(t)$ is compact with nonempty interior, for all $t \in [0,1],$
\item $\cout(1)$ is a strip $\RR \times [-d,d] \times \{0\}$ such that the distance from 
$\partial \cout(1)$ to $\cin(1)$ is $\geq \pi.$
\end{enumerate}
We define:
\begin{eqnarray*}
\Gamma_{\rm in}(t)& := & \partial \cin(t), \\
\Gamma_{\rm out}(t)& := & \partial \cout(t), \\
\Gamma(t)& := & \Gamma_{\rm in}(t) \sqcup \Gamma_{\rm out}(t).
\end{eqnarray*}

We also define
\begin{align*}
\Cc_{\Gamma(t)} &= \{M\in \Cc: \partial M=\Gamma(t)\},    \\
\Cc_{\Gamma(\cdot)} &= \cup_{t\in (0,1)} \Cc_{\Gamma(t)}.
\end{align*}
and we define a map
\[
   \pi: \Cc_{\Gamma(\cdot)} \to (0,1)
\]
by letting $\pi(M)$ be the unique $t\in (0,1)$ such that
\[
   \partial M = \Gamma(t).
\]
(Uniqueness of $t$ follows from~\eqref{nested-outers}.)

Theorem \ref{th:gap} implies the following:
\begin{lemma} \label{lem:gap}
For $t$ sufficiently close to $1$, $\Gamma(t)$ bounds no connected translator in the
halfspace $\{z \geq 0\}.$ 
\end{lemma}
\begin{figure}[htbp]
\begin{center}
\includegraphics[width=.55\textwidth]{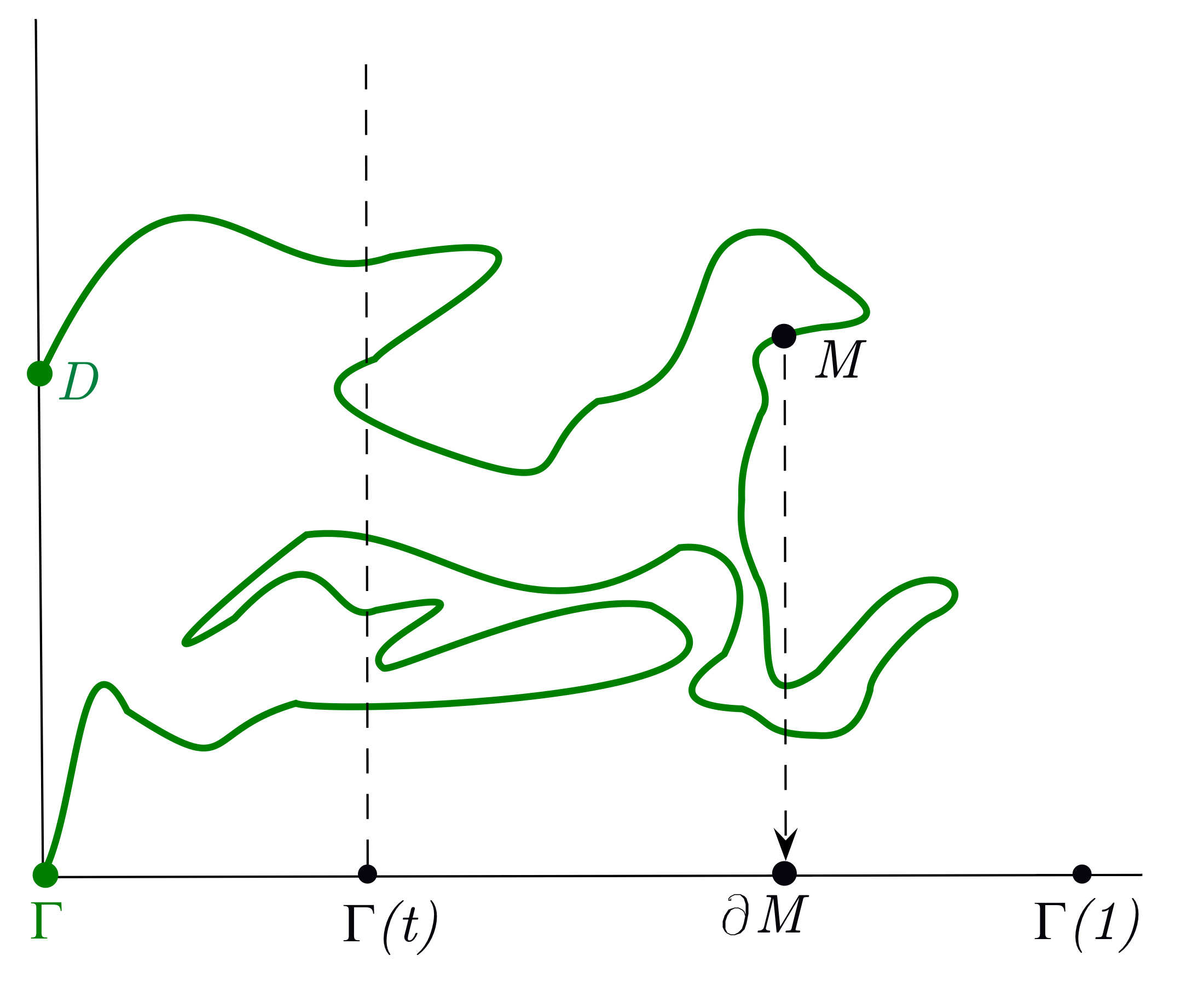}
\caption{\small $\Cc_{\Gamma(\cdot)}$ is homeomorphic to a $1$-manifold without boundary.}
\label{fig:h3}
\end{center}
\end{figure}

Let $\Gamma=\Gamma_{\rm in}(0)=\Gamma_{\rm out}(0)$ and let $D$ be the unique graphical translator with boundary $\Gamma$. The main result of this section asserts the existence of a  connected family, ${\mathcal{F}'}$, of compact translating annuli, each of which has boundary $\Gamma(t)$ for some $t\in(0,1)$. They have  the further property that, when  the elements  are considered as subsets of $\RR^3$, the  closure of $\mathcal{F}'$ is compact and equal to $\mathcal{F}:=\mathcal{F}'\cup\{\Gamma, D\}$.

\begin{theorem} \label{th:family}
Let $\Gamma= \Gamma_{\rm in}(0)=\Gamma_{\rm out}(0).$ 
Let $(x(\Gamma),0,0)$ be the point of $\Gamma$ in the positive $x$-axis.
\begin{enumerate}[\upshape(1)]
\item\label{family-1}
If $0< \hat x < x(\Gamma)$, then there exists a surface $M\in \Cc_{\Gamma(\cdot)}$ such that $x(M)=\hat x$.
\item\label{family-2}
 If $I$ is a closed interval in $(0,x(\Gamma))$, then there is a compact, connected subset $\Gg_I$ of $\Cc_{\Gamma(\cdot)}$ such that
\[
  \{ x(M): M\in \Gg_I\} = I.
\]
\end{enumerate}
\end{theorem}

(We are not asserting uniqueness of the $M$ in Assertion~\eqref{family-1}
 or of the $\Gg_I$ in Assertion~\eqref{family-2}.)

\TOCstop
\section*{A special case}
\TOCstart

\newcommand{\tpp}{\tilde\Pp}
\newcommand{\tcc}{\tilde \Cc}

Fix an integer $k\ge 2$ and an $\alpha\in (0,1)$.
If $M$ and $N$ are smooth manifolds, recall that $C^{k,\alpha+}(M,N)$ is the closure
in $C^{k,\alpha}$ of $C^\infty(M,N)$.  
Let $\tpp$ be the set of $\Gamma\in \Pp$ such that $\Gammain$ and $\Gammaout$ are $C^{k,\alpha+}$ curves
with nowhere vanishing curvature.  Let $\tcc$ be the set of $M\in \Cc$ such that $\partial M\in \tpp$.
 
First, we will prove Theorem~\ref{th:family} under some additional hypotheses.
Then we will deduce the general theorem from the special case. 
The additional hypotheses are:
\begin{enumerate}[\upshape (h1).]
\item\label{special-h1} Each $\Gamma(t)$ is in $\tpp$.
\item\label{special-h2} The map $t\mapsto \Gamma(t)$ is a  a smooth embedding of $[0,1)$ into $\tpp$.
\item\label{special-h3} $\Cc_{\Gamma(\cdot)}$ is homeomorphic to a $1$-manifold without boundary. (The $1$-manifold need not be connected.)
Thus each connected component of $\Cc_{\Gamma(\cdot)}$ is homeomorphic to a circle or to $\RR$ (see Figure~\ref{fig:h3}.)
\end{enumerate}

By~\cite{white87}*{Theorem~3.3(8) and~\S1.5} (see also Remark~\ref{thesis-remark}), $\tcc$ and $\tpp$ are smooth, separable Banach manifolds and 
\begin{equation}\label{the-map}
   \partial: M\in \tcc \mapsto \partial M \in \tpp  \tag{*}
\end{equation}
is a smooth map of Fredholm index $0$.  
By the Sard-Smale Theorem, a generic map satisfying~(h\ref{special-h1}) and~(h\ref{special-h2})
is transverse to the map~\eqref{the-map}, and therefore also satisfies~(h\ref{special-h3}).

\begin{proof}[Proof of Theorem~\ref{th:family} assuming~(h\ref{special-h1})--(h\ref{special-h3})]
By Theorem~\ref{graphs-theorem}, there is a $\delta\in (0,1)$ such that if $t\in (0,\delta]$,
 then $\Gamma(t)$ bounds a translating graph $\Sigma(t)$. By the maximum principle, $\Sigma(t)$ is the unique
graphical translator with boundary $\Gamma(t)$:
\begin{equation}\label{not-graphical}
\text{If $M\in \Cc_{\Gamma(t)}$ and $M\ne \Sigma(t)$, then $M$ is not graphical}.
\end{equation}
The uniqueness implies that $\Sigma(t)$ depends smoothly on $t$ for $t\in (0,\delta]$.

Let $\Ee= \{\Sigma(t): 0<t\le \delta \}$ and let $\Gg$ be the connected component of $\Cc_{\Gamma(\cdot)}$ containing $\Ee$.

As $t\to 0$, the $\Sigma(t)$ converge (as sets) to the curve $\Gamma$.   Thus $\Ee$ is one end of $\Gg$.  
Let $M_i\in \Gg$ be a sequence that diverges to the other end.
Let $t_i=\pi(M_i)$ (i.e, let $\partial M_i=\Gamma(t_i)$).

By Lemma~\ref{lem:gap}, the sequence $t_i$ is bounded above by some $T<1$.
By properness of the map $\partial: \Cc\to \Pp$ (see Lemma~\ref{proper-lemma}), 
$t_i\to 0$.

Thus for large $i$, $t_i\in (0,\delta)$, and, since $M_i\notin \Ee$,
it follows that $M_i$ is not graphical (by~\eqref{not-graphical}).  Thus $x(M_i)\to 0$ by Lemma~\ref{dichotomy-lemma}.

Let $s\in \RR\mapsto M(s)$ be a parametrization of $\Gg$
such that $M(s)\in \Ee$ if and only if $s<0$.

We have shown that $x(M(s)) \to 0$ as $s\to -\infty$ and that $x(M(s))\to x(\Gamma)$
as $s\to\infty$.  Thus $x(M(\cdot))$ takes every value in $(0,x(\Gamma))$,
so Asssertion~\eqref{family-1} holds.

Now let $I=[d_1,d_2]$ be a compact subinterval of $(0,x(\Gamma))$.
Let $s_1$ be the largest $s$ for which $x(M(s))=d_1$.
Now let $s_2$ be the smallest $s\le s_1$ for which $x(M(s))=d_2$.
Then
\[
  \Gg_I:= \{M(s): s_1\le s \le s_2\}
\]
is a compact, connected subset of $\Cc_{\Gamma(\cdot)}$ such that
\[
   \{ x(M): M\in \Gg_I\} = [d_1, d_2].
\]
This completes the proof of Theorem~\ref{th:family}
 assuming the extra hypotheses~(h\ref{special-h1})--(h\ref{special-h3}).
\end{proof}

\TOCstop
\section*{The general case}
\TOCstart

Now we prove Theorem~\ref{th:family} without assuming extra hypotheses:

\begin{proof}[Proof of Theorem~\ref{th:family}]
The curves $\Gamma(t)$ need not be smooth.
However, we can smooth those curves to get, for $n\in \NN$, a family
\[  
    t\in [0,1]\mapsto \Gamma^n(t)
\]
that satisfies the hypotheses of Theorem~\ref{th:family} and also the addition hypotheses~(h\ref{special-h1})
  and~(h\ref{special-h2}).
We can do the smoothing in such a way that:
\begin{equation}\label{smoothing}
   \text{If $t_n\in (0,1]$ converges to $t\in (0,1]$, then $\Gamma^n(t_n)$ converges to $\Gamma(t)$}
\end{equation}

As mentioned at the beginning of the proof of Theorem~\ref{th:family}, 
the condition~(h\ref{special-h3}) is generic.
Thus by making a small, generic perturbation of $t\mapsto \Gamma^n(t)$, we can assume that it also
satisfies the hypothesis~(h\ref{special-h3}). 
Since the perturbations can be arbitrarily small, we can do them in such a way that~\eqref{smoothing} still holds 
 for the perturbed families.

\begin{claim}\label{itamar-claim}
 Suppose that $M_n\in \Cc_{\Gamma^n(t_n)}$ and that $x(M_n)$ converges to a limit 
   $\hat x$ with $0<\hat x < x(\Gamma)$.
Then there exist $i(n)\to\infty$ such that $t_{i(n)}$ converges to a limit $t\in (0,1)$ 
and such that $M_{i(n)}$ converges to a limit $M\in \Cc$ with $\partial M=\Gamma(t)$.

Likewise if there is a subsequence $M_{i(n)}\in \Cc_{\Gamma^{i(n)}}(t_{i(n)})$ with $x(M_{i(n)})$ bounded away from $0$, then, after passing to a further subsequence,
we get convergence to limits $t\in (0,\infty)$ and $M\in \Cc_{\Gamma(t)}$.
\end{claim}

To prove the claim, 
 we can assume, by passing to a subsequence, that the $t_n$ converge to a limit $t\in [0,1]$.
 By Lemma~\ref{dichotomy-lemma}, $t>0$.  By Theorem~\ref{th:gap}, $t<1$.
  Hence $\partial M_n=\Gamma^n(t_n)$ converges
 to $\Gamma(t)$.  
 By properness (Lemma~\ref{proper-lemma}),  the $M_n$ converge (after passing to a further subsequence) 
 to a limit $M\in \Cc$ with $\partial M=\Gamma(t)$.
  This completes the proof of the claim.

Now let $\hat x \in (0,x(\Gamma))$.  Then $\hat x\in (0,x(\Gamma^n))$ for all sufficiently large $n$.
Thus (for such $n$) there exists an $t_n\in (0,1)$ and an $M_n\in \Cc$ such that $\partial M_n=\Gamma^n(t_n)$ and such 
that $x(M_n)=\hat x$.  By the claim, a subsequence of the $M_n$ will converge to an $M\in \Cc_{\Gamma}$
with $x(M)=\hat x$.  This proves Assertion~\eqref{family-1} of Theorem~\ref{th:family}.

Now let $I=[d,d']$ be a compact subinterval of $(0,x(\Gamma))$.
By passing to a subsequence, we can assume that $d' < x(\Gamma^n)$ for all $n$.
Since Theorem~\ref{th:family} holds for $\Gamma^n(\cdot)$, there is a compact, connected set $\Gg^I_n$ of $\Cc_{\Gamma^n(\cdot)}$
such that 
\[
   \{x(M): M\in \Gg^I_n  \} = [d,d'].   
\]
By Theorem~\ref{connected-limits-theorem}, a subsequence of the $\Gg^I_n$  converges
to a compacted, connected subset $\Gg^I$ of $\KK(\RR^3)$, the space of all closed subsets of $\RR^3$.
To simplify notation, we assume that the original sequence $\Gg^I_n$ converges to $\Gg^I$.

Let $M\in \Gg^I$.  Then there exist $M_n\in \Gg^I_n$ such that the $M_n$ converge
as sets to $M$. 
By Claim~\ref{itamar-claim}, $M\in \Cc_{\Gamma(\cdot)}$ and  the convergence is smooth away from the boundary, so
\[
  x(M)= \lim x(M_n).
\]
Thus  $\Gg^I\subset \Cc_{\Gamma(\cdot)}$ and 
\begin{equation}\label{reverse-inclusion}
   \{ x(M): M\in \Gg^I\} \subset [d,d'].
\end{equation}

Now let $\hat x\in [d,d']$.
Then there exist $M_n\in \Gg^I_n$ with $x(M_n)=\hat x$.
By Claim~\ref{itamar-claim}, a subsequence $M_{n(i)}$ will converge to a limit set $M\in \Cc_{\Gamma(\cdot)}$ 
with $x(M)=\hat x$.
The existence of a subsequence $M_{n(i)}\in \Gg^I_{n(i)}$ converging to $M$ means,
 by the definition of $\limsup$ of sets (Definition~\ref{limsup-definition})
that
\[
   M\in \limsup_n \Gg^I_n.
\]
But since the $\Gg^I_n$ converge, the limsup is the same as the limit, so $M\in \Gg$.

We have shown that for each $\hat x\in [d,d']$, there exist $M\in \Gg$ with $x(M)=\hat x$,
so
\[
  [d,d'] \subset \{x(M): M\in \Gg^I\}.
\]
But the reverse inclusion~\eqref{reverse-inclusion} also holds, so the two sets are equal.
\end{proof}

\begin{remark}\label{thesis-remark}
In the discussion at the beginning of the proof of Theorem~\ref{th:family}, we asserted that $\tpp$ was a Banach manifold.
This requires a little explanation, since, in general, the space of $C^k$ (or $C^{k,\alpha}$, or $C^{k,\alpha+})$ compact submanifolds 
of $\RR^N$ is {\em not} a Banach manifold.

If $\Gamma \in \tpp$, let $f_\Gamma: \SS^1\to (0,\infty)^2$ be the map such that
for each $p\in \SS^1$, $f_\Gamma(p)$ is the unique $(r_\textnormal{in}, r_\textnormal{out})\in (0,\infty)^2$ such that
$r_\textnormal{in}p\in \Gammain$ and $r_\textnormal{out}p\in \Gammaout$.
Then
\[
   \Gamma \mapsto f_\Gamma
\]
maps $\tpp$ homeomorphically onto an open subset of $C^{k,\alpha+}_\textnormal{sym}(\SS^1,\RR^2)$, 
the space of $f\in C^{k,\alpha+}(\SS^1,\RR^2)$ that have the symmetries $f(x,y)\equiv f(-x,y)\equiv f(x,-y)$.
  Thus we can regard $\tpp$ as
a smooth, separable Banach manifold.

For various technical reasons, \cite{white87} works with spaces of {\em parametrized} boundaries. But since each $\Gamma\in \tpp$
has a canonical parametrization $f_\Gamma$, here we need not distinguish between parametrized and unparametrized boundaries.
\end{remark}

%%%%%%%%%%%%%%%%

%%%%%%
%

 %%%%%%%%%%%%%%%%%%%%%%

\section{A connected family of Annuloids in $\Aa(b)$}\label{connected-section}

\begin{lemma}\label{connected-lemma}
Let $b\ge \pi/2$ and $a>0$.
\begin{enumerate}[\upshape(1)]
\item\label{connected-lemma-item-1}
 If $\hat x\in (0,a)$, there exists an $M\in \Rr$ with $a(M)\ge a$, $b(M)=b$, and $x(M)=\hat x$.
\item\label{connected-lemma-item-2}
 If $I=[d,d']$ is a compact interval in $(0,a)$, then there exists a compact, connected subset $\Gg^I$
of $\Rr\subset \KK(\RR^3)$ such for each $M\in \Gg^I$, 
\begin{align*}
a(M)&\ge a,\\
b(M)&=b, 
\end{align*}
and such that 
\[
  \{ x(M): M\in \Gg^I\} = I.
\]
\end{enumerate}
\end{lemma}

\begin{proof}
Define a one-parameter family $t\in [0,1]\mapsto \Gamma(t) = \Gammain(t)\cup \Gammaout(t)$ as follows:
\begin{enumerate}[\upshape(i)]
\item $\Gammain(t)$ is the boundary of $[-a,a]\times [-b,b]$. 
\item  $\Gammaout(t)$ is the boundary of $[-A(t), A(t)]\times [-B(t),B(t)]$, where $A(\cdot)$ and $B(\cdot)$ are continuous,
strictly increasing functions such that $A(0)=a$, $A(1)=\infty$, $B(0)=b$, and $B(1)=b+\pi$.
\end{enumerate}
Assertions~\eqref{connected-lemma-item-1} and~\eqref{connected-lemma-item-2}
  now follow immediately from Theorem~\ref{th:family}.
\end{proof}

\begin{theorem}\label{connected-family-theorem}
Suppose $b\ge \pi/2$.
The space $\Aa(b)$ contains a closed, connected subset $\Ff=\Ff(b)$ such that
\[
  \{ x(M): M\in \Ff\} = (0,\infty).
\]
\end{theorem}
We do not know whether there is a unique such subset $\Ff$.

\begin{proof}
Let $I=[d,d']$ be a compact interval in $(0,\infty)$.
Let $a_n$ be a sequence of numbers such that $a_n>d'$ and such that $a_n\to\infty$.
By Lemma~\ref{connected-lemma}, there exists
 a compact, connected subset $\Gg^I_n$ of $\Rr\subset \KK(\RR^3)$ 
such that for each $M\in \Gg^I_n$,
\begin{gather*}
a(M) \ge a_n, \\
b(M)=b, 
\end{gather*}
and such that 
\[
\{x(M): M\in \Gg_n^I\} = I.
\]
Let
\[
 \Ff^I_n = \{ M - (0,0,z(M)):   M \in \Gg^I_n\}.
\]
Then $\Ff^I_n$ is a compact, connected subset of $\KK(\RR^3)$, and
\begin{equation}\label{gonzo}
\{x(M): M\in \Ff_n^I\} = I.
\end{equation}
By passing to a subsequence, we can assume 
(by~Theorem~\ref{connected-limits-theorem}) that the $\Ff^I_n$ converge to a compact, connected
subset $\Ff^I$ of $\KK(\RR^3)$.

By definition of $\Aa(b)$, $\Ff^I\subset \Aa(b)$.  Letting $n\to\infty$ in~\eqref{gonzo} gives
\begin{equation}\label{morg}
    \{ x(M): M\in \Ff^I\} = I.
\end{equation}

We have shown
\begin{equation}\label{intervals}
\begin{aligned}
&\text{For every compact interval $I\subset (0,\infty)$, there is a compact,} \\
&\text{connected family $\Ff^I$ of $\Aa(b)$ such that~\eqref{morg} holds.}
\end{aligned}
\end{equation}
Note that $\Aa(b)$ is a metric space (it is a subspace of $\KK(\RR^3)$), and that
\begin{equation}\label{x-proper}
 \text{$M \mapsto x(M)$ is a continuous, proper map from $\Aa(b)$ to $(0,\infty)$.}
\end{equation}

By a general theorem about metric spaces (Theorem~\ref{connected-extraction-theorem}), 
statements~\eqref{intervals} and~\eqref{x-proper} imply that
the space $\Aa(b)$ contains a closed, connected subset $\Ff(b)$ such that
\[
  \{ x(M): M\in \Ff(b)\} = (0,\infty).
\]
\end{proof}

%%%
%%%
%%%
%%%
%%%
%%%

%%%
%%%
%%%

\section{Capped and Uncapped Annuloids}\label{capping-section}

Suppose $M\in \Aa$ is an annuloid with $B(M)>\pi/2$.  Recall (see \eqref{upper-slope-item} in Theorem \ref{main-theorem-concluded}) that the limit 
$\displaystyle L:=\lim_{x\to\infty}(\uup)'(x)$ exists and is either $s(B)$ or $-s(B)$.
If $L= -s(B)$, we say that the annuloid is {\bf capped}. If $L= s(B)$, we say that it is {\bf uncapped}.

(We do not defined capped and uncapped for annuloids $M$ with $B(M)=\pi/2$.)

\begin{theorem}
Suppose that $M\in \Aa$ and that $B=B(M)>\pi/2$.
If $M$ is uncapped, then 
   $z(\cdot)|M$ is not bounded above.
If $M$ is capped, then $z(\cdot)|M$ attains its maximum.
\end{theorem}

\begin{proof}
The first assertion follows immediately from the definition.  For the second, recall that  the function
\[
\psi_M(t):=\max_{M\cap\{x=t\}} z(\cdot)
\]
is a Lipschitz function. Also, $\psi_M(x)=\uup(x)$ for $x\ge x(M)+\pi$ (by Theorem~\ref{y-graph-theorem}),
 so $\lim_{x\to\infty}\psi_M(x)= -\infty$.
The theorem follows immediately.
\end{proof}

\begin{lemma}
Let $h(x,y)=\log(\cos y)$ and $H(x,y,z)=z-h(x,y)$ for $|y|<\pi/2$.
If $M\in \Cc$ or $M\in \Aa$, then
\[
 \mathsf{N}(H|M\cap\{x>0\}) \le 1.
\]
\end{lemma}

\begin{proof}

By lower semicontinuity, it suffices to prove it for $M\in \Cc$.
We use
\[
 \mathsf{N}(H|M) \le |S| - \chi(M\cap W)
\]
where $S$ is the set of local minima of $H|\partial M$ and $W$ is the slab $\{|y|<\pi/2\}$. 
Note that $S$ consists of four points, namely $(\pm a(M),0,0)$ and $(\pm A(M),0,0)$.  Also, by symmetry,
\[
 \mathsf{N}(H|M) = 2 \mathsf{N}(H|M\cap\{x>0\}) + \mathsf{N}(H|M\cap\{x=0\}).
\]
Thus
\[
  2 \mathsf{N}(H|M\cap\{x>0\}) + \mathsf{N}(H|M\cap \{x=0\}) + \chi(M\cap W) \le 4,
\]
Hence to prove the lemma, it suffices to show that
\begin{equation}\label{to-show}
\mathsf{N}(H|M\cap\{x=0\}) + \chi(M\cap W)\ge 1.
\end{equation}
There are two cases: $y(M)<\pi/2$ and $y(M)\ge \pi/2$.

If $y(M)\ge \pi/2$, then $(M\cap W)\cap\{x=0\}=\emptyset$, so $M\cap W$ has two components, both simply
connected, and thus $\chi(M\cap W)=2$, so~\eqref{to-show} holds.

If $y(M)<\pi/2$, note that the function $H$ is $\infty$ on the endpoints of the curve $M\cap\{x=0\}\cap \{y>0\}$ and thus
it has at least one critical point on that curve. By symmetry, that point is also a critical point of $H|M$.  Thus $\mathsf{N}(H|M\cap\{x=0\})>0$,
so~\eqref{to-show} holds. (Recall that $M\cap W$ has nonnegative Euler characteristic by Lemma~\ref{topology-lemma}.)
\end{proof}

\begin{corollary}\label{capped-critical}
Suppose that $M\in \Aa$ and that $B(M)>\pi/2$.  Then either
\begin{enumerate}
\item $\uup$ is strictly increasing, with no critical points, or
\item $\uup$ has exactly one critical point, an absolute maximum.
\end{enumerate}
In the first case, $M$ is uncapped, and in the second case, $M$ is capped.
\end{corollary}

\begin{proof}
Since any critical point of $\uup$ is also a critical point of $H|M\cap\{x>0\}$, we see (from the lemma) that $\uup$
either has no critical point or exactly one critical point, a nondegenerate one.  The corollary now follows
from the fact that $(\uup)'(x(M))=\infty$.
\end{proof}

\begin{theorem}
Let $\Ss$ be the set of $M\in \Aa$ such that $B(M)>\pi/2$ and such that $M$ is capped. 
Then $\Ss$ is an open subset of $\Aa$.
\end{theorem}

\begin{proof}
Suppose that $M_n\in \Aa$ converge to $M\in \Ss$.
By Corollary~\ref{capped-critical}, $\uup$ has a nondegenerate local maximum (its absolute maximum).
Thus $\unup$ has a local maximum for all sufficiently large $n$.  For such $n$, $M_n$ is capped by 
Corollary~\ref{capped-critical}.
\end{proof}

\begin{theorem}\label{cutoffs-theorem}
Suppose $b\ge \pi/2$.
\begin{enumerate}[\upshape(1)]
\item\label{cutoffs-item-1}
 There is a $\lambda=\lambda(b)$ such that if $M\in \Aa(b)$ and $x(M)>\lambda$, then $B(M)>b$,
   and therefore $M$ is uncapped.
\item\label{cutoffs-item-2}
 If $b>\pi/2$, there is an $\eps=\eps(b)$ with the following property.  If $M\in \Aa(b)$ and if $x(M)<\eps$, 
   then $M$ is capped, and therefore 
   $B(M)=b$.
\end{enumerate}
\end{theorem}

\begin{proof}
Let $M_n\in  \Aa(b)$ with $x(M_n)\to\infty$. 
 Then $B(M_n)\to b+\pi$ (by Theorem~\ref{prong-theorem}\eqref{pi-prong}), so $B(M_n)>b$ for all sufficiently
large $n$.
  If $B(M)>b(M)$, then $M$ is uncapped by Theorem~\ref{main-theorem-concluded}\eqref{uncapped-item}.
    This proves Assertion~\eqref{cutoffs-item-1}.

Now suppose that $b>\pi/2$, that $M_n\in \Aa(b)$, and that $x(M_n)\to 0$.
Then for each $x>0$, $(\unup{})'(x)<0$ for all sufficiently large $n$ by Theorem~\ref{small-neck-theorem}[5].
   If $\unup{}'(x)<0$ for some $x$, then $M_n$
is capped by Corollary~\ref{capped-critical} and therefore $B(M_n)=b$ by Assertion~\eqref{cutoffs-item-1}.
 This proves Assertion~\eqref{cutoffs-item-2}.
\end{proof}

\begin{theorem}\label{uncapping-theorem}
Suppose $b>\pi/2$.
There exist capped annuloids $M_n\in \Aa(b)$ that converge to an uncapped annuloid $M\in \Aa(b)$.
Let $p_n$ be the highest point on $\graph(\unup)$.  Then $M_n-p_n$ converges smoothly as $n\to\infty$
to the $\Delta$-wing in the slab $\{|y|<b\}$ whose highest point is the origin.
\end{theorem}

See   Figure~\ref{peaked-figure}  for a sketch of $M_i\cap \{y=0\}\cap\{x>0\}$ for large $i$.
\begin{figure}[htbp]
\begin{center}
\includegraphics[width=.33\textwidth]{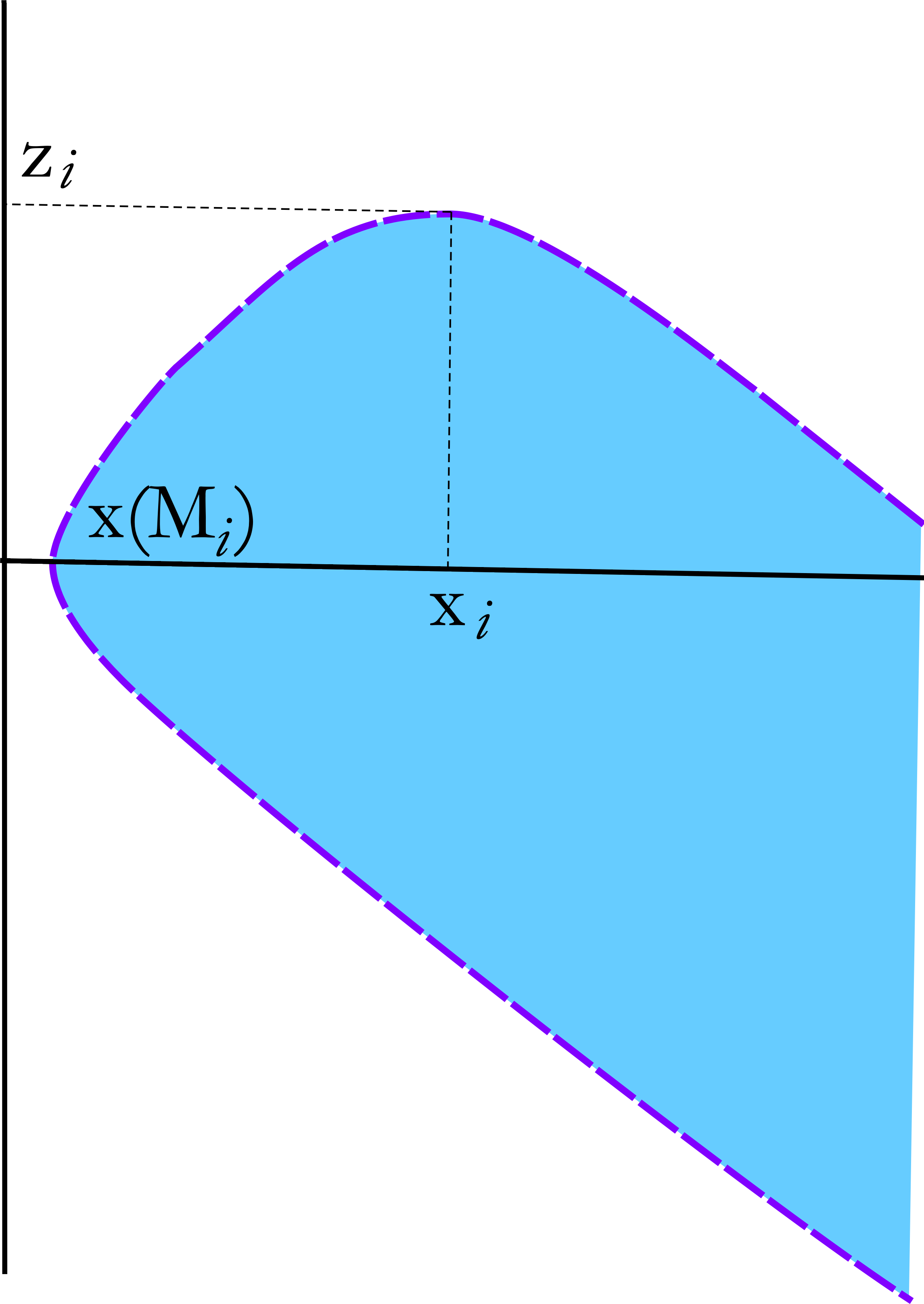}
\caption{\small The curve $M_i\cap \{y=0\}$ in Theorem \ref{uncapping-theorem}.}
\label{peaked-figure}
\end{center}
\end{figure}

\begin{proof}
Recall (Theorem~\ref{connected-family-theorem})
that $\Aa(b)$ has a connected subset $\Ff(b)$ containing annuloids $M$ of every necksize.
Thus by Theorem~\ref{cutoffs-theorem}, $\Ff(b)$ contains both capped and uncapped annuloids.
Since the set of capped $M\in \Ff(b)$ is relatively open, it cannot be closed.  Thus there are capped $M_n\in \Ff(b)$
that converge to an uncapped $M$ in $\Ff$.  Since the $M_n$ are capped,
\[
  B(M_n)=b.
\]

Let $p_n=(x_n,0,\unup(x_n))$.  Recall (see Theorem~\ref{R-lambda-theorem}) 
 that $M_n\cap \{x\ge x(M_n)+\pi\}$ is the graph of a function
\[
  u_n: [x(M_n)+\pi,\infty)\times (-b,b) \to \RR.
\]
The gradient bound~\eqref{R-lambda-slope-item} in Theorem~\ref{R-lambda-theorem}) implies that $u_n(x_n+x,y) - u_n(x_n,0)$ converges smoothly (perhaps after passing to a subsequence)
to a smooth translator
\[
  u: \RR\times (-b,b)\to \RR.
\]
Since $\pdf{}xu_n(x_n,0)=0$, we see that $\pdf{}x u(0,0)=0$ and thus that $u$ is a $\Delta$-wing.
(We are assuming that $b>\pi/2$, so $u$ cannot be an untilted grim reaper surface.)
\end{proof}

\begin{corollary}\label{uncapped-but-squeezed}
If $b>\pi/2$, then there exist an uncapped annuloid $M\in \Aa(b)$ for which $B(M)=b$.
\end{corollary}

\begin{proof}
Let $M\in \Aa(b)$ be as in Theorem~\ref{uncapping-theorem}.  Since $B(M_n)=b$, we see that $B(M)=b$
(by Theorem~\ref{b-B-x-continuity-theorem}).
\end{proof}

%%%%%%%%%%%%%%%%%%%%%%%%%

\section{A graphical property of uncapped annuloids and prongs}\label{final-graphicality-section}

\newcommand{\tM}{\tilde M}

Here we show that if $M$ is an uncapped annuloid in $\Aa$, then $M\cap\{x>0\}$
is a sideways graph $x=x(y,z)$.  It follows (Theorem~\ref{prong-graph-theorem}) that if $M$ is a prong, then $M$ is also a sideways graph.

For $c\in\RR$, let 
\begin{align*}
   I_c &= (c-\pi/2,c+\pi/2), \\
   W_c &= \RR\times I_c\times \RR.
\end{align*}
Let
$
  h_c:\RR\times I_c \to \RR
$
be the untilted grim reaper surface with $h_c(x,c)\equiv 0$, and let
\begin{align*}
&H_c: W_c\to\RR, \\
&H_c(x,y,z)= z - h_c(x,y).
\end{align*}

For $M\in \Rr$, we will use the formula (see Theorem~\ref{morse-rado-theorem})
\begin{equation}\label{S-T}
\mathsf{N}(H_c|M) \le |S| - |T| - \chi(M\cap W_c).
\end{equation}
where $S$ is the set of local minima of $H_c|\partial M$ that are also local minima of $H_c|M$,
and where $T$ is the set of local maxima of $H_c|\partial M$ that are {\bf not} local maxima of $H_c|M$.

By symmetry,
\[
  \mathsf{N}(H_c|M) = \mathsf{N}(H_c|M\cap \{x=0\}) + 2 \mathsf{N}(H_c|M\cap \{x>0\}).
\]

Thus we can rewrite~\eqref{S-T} as
\begin{equation}\label{double}
\chi(M\cap W_c) + \mathsf{N}(H_c|M\cap \{x=0\})  +   2 \mathsf{N}(H_c|M\cap \{x>0\}) 
\le
|S| - |T|.
\end{equation}

Recall that
\[
  y(M): = \inf \{|y|: (0,y,z)\in M\}.
\]

\begin{lemma}\label{easy-lemma}
Let $M\in \Rr$ and suppose $c\ge 0$.
If $-y(M)\in I_c$, then $H_c|M$ has a critical point in $M\cap \{x=0\}\cap \{y<0\}$.
\end{lemma}

\begin{proof}
Let 
\[
  \Gamma= M\cap\{x=0\}\cap \{ (c-\pi/2) < y < 0\}.
\]
Note that $\Gamma$ is a compact curve with both endpoints on the line $\{(0,c-\pi/2)\}\times\RR$.
Now $H|\Gamma$ is a continuous map to $(-\infty,\infty]$, and $H=\infty$ on each of the endpoints
of $\Gamma$.  Thus $H|\Gamma$ has an interior minimum.  By symmetry, that point is a critical point of $H|M$.
\end{proof}

\begin{corollary}\label{double-corollary}
Suppose $M\in \Rr$ and $M\cap W_c$ is nonempty.
Then
\begin{equation}\label{double-corollary-inequality}
  \chi(M\cap W_c) + \mathsf{N}(H_c|M\cap \{x=0\}) \ge 1,
\end{equation}
and thus (by~\eqref{double})
\begin{equation}\label{special-counting}
 1 + 2 \mathsf{N}(H_c|M\cap\{x>0\}) \le |S| - |T|.
\end{equation}
\end{corollary}

\begin{proof}
By symmetry, we can assume that $c\ge 0$.
If $-y(M)\notin I_c$, then $M\cap W_c$ does not contain a curve that winds around $Z$,
so $M\cap W_c$ is a union of disks.  Thus (in this case) $\chi(M\cap W_c)\ge 1$,
 so the inequality~\eqref{double-corollary-inequality} holds.
If $-y(M)\in I_c$, then (by Lemma~\ref{easy-lemma}) $\mathsf{N}(H_c|M\cap\{x=0\})\ge 1$, so the inequality~\eqref{double-corollary-inequality} holds.
\end{proof}

In the rest of this section, $M$ will be in $\Aa$ and $M_n\in \Rr$ will be a sequence such that suitable vertical
translates of the $M_n$ converge to $M$.   We write $a_n$, $b_n$, $A_n$, and $B_n$ for $a(M_n)$, $b(M_n)$,
$A(M_n)$, and $B(M_n)$.  We let $L_n$ be the edge
\[
   [-A_n, A_n] \times \{B_n\} \times \{0\}
\]
and $\ell_n$ be the edge
\[
 [-a_n,a_n] \times \{b_n\} \times \{0\}.
\]

\begin{proposition}\label{straddling-proposition}
Suppose $M\in \Aa$.
If $B\in \overline{I_c}$, then
\[
   \mathsf{N}(M\cap\{x>0\}) = 0.
\]
\end{proposition}

Note that this proposition is true for capped and uncapped annuloids.

\begin{proof}
The condition $B\in\overline{I_c}$ is equivalent to $|c-B|\le \pi/2$.
By semicontinuity, it suffices to prove it for a dense set of $c$ with $|c-B|<\pi/2$.
Thus we can assume that $B\in I_c$ and that $c$ is not equal to $b$ or to $B$.
By~\eqref{special-counting},
\begin{equation}\label{special-counting-n}
1 + 2\mathsf{N}(H_c|M_n\cap \{x>0\}) \le |S_n| - |T_n|.
\end{equation}

\noindent
{\bf Claim}: $|S_n| - |T_n| \le 2$ for all sufficiently large $n$.

Assuming the claim, we see that
\[
  2 \mathsf{N}(H_c|M_n\cap \{x>0\}) \le 1
\]
for large $n$.  Thus for such $n$, 
\[
  \mathsf{N}(H_c|M_n\cap \{x>0\}) = 0.
\]
Now letting $n\to\infty$ gives $\mathsf{N}(M\cap\{x>0\}=0$.
Thus it remain only to prove the claim.

\begin{proof}[Proof of Claim]
We divide the proof of the claim into the three cases $c>B$, $b< c < B$, and $c<b$.

{\bf Case 1}: $c>B$.  Then $c>B_n$ for large $n$, and thus $S_n$ consists of $L_n$  and (if $b_n \in I_c$)
 also of $\ell_n$.
Thus $|S_n|\le |S_n| - |T_n| \le 2$.

{\bf Case 2}: $b< c< B$.  Then $b_n < c < B_n$ for large $n$.
In this case, the local minima of $H_c|\partial M_n$ are the two points $(\pm A_n,c,0)$ and also (if $b_n\in I_c$) the ``point'' $\ell_n$.
So $|S_n|\le 3$.

Now $L_n$ is a local maximum of $H_c|\partial M_n$.   But for large $n$, $L_n$ is not a local maximum of 
$H_c| M_n$.

(To see that, note that $\Tan(M_n,(A_n,0,0))$ is nearly vertical.  It follows (for large $n$) that $H_c$ restricted to the
curve $M_n\cap \{x=0\}$ has a strict local minimum at the point $(A_n,0,0)$.

Thus $|T_n|\ge 1$, so
\[
  |S_n| - |T_n| \le 3-1 \le 2.
\]

{\bf Case 3}: $0\le c < b$.
Then $0<c< b_n$ for large $n$.

In this case, $H_c|\partial M_n$ has exactly four local minima, namely
   $(\pm a_n,c,0)$ and $(\pm A_n,c,0)$.  Thus $|S_n|\le 4$.

Also, $H_c| \partial M_n$ has two local maxima, $\ell_n$ and $L_n$.
As in the case 2, when $n$ is large, those points are not local maxima of $H_c|M_n$. 
Thus $|T_n|\ge 2$.
Thus (in case 3)
\[
 |S_n| - |T_n| \le 4 - 2 = 2.
\]
This completes the proof of the claim, and therefore the proof of Proposition~\ref{straddling-proposition}.
\end{proof}
\let\qed\relax
\end{proof}

\begin{theorem}\label{fundamental-graph-theorem}
Suppose $M\in \Aa$.
Suppose also that
\begin{enumerate}
\item\label{graph-item-1} $B=\pi/2$, or 
\item\label{graph-item-2} $B>\pi/2$ and $M$ is uncapped.
\end{enumerate}
Then $\mathsf{N}(H_c|M\cap \{x>0\})=0$.
\end{theorem}

\begin{proof}
By Proposition~\ref{straddling-proposition}, it suffices to prove it when $\overline{I_c}$ is contained in $(-B,B)$.
In particular, this means we are in Case~\eqref{graph-item-2} of the theorem: $B>\pi/2$ and $M$ is uncapped.

Let $\alpha> x(M)+ \pi$.   
Then $\alpha> x(M_n) + \pi$ for all large $n$.  We may suppose that $\alpha>x(M_n)+\pi$
for all $n$.

Then $\Mup\cap \{x\ge \alpha\}$ is the graph of a function
\[
   u : [\alpha,\infty) \times (-B,B) \to \RR.
\]
Likewise, $\Mup_n\cap \{x\ge \alpha\}$ is the graph of
\[
   u_n: [\alpha,A_n]\times (-B_n,B_n) \to \RR.
\]

Recall that 
\[
   y \mapsto \pdf{}xu(x,y)
\]
converges to $s(B)$ uniformly on compact subsets of $(-B,B)$.

Thus there is a $\Lambda$ such that 
\[
   \pdf{}xu > s(B)/2 > 0
\]
on
\[
  \overline{I_c} \times [\Lambda,\infty).
\]
Now let $\lambda>\Lambda$.
Then for all sufficiently large $n$,
\begin{equation}\label{uppity}
 \text{$\pdf{}xu_n > 0$ on $\{\lambda\} \times \overline{I_c}$.}
\end{equation}

Let 
\[
  E_n^+ := \Mup_n\cap \{x>\lambda\}
\]
and let $E_n^-$ be the image of $E_n$ under reflection in the plane $\{x=0\}$.
Let
\[
  \tM_n:= M_n\setminus (E_n^+\cap E_n^-),
\]

Now we wish to apply the formula (see Theorem~\ref{morse-rado-theorem})
\begin{align*}
\chi(\tM_n\cap W_c)  + \mathsf{N}(H_c|\tM_n) 
&\le |S_n| 
\end{align*}
(where $S_n$ is the set of local minima of $H_c|\partial M_n$ that are also local
minima of $H_c|M_n$), which we can rewrite as
\[
\chi(\tM_n\cap W_c) + \mathsf{N}(\tM_n\cap \{x=0\}) + 2 \mathsf{N}( \tM_n\cap\{x>0\}) \le |S_n|.
\]
By Corollary~\ref{double-corollary}, the sum of the first two terms on the left is at least $1$. Thus
\begin{equation}\label{charles}
1 + 2 \mathsf{N}(\tM_n\cap \{x>0\}) \le |S_n|.
\end{equation}

The boundary of $\tM_n$ contains  two curved portions, namely
\[
\Gamma_n^+:= \Mup_n\cap \{x=  \lambda\} 
\]
and its mirror image $\Gamma_n^-$ under reflection it the plane $\{x=0\}$.
These portions $\Gamma_n^\pm$ are the portions of $\partial \tM_n$ in $\{z>0\}$.
Because 
\[
  \pdf{}x u_n  > 0
\]
on $\overline{I_n}\times \{\lambda\}$, we see that $H_c|\tM_n$ has {\em no} local minima on $\Gamma_n^+$.
By symmetry, there are none on $\Gamma_n^-$.   Thus
\begin{equation}\label{curved-ok}
  |S_n\cap\{z>0\}| = 0.
\end{equation}

If $c<b_n$, then $S_n\cap\{z=0\}$ has just the $2$ points $(\pm a_n, c, 0)$.
If $b_n\le c$, then $S_n\cap \{z=0\}$
 has either no points (if $b_n\notin I_c$) or one point (namely $\ell_n$) if $b_n\in I_c$.
Thus
\begin{equation}\label{zero-level}
  |S_n\cap\{z=0\}| \le 2.
\end{equation}
Thus in either case ($c<b_n$ or $c\ge b_n$), we have~\eqref{zero-level},
so by~\eqref{curved-ok}, we see that $|S_n|\le 2$.

Thus by~\eqref{charles},
\[
1 +   2 \mathsf{N}(H_c | \tM_n\cap \{x>0\}) \le 2.
\]
which implies that 
\[
  \mathsf{N}(H_c|\tM_n\cap \{x>0\}) = 0
\]
for all sufficiently large $n$.  Thus
\begin{align*}
\mathsf{N}(H_c|M \cap \{0 < x < \lambda\})
&\le
\liminf \mathsf{N}(H_c| M_n\cap \{0< x < \lambda \} \\
&=
\liminf \mathsf{N}(H_c |\tM_n \cap \{0<x\}) \\
&= 0.
\end{align*}
This holds for each $\lambda\ge \Lambda$.
Hence $\mathsf{N}(H_c|M)=0$.
\end{proof}

\begin{theorem}\label{main-graph-theorem}
Suppose $M\in \Aa$.
\begin{enumerate}
\item if $b=\pi/2$, or
\item if $b>\pi/2$ and $M$ is uncapped, 
\end{enumerate}
then $M\cap\{x>0\}$ is the graph of a function $x=x(y,z)$ over a domain in the $yz$-plane.
\end{theorem}

\begin{proof}
We claim that $\ee_1\cdot \nu$ never vanishes in $M\cap\{x>0\}$.
For suppose to the contrary that $\ee_1\cdot\nu=0$ at point $p$ with $x(p)>0$.  If $\Tan(M,p)$ were vertical,
then $\vv:=\nu(M,p)$ would be $\pm \ee_2$, and then $(-x(p),y(p),z(p))$ would be a second point
with $\nu=\vv$, contrary to Corollary~\ref{x-critical-corollary}.
Thus $\Tan(M,p)$ is not vertical, so
 there is a strip $\RR\times (c-\pi/2,c+\pi/2)$ and a grim reaper surface over that strip that is tangent
to $M$ at $p$.
  But then $\mathsf{N}(H_c|M\cap\{x>0\})\ge 1$, contrary to Theorem~\ref{fundamental-graph-theorem}.

Since $M\cap\{x>0\}$ is connected and since $\nu(x(M),0,0)\cdot\ee_1=-1<0$, it follows that $\nu\cdot\ee_1<0$
at all points of $M\cap\{x>0\}$.

Since $M$ is connected and properly embedded in $\RR^3$, $M$ divides $\RR^3$ into two components.
Let $K$ be the component such that $\nu$ is the unit normal to $M$ that points out of $K$.
Let $L$ be a line parallel to the $x$-axis.  Noe $L$ cannot $M\cap\{x>0\}$ in more than one point, since
if it did, the sign of $\nu\cdot\ee_1$ would alternate from one point to the next.

Thus $L$ intersects $M\cap\{x>0\}$ in at most one point, and the intersection is transverse.
\end{proof}

\begin{theorem}\label{prong-graph-theorem}
Suppose that $M$ is a prong.  Then $M$ is the graph of a function $x=x(y,z)$ over
an open subset of the $yz$-plane.
\end{theorem}

\begin{proof}
Recall that $M$ is a limit of $M_n':=M_n - (x(M_n),0,0)$ where $M_n\in \Aa(b)$ and $x(M_n)\to \infty$.
Since $B_n\to b+\pi$, we see that the $M_n$ are uncapped for large $n$.
For such $n$, $\nu\cdot \ee_1<0$ everywhere in $M_n'\cap\{x>-x(M_n)\}$.
Thus $\nu\cdot\ee_1\le 0$ everywhere on $M$.  Now $M$ is connected, and $\nu\cdot\ee_1$ is a Jacobi field,
so if it vanished anywhere, it would vanish everywhere by the strong maximum principle.  Since
$\nu\cdot\ee_1=-1$ at the origin, we see that $\nu\cdot\ee_1<0$ everywhere on $M$.

Let $L$ be a line parallel to the $x$-axis.  If $L$ intersected $M$ in more than one point, then for large $n$,
it would intersect $M_n\cap\{x>0\}$ in more than one point, contrary to Theorem~\ref{main-graph-theorem}.
\end{proof}

%%%%%%
%%%%%%
%%%%%%
%%%%%%

\appendix

\section{Some Useful Barriers}

\begin{theorem}\label{short-barrier-theorem}
For every $a>0$ and $b>0$, there is a translator 
\[
    u=u_{a,b}: [0,a]\times [-b,b] \to [0,\infty]
\]
such that 
\begin{align*}
&u(0,\cdot) = u(\cdot, \pm b) = 0, \quad\text{and}\\
&u(a,\cdot) = \infty.
\end{align*}
Suppose $M$ is a translator in $[0,a]\times [-b,b]\times \RR$.
\begin{enumerate}[\upshape(i)]
\item\label{under}
If $\partial M$ lies in $\{z\le u(x,y)\}$, then $M$ lies in $\{z\le u(x,y)\}$.
\item\label{over}
If $\partial M$ lies in $\{z\ge u(x,y)\}$, then $M$ lies in $\{z \ge u(x,y)\}$.
\end{enumerate}
Furthermore, if $M$ is any translator in $[0,a]\times [-b,b]\times\RR$ with $\partial M=\partial (\graph(u))$,
then $M=\graph(u)$.
\end{theorem}

\begin{proof}
Existence of $u$ can be proved by solving the translator equation on $[0,a]\times [-b,b]$ with boundary values
$n$ on the side $\{a\}\times (-b,b)$ and $0$ on the other three sides, and then letting $n\to\infty$; 
see \cite{Gama}*{Theorem 9}.  
\begin{figure}[htbp]
\begin{center}
\includegraphics[width=.33\textwidth]{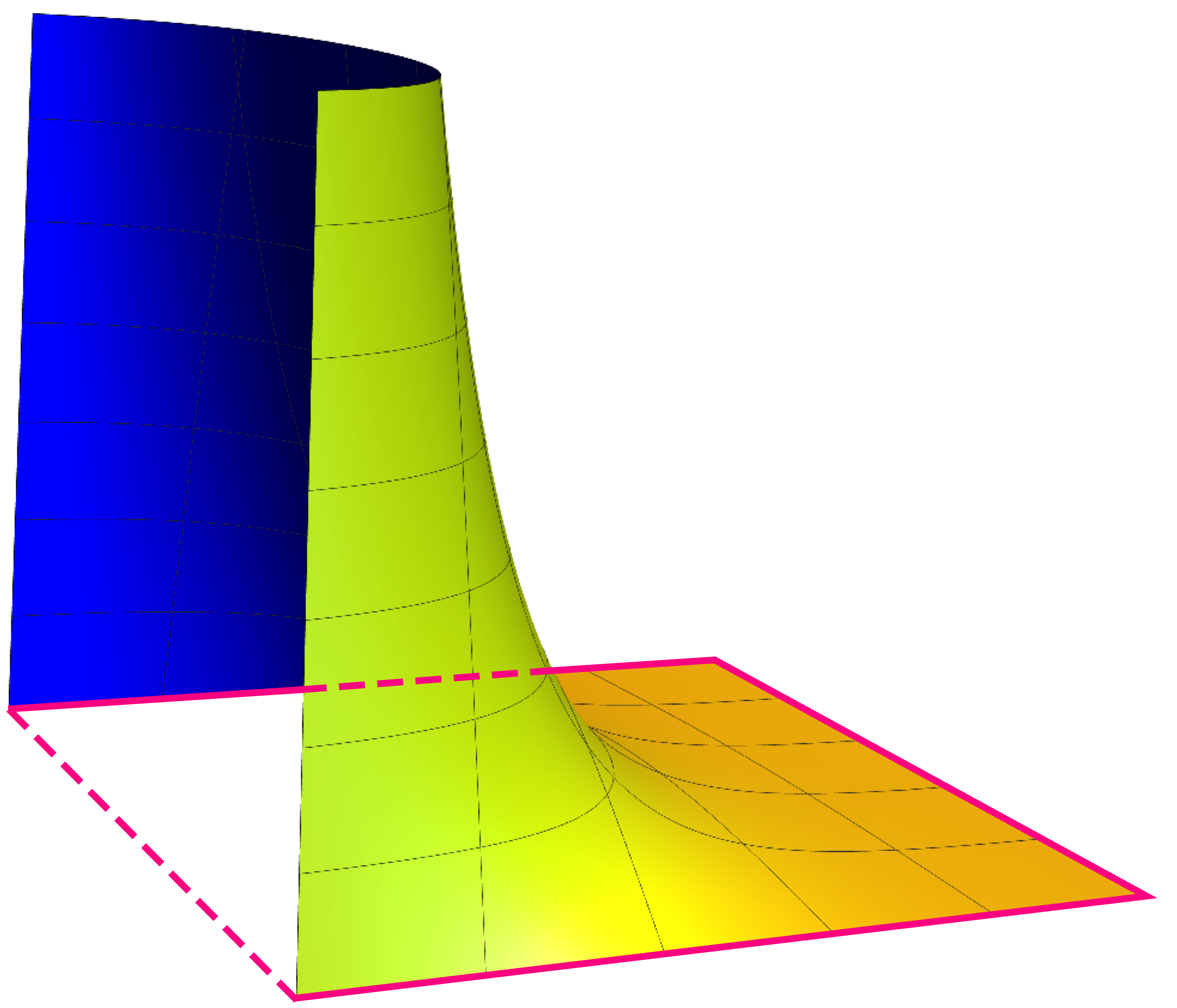}
\caption{\small The barrier $u$.}
\label{fig:barrier}
\end{center}
\end{figure}

Suppose that $\partial M$ is contained in $\{z\le u(x,y)\}$. 
We claim that
\begin{equation}\label{vertical-end}
\text{If $(x_i,y_i,z_i)\in M$ and $z_i\to \infty$, then $x_i\to a$.}
\end{equation}
This can be proved using rotationally symmetric translating annuli as barriers.
Alternatively, let 
\[
   \tilde x = \lim_{\zeta\to\infty} \inf_{M\cap\{z\ge \zeta\}} x(\cdot).
\]
Let $(x_i,y_i,z_i)\in M$ with $z_i\to\infty$ and $x_i\to \tilde x$.
Then, after passing to a subsequence, $(x_i,y_i,0)$ converges to a point $(\hat x, \hat y, 0)$
and the sets $M-(0,0,z_i)$ converges (as sets) to a limit set $M'$. 
(In the language of~\cite{white16}, $M'\cap\{x<a\}$ is a ``$(2,0)$-set'' with respect to the translator metric.)
 Note that $x(\cdot)|M'$ attains its minimum
at the point $(\hat x,\hat y, 0)$.  If $\hat x< a$, then by the strong maximum principle in~\cite{white16},
 $M'$ would contain the plane $\{x=\hat x\}$, which is impossible since $M'$ is contained in 
  $[0,a]\times [-b,b]\times \RR$.  This completes the proof of~\eqref{vertical-end}.
  
Suppose, contrary to~\eqref{under}, that $M$ contains points with $z>u(x,y)$.
By~\eqref{vertical-end}, there would be a largest $t>0$ such that $M+(t,0,-t)$ intersects $\graph(f)$.
At  the point of contact, the maximum principle would be violated.  Thus $M$ lies in $\{z\le u(x,y)\}$.
This completes the proof of~\eqref{under}.

Now suppose, contrary to~\eqref{over}, that $\partial M$ is contained in $\{z\ge u(x,y)\}$
but that $M$ contains points with $z<u(x,y)$.
Then there would be a largest $t>0$ such that $M + (-t,0,t)$ intersects $\graph(u)$.
At the point of contact, the maximum principle would be violated. 
Thus $M$ lies in $\{z\ge u(x,y)\}$.  

The last statement of the theorem follows immediately from~\eqref{over} and~\eqref{under}.
\end{proof}

\begin{corollary}
For each $(x,y)\in [0,a)\times [-b,b]$, 
\[
   u_{a,b}(x,y)=u_{a,b}(x,-y),
\]
and $u(x,y)$ is a decreasing function of $|y|$.  In particular, 
\[
  u_{a,b}(x,y)\le u_{a,b}(x,0)
\]
on $[0,a)\times [-b,b]$.
\end{corollary}

\begin{proof}
Let $M$ be the graph of $u_{a,b}(x,-y)$.  By the last statement of Theorem~\ref{short-barrier-theorem}, $M=\graph(u_{a,b})$
and thus $u(x,-y)\equiv u(x,y)$.

Let $0\le y_1 < y_2\le b$ and let $\hat y=(y_1+y_2)/2$.  Let $M$
be the image of  $\graph(u_{a,b})| [0,a]\times [\hat y,b]$ under reflection in the plane $\{y=\hat y\}$.
Then $\partial M$ lies in $\{z\le u_{a,b}(x,y)\}$, so, by Theorem~\ref{short-barrier-theorem},
 $M$ lies in $\{z\le u_{a,b}\}$.   Thus 
\[
    u_{a,b}(x,y_2) \le u_{a,b}(x,y_1).
\]
Hence $u_{a,b}(x,y)$ is a decreasing function of $|y|$ for $y\in [0,b]$.  
\end{proof}

\begin{corollary}\label{half-slab-corollary}
Let $M$ be a  translator in a slab $\{|y|\le b\}$.  If $\inf_M x(\cdot)> -\infty$, then
\[
  \inf_Mx(\cdot) = \inf_{\partial M} x(\cdot).
\]
Likewise, if $\sup_M x(\cdot)<\infty$, then $\sup_M x(\cdot)=\sup_{\partial M} x(\cdot)$.

In particular, if $M$ is a complete, nonempty translator in a slab $\{|y|\le b\}$, 
then $M\cap\{x=t\}$ is nonempty for every $t$.
\end{corollary}

\begin{proof}
Suppose the statement about infima is false.  By translating, we can assume that
\begin{equation}\label{a-in-between}
   0 < \inf_M x(\cdot) < a < \inf_{\partial M} x(\cdot)  
\end{equation}
for some $0<a<\infty$.   By Theorem~\ref{short-barrier-theorem}, $M\cap \{x\le a\}$ lies 
below the graph of $u_{a,b}$.  The same is true for any vertical translate
of $M$.  Thus $M\cap\{x<a\}$ is empty, contrary to~\eqref{a-in-between}.

The statement about suprema follows by reflection.
\end{proof}

\begin{theorem}\label{lipschitz-theorem}
Let
\[
   s_b = \inf_{a>0} \pdf{}x u_{a,b}(0,0).
\]
Then $s_b<\infty$.  

Suppose $M$ is a translator in $\{|y|\le b\}$, and let
\[
  \psi_M(t):= \sup_{M\cap \{x=t\}} z(\cdot),
\]
If $\partial M\subset \{x=x_0\}$, then
\begin{equation}\label{lipschitz}
   \psi_M(x_0+h) \le \psi(x_0) + s_b |h|   \tag{*}
\end{equation}
for all $h$.
In particular, if $M$ has no boundary, then~\eqref{lipschitz} holds for all $x_0$ and $h$.
\end{theorem}

\begin{proof}
Note that $s_{a,b}:=\pdf{}x u_{a,b}(0,0)<\infty$ by the strong maximum principle.  Thus $s_b<\infty$.

By symmetry, it suffices to prove~\eqref{lipschitz} for $h>0$.
First we claim that
\begin{equation}\label{for-h-small}
  \psi(x'+h)\le \psi(x') + u_{a,b}(h,0)  \quad\text{if $x'\ge x_0$ and $a>h\ge 0$}.
\end{equation}
We may assume that $\psi(x')<\infty$, as otherwise~\eqref{for-h-small} is trivially true.  Let $\psi_M(x')< \zeta < \infty$,
and let $M' = (M - (x',0,\zeta)) \cap \{0 \le x\le a\}$.
Then $\partial M'$ lies in $\{z\le u_{a,b}(x,y)\}$, so
  $M'$ lies in $\{z\le u_{a,b}(x,y)\}$ by Theorem~\ref{short-barrier-theorem}.
Hence
\[
  \psi_{M'}(h) \le u_{a,b}(h,0),
\]
which is equivalent (by translation) to
\[
  \psi_M(x'+h)  \le \zeta + u_{a,b}(h,0).
\]
This holds for all $\zeta> \psi_M(x')$, and therefore it holds for $\zeta=\psi_M(x')$.
Thus we have proved~\eqref{for-h-small}.

Now let $a>0$ and let $n$ be a positive integer such that $h/n < a$.
Then
\[
   \psi_M(x_0 + kh/n) \le \psi_M(x_0 + (k-1)h/n) + u_{a,b}(h/n,0)
\]
for all integers $k$, and, in particular, for $k=1,2,\dots, n$.
Thus
\begin{align*}
\psi_M(x_0 + h) 
&\le \psi_M(x_0) + n u_{a,b}(h/n,0) \\
&= \psi_M(x_0) + h (n/h) u_{a,b}(h/n,0).
\end{align*}
Letting $n\to\infty$ gives
\[
 \psi_M(x_0 + h) \le \psi_M(x_0) + h \, \pdf{}xu_{a,b}(0,0) = \psi_M(x_0) + h\,s_{a,b}.
\]
Taking the infimum over $b>0$ gives~\eqref{lipschitz}.
\end{proof}

\section{Translators in half-slabs}

In this section, we investigate the asymptotic behavior translators in a half-slab as $|x|\rightarrow\infty$.
Some of the ideas of this section are inspired  by~\cite{Chini}*{Theorem~10}.

\begin{definition}\label{Phi-definition}
If $M\subset \RR^3$, we let $\Phi(M)$ be the union of all subsequential limits of $M+(0,0,z)$ as $z\to\infty$.
Equivalently, $\Phi(M)$ is the set of $(x,y,z)$ such that there exists $(x_i,y_i,z_i)\in M$ with $(x_i,y_i)\to (x,y)$
and $z_i\to -\infty$.
\end{definition}

Note that if $K$ is a compact subset of $\RR^2$ with $K\times\RR$ disjoint from $\Phi(M)$, then
\[
   \inf_{M\cap (K\times \RR)} z(\cdot) >  -\infty.
\]

Let $b\ge \pi/2$, 
let $I=(c-b,c+b)$ and let $w: \RR\times I\to\RR$ be a complete translating graph.
Define $f: \RR^2\to\RR$ by
\[
f(x,y)
=
\begin{cases}
w(x,y) &\text{if $y\in I$}, \\
-\infty &\text{if $y\notin I$},
\end{cases}
\]
and define $F: \RR^2\times (-\infty,\infty] \to (-\infty,\infty]$ by
\[
F(x,y,z) = z - f(x,y).
\]
Note that $f$ and $F$ are continuous.

\begin{theorem}\label{chini-theorem-1}
Suppose $M\subset \RR^3$ is a translator such that
\[
   \inf_M y(\cdot) > - \infty,
\]
and such that 
\begin{equation}\label{a-ends}
\text{$(\partial M)\cap \{|x|\ge a\}$ is contained in $\{y\ge c\}$ for some $a\ge 0$.}
\end{equation}
Suppose also that
\begin{equation}\label{downer}
  \text{$\Phi(M)$ is contained in $\{y>c+b\}$.}
\end{equation}
Then
\begin{equation}\label{equal-infima}
  \inf_M F = \inf_{\partial M} F.
\end{equation}
\end{theorem}

\begin{lemma}\label{pre-chini-lemma}
Suppose $C$ is a closed set in $\RR^3$ such that $\Phi(C)$ is disjoint from the 
 the closed slab $\{y \in \overline{I}\}$.
Suppose also that $C\cap \{|y|<b\}$ is contained in $K\times\RR$ for some compact set $K$, 
 and that
\[
   \inf_C F < \lambda < \infty.
\]
Then $C'\cap\{F\le \lambda\}$ is compact, and therefore $F|C$ attains its minimum: there is a point $q\in C$
such that $F(q)=\inf_C F$.
\end{lemma}

\begin{proof}[Proof of Lemma~\ref{pre-chini-lemma}]
Note that $\{y\in I\}= \{F<\infty\}$, so
\begin{equation}\label{KxR}
   C' \subset C\cap \{y\in I\} \subset K\times\RR.
\end{equation}
Note also that 
\begin{equation}\label{emptiness}
  \Phi(C') =\emptyset
\end{equation}
because $\Phi(C)$ is contained in $\{y\in \overline{I}\}$ and also in $\Phi(C)$.
Thus
\[    
   \zeta:= \inf_{C'} z(\cdot) > -\infty.
\]
by~\eqref{KxR} and~\eqref{emptiness}.
 Note that $f | K$ is bounded above, so $\eta:=\sup_K f <\infty$.
For $(x,y,z)\in C'$, we have (by definition of $F$)
\begin{align*}
z
&=
F(x,y,z) + f(x,y) \\
&\le
\lambda + \eta.
\end{align*}
Thus $C'$ is contained in $K\times [\zeta,\lambda+\eta]$, so $C'$ is compact.
\end{proof}

\begin{proof}[Proof of Theorem~\ref{chini-theorem-1}]
It suffices to prove the theorem for $c=0$.
Let $\beta$ be a negative number such that $\beta \le \inf_M y(\cdot)$.

By replacing $M$ by $M\cap\{y\le b\}$,
we can assume
 that
\begin{equation}\label{banded}
    \text{$M$ lies in the slab $\{\beta \le y \le b\}$.}
\end{equation}
Note that~\eqref{banded} and~\eqref{downer} imply that $M+(0,0,z)$ converges to the empty set as $z\to\infty$.
Equivalently,
\begin{equation}\label{bounded-below}
 \inf_{M\cap (K\times \RR)} z(\cdot) >  -\infty \quad\text{for every compact $K\subset \RR^2$.}
\end{equation}

To prove the theorem, it suffices to prove that
\begin{equation}\label{one-way}
\inf_M F \ge \inf_{\partial M} F,
\end{equation}
as the reverse inequality is trivially true.

By horizontal translation, we may assume that if $u$ is a $\Delta$-wing, then it is centered at the origin:
$Du(0,0)=0$.

We may choose the $a$ in~\eqref{a-ends} so that $a> b$.

\begin{claim}\label{M'-claim}
 Let $M' = M\cap \{x\ge a\}$.   Then 
\[
    \inf_{M'} F = \inf_{\partial M'}F.
\]
\end{claim}

If $u$ is a grim reaper surface with $\pdf{u}x\le 0$ or if $u$ is a $\Delta$-wing (in which case $\pdf{u}x<0$ on $(0,\infty)\times (-b,b)$),
  let $L$ be the vertical line through $(0,b,0)$. 
If $u$ is a grim reaper surface with $\pdf{u}x> 0$,  let $L=Z$.   
Let $R(\theta)$ denote counterclockwise rotation about $L$ through angle $\theta$.  We also let $R(\theta)$ denote
the corresponding rotation in the $xy$-plane.   For each $p\in [a,\infty)\times [0,b]$, note that 
\[
   \theta\in [0,\pi/2) \mapsto f(R(\theta)p)
\]
is a decreasing function.  Thus, for each $p=(x,y,z)\in [a,\infty)\times[0,b]\times\RR$ 
(and thus, in particular, for each $p\in \partial M'$),
\[
  \theta \in [0,\pi/2)\mapsto F(R(\theta)p) = z - f(R(\theta)(x,y))
\]
is an increasing function.
 Consequently,
\begin{equation}\label{rotated-inf}
   \inf_{R(\theta)\partial M'} F \ge \inf_{\partial M'} F.
\end{equation}

For $\theta\in (0,\pi/2)$), we claim that 
\begin{equation}\label{theta-version}
  \inf_{R(\theta) M'} F = \inf_{R(\theta)\partial M'} F.
\end{equation}
To prove~\eqref{theta-version}, we may assume that the left side is $<\infty$, as otherwise~\eqref{theta-version} is trivially true.
Now
\[
   R(\theta)M' \cap \{|y|< b\}   \subset T\times \RR,
\]
where $T=T_\theta$ is the compact triangular region 
\[
  R(\theta) ( [0,\infty)\times [\beta,B] ) \cap \{y\le b\}.
\]
By Lemma~\ref{pre-chini-lemma}, the infimum of $F$ on $R(\theta)M'$ is attained at a point $q$. By the strong maximum principle, 
   $q$ is in the boundary $\partial R(\theta)M' = R(\theta)\partial M'$.
Hence~\eqref{theta-version} holds.

Now let $p\in M'$.  By~\eqref{theta-version} and~\eqref{rotated-inf},
\begin{align*}
F(R(\theta)p)
&\ge
\inf_{R(\theta)M'} F  \\
&=
\inf_{R(\theta)\partial M'} F \\
&\ge
\inf_{\partial M'}F.
\end{align*}
Thus
\[
  F(R(\theta)p) \ge \inf_{\partial M'} F.
\]
Letting $\theta\to 0$ gives
\[
  F(p)\ge \inf_{\partial M'}F.
\]
Taking the infimum over $p\in M'$ gives~\eqref{one-way}.  Thus we have
proved Claim~\ref{M'-claim}.

Using Claim~\ref{M'-claim}, we now prove that 
\begin{equation}\label{repeated}
\inf_M F= \inf_{\partial M} F.
\end{equation}
We may assume that $M$ contains a point $p_0$ with $F(p_0)<\infty$, as otherwise~\eqref{repeated} is trivially true.

Let
\[
  \alpha := \inf_{M\cap \{|x|\le a\}}F = \inf_{M\cap (K \times \RR)} F,
\]
where $K$ is the rectangle $[-a,a]\times[-b,b]$.
Since $p_0\in M\cap\{|x|\le a\}$,
\[
  \alpha \le F(p_0) < \infty.
\] 
By Lemma~\ref{pre-chini-lemma}, the infimum is attained: $M\cap \{|x|\le a\}$ contains a point $p'$ such that
\[
   F(p') = \inf_{M\cap \{|x|\le a\}} F.
\]

Suppose first that
\begin{equation}\label{nested-infs}
  \inf_M F = \inf_{M\cap\{|x|\le a\}} F.
\end{equation}
Then $F|M$ attains its minimum at $p'$, and therefore $p'\in \partial M$ by the strong maximum principle.
Thus we have proved~\eqref{repeated} in case~\eqref{nested-infs} holds.

Now suppose that~\eqref{nested-infs} does not hold.
Then
\begin{equation}\label{less-than-alpha}
  \inf_M F < \inf_{M\cap\{|x|\le a\}} F = \alpha.
\end{equation}
It follows that $\inf_M F$ is equal to one or both of $\inf_{M\cap\{x\ge a\}}F$ and $\inf_{M\cap\{x\le -a\}}F$.
By symmetry, it is enough to consider the case
\[
  \inf_M F = \inf_{M'} F.
\]
where $M'=M\cap \{x\ge a\}$.  By Claim~\ref{M'-claim}, 
\begin{align*}
\inf_{M'}F 
&= \inf_{\partial M'}F   \\
&\ge
\min \{   \inf_{M\cap \{x=a\}} F , \inf_{(\partial M)\cap \{x>a\}} F  \}             \\
&\ge
\min\{\alpha, \inf_{\partial M} F \}.
\end{align*}
Thus by~\eqref{less-than-alpha},
\[
 \alpha > \inf_M F \ge \min\{ \alpha, \inf_{\partial M} F \},
\]
which immediately implies that
\[
\inf_M F \ge \inf_{\partial M} F.
\]
The reverse inequality is trivially true.
\end{proof}

\begin{corollary}\label{first-chini-corollary}
In Theorem~\ref{chini-theorem-1},
\begin{enumerate}
\item\label{first-chini-1}
  If $M$ is complete or, more, generally, if $(\partial M)\cap \{y\in I\}$ is empty, 
then $M\cap \{y\in I\}$ is empty.
\item\label{first-chini-2}
 If $\sup_MF<\infty$  and if $\partial M\subset \{x=0\}$, then
  $F|M$ attains its minimum at a point $q\in (\partial M)\cap\{y\in I\}$, and thus
  \[
      \min_M F = F(q) > -\infty.
  \]
\item\label{first-chini-3}
 If $\partial M$ is contained in $\{x=0\}$,
 and if $J$ is a closed interval in $I$, then there is $\gamma=\gamma_J$ such that
\[
    z \ge s(b) x + \gamma \quad\text{for $(x,y,z)\in M\cap \{y\in J\}$.}
\]
\end{enumerate}
\end{corollary}
  
\begin{proof}
If $(\partial M)\cap\{y\in I\}$ is empty, then $\inf_{\partial M}F=\infty$,
so $\inf_M F = \infty$ (by Theorem~\ref{chini-theorem-1}), 
and thus $M\cap\{y\in I\}$ is empty.  

Now suppose that $\partial M$ is contained in $\{x=0\}$
and that $\inf_M F <\infty$.
Then
\begin{align*}
\infty 
&> \inf_M F  \\
&=\inf_{\partial M} F.
\end{align*}
Now 
\[
  (\partial M)\cap \{y\in I \} \subset K\times \RR,
\]
where $K$ is the compact set $\{0\}\times \overline{I}$.
By Lemma~\ref{pre-chini-lemma}, the infimum is attained at some point $q$ in $\partial M$.
Since $F(q)<\infty$, the point $q$ lies in the slab $\{y\in I\}$.

To prove Assertion~\eqref{first-chini-3}, let $w:\RR\times I \to \RR$ be the grim reaper surface with $\pdf{w}x\equiv s(b)$.
We may assume that $\inf_MF<\infty$, as otherwise the assertion is trivially true.  
Let $q$ be as in Assertion~\eqref{first-chini-2}.  Then
\[
-\infty 
< F(q) 
\le F(x,y,z) 
= z - f(x,y),
\]
so
\begin{align*}
z 
&\ge F(q) + f(x,y) \\
&= F(q) + f(0,y) + s(b) x \\
&\ge  F(q) + \min_{y\in J}f(0,y) + s(b)x.
\end{align*}
\end{proof}

%%%%%%%%%%%%%%%%%%%%%%%%%%%%%%%%%%%%%%%%%%%%%

\begin{corollary}\label{complete-chini-corollary}
Suppose that $M$ is a translator such that
\[
  \inf_M y(\cdot) > -\infty
\]
 and that $\Phi(M)$ is contained
in the halfspace $\{y\ge d\}$. 
If $M$ is complete, or, more generally, if $\partial M$ is contained in $\{y\ge d\}$,
then $M$ is contained in  $\{y\ge d\}$.
\end{corollary}

\begin{proof}
Let $I$ be any open interval of width $\ge \pi$ whose closure is in $(-\infty, d)$.
Then by Theorem~\ref{chini-theorem-1}
 (applied to any translating graph $w: \RR\times I \to \RR$), $M$ is disjoint from the slab $\{y\in I\}$.
\end{proof}

\begin{corollary}\label{chini-corollary}
Suppose that $M$ is a translator in $\{x\ge 0\}$ with $\partial M$ in $\{x=0\}$.
Suppose also that 
\[
  \inf_M y(\cdot) > - \infty
\]
and that $\Phi(M)$ is contained in the halfspace $\{y\ge d\}$.
\begin{enumerate}[\upshape(1)]
\item\label{chini-corollary-A1} If $(x_i,y_i,z_i) \in M$ with $x_i\to\infty$ and $y_i\to \hat y<d$, then $z_i/x_i\to \infty$.
\item\label{chini-corollary-A2}
   If $M$ is contained in the region $\{z \le mx + c\}$ for some $m$ and $c$, and if $d'<d$, then $M\cap \{y\le d'\}$
  is compact.
\end{enumerate}
\end{corollary}

\begin{proof}
To prove Assertion~\eqref{chini-corollary-A1}, 
note that for every $b\ge \pi/2$, there is an open interval $I$ of width $\ge 2b$
such that $\hat y\in I$ and such that $\overline{I}\subset (-\infty,d)$.
By Assertion~\eqref{first-chini-3} of Corollary~\ref{first-chini-corollary}, 
\[
  \limsup_{i\to\infty} z_i/x_i \ge s(b).
\]
Now let $b\to\infty$.

To prove Assertion~\eqref{chini-corollary-A2}, let $p_i=(x_i,y_i,z_i)$ be a sequence of  points in $M\cap \{y\le d'\}$.
Note that $y_i$ is a bounded sequence (since $\inf_M y(\cdot)> -\infty$).
If $x_i\to \infty$, then $z_i/x_i\to \infty$ by Assertion~\eqref{chini-corollary-A1}, which is impossible since $z_i\le mx_i+c$.
Thus $x_i$ is bounded.  Since $z_i\le mx_i+c$, $z_i$ is bounded above.
After passing to a subsequence, $(x_i,y_i)$ converges to point $(\hat x,\hat y)$ with $\hat y< d$.
Since $y(\cdot)\ge d$ on $\Phi(M)$, we see that $z_i$ is bounded below.  We have shown that every sequence
in $M\cap\{y\le d'\}$ is bounded.  Thus $M\cap\{y\le d'\}$ is compact.
\end{proof}

\begin{theorem}\label{chini-theorem-2}
Suppose $M$ is a translator in $\{x\ge 0\}$ such that
\begin{enumerate}[\upshape(1)]
\item $\partial M$ is contained in $\{x=0\}$.
\item $z(\cdot)$ is bounded above on $\partial M$.
\item $\sup_M |y(\cdot)| < \infty$.
\item Every subsequential limit as $z\to \infty$ of $M+(0,0,z)$
  is contained in the slab $\{|y|\le b\}$.
\end{enumerate}
Then
\begin{enumerate}[\upshape(i)]
\item\label{chini-compact} If $\beta>b$, then $M\cap\{|y|\ge \beta\}$ is compact.
\item\label{chini-thin}
If  $b< \pi/2$, then $M$ is contained in the plane $\{x=0\}$.
\end{enumerate}
\end{theorem}

\begin{proof}
Let $\beta>b$.  
By Theorem~\ref{lipschitz-theorem}, $M$ is contained in a region of the form $\{z\le mx + c\}$.
Thus, $M\cap \{y\le -\beta\}$ is compact by Corollary~\ref{chini-corollary}. 
Likewise, $M\cap \{y\ge \beta\}$ is compact.
Thus $M\cap\{|y|\ge \beta\}$ is compact.  
This completes the proof of Asssertion~\eqref{chini-compact}.

Now suppose that $b<\pi/2$ and let $b< \beta < \beta' < \pi/2$.
Since $M\cap \{|y|\ge \beta\}$ is compact, there is an $a\in [0,\infty)$ such that
\begin{equation}\label{thin-part}
   M\cap \{x\ge a\} \subset \{|y|<\beta\}.
\end{equation}
By \cite{scherk}, there is an $a'>a$ and a translator
\[
   u: [a,a' ] \times [-\beta', \beta'] \to \RR
\]
such that $u(a,\cdot)=u(a',\cdot)=\infty$ and $u(\cdot,-\beta')=u(\cdot,\beta')= -\infty$.

If $M\cap \{a< x < a'\}$ were nonempty, then $z-u(x,y)$ would attain a maximum on $M$, violating the strong maximum 
principle.

Thus $M\subset \{x\le a\} \cup \{x\ge a'\}$.

By Corollary~\ref{half-slab-corollary}, $M$ is contained in $\{x=0\}$.  
\end{proof}

\begin{corollary}\label{outermost-corollary}
Suppose $M$ is a complete translator in $\RR^3$ such that $\sup_M |y(\cdot)|< \infty$
and such that $M+(0,0,z)$ converges as $z\to \infty$ to planes
$\{y=b_i\}$, with $b_1\le b_2 \le \dots \le b_k$.
Then
\begin{align*}
b_1 &= \inf_M y(\cdot), \\
b_k &= \sup_M y(\cdot).
\end{align*}
\end{corollary}

\begin{proof}
It suffices to prove the equation for $b_k$.
Trivially $b_k\le \sup_M y(\cdot)$.  The reverse inequality holds by Corollary~\ref{complete-chini-corollary}.  
Hence~$b_k=\sup_M y(\cdot)$.  
\end{proof}

\begin{theorem}\label{wings-theorem}
Suppose $M$ is a finite-type translator in $\{x\ge 0\}$ such that
\begin{enumerate}[\upshape (1)]
\item $\partial M\subset \{x=0\}$.  
\item $B:=\sup_M y(\cdot) <\infty$. 
\item $M$ is invariant under reflection in the plane $\{y=0\}$. 
\item $M+(0,0,z)$ converges smoothly, as $z\to \infty$, to the halfplanes
$\{y=\pm b\}\cap \{x\ge 0\}$.
\item $M\cap\{y=0\}$ is graph of a function $z=\phi(x)$.
\end{enumerate}
Then the limit $L=\lim_{x\to\infty} \phi'(x)$ exists, and $L=\pm s(b)$.
\end{theorem}

\begin{proof}
Theorem~\ref{second-symmetry-theorem}, the limit $L$ exists,
and $M-(x,0,\phi(x))$ converges (as $x\to\infty$) to a grim reaper surface $G$ containing the 
 line $\{(x,0,Lx)\}$.
By Theorem~\ref{chini-theorem-2}\eqref{chini-compact}, $G$ is contained in the
slab $\{|y|\le b\}$.  Thus
\begin{equation}\label{slope-range}
 |L| \le s(b).
\end{equation}
If $b=\pi/2$, then $s(b)=0$, so $L=0$ by~\eqref{slope-range}.  Thus we may assume that $b>\pi/2$.

We prove the theorem by assuming that $L\ne -s(b)$ and then proving that $L=s(b)$.
Since $L\ne -s(b)$, 
\[
   L > -s(b)
\]
by~\eqref{slope-range}.  By choosing $\beta\in (\pi/2, b)$ close to $b$, we can ensure that
\begin{equation}\label{slopey}
   L > -s(\beta)> -s(b).
\end{equation}
By translating $M$ vertically, we can assume that
\begin{equation}\label{disjoint}
\text{$\{0\} \times [-\beta,\beta] \times \RR$ is disjoint from $\partial M$.}
\end{equation}
For $a\ge 0$, let $h_a: \RR\times (-\beta,\beta)\to \RR$ be the $\Delta$-wing
such that $h_a(0,0)=0$ and $Dh_a(a,0)=0$.  Let
\[
H_a(x,y,z)
=
\begin{cases} 
z - h_a(x,y)  &\text{if $|y|< \beta$, and} \\
\infty  &\text{if $|y|\ge \beta$}.
\end{cases}
\]

Let $M^+ := M\cap \{y\ge 0\}$.

Then the boundary $\partial M^+$ of $M^+$ consists of $(\partial M)\cap\{y\ge 0\}$ together with $\{(x,0,\phi(x))$, $x\ge 0$.
Now
\begin{align*}
H_a(x,0,\phi(x)) 
&= 
\phi(x) - h_a(x,0),
\end{align*}
which tends to $\infty$ as $x\to\infty$ by~\eqref{slopey}. (Note that $\phi(x)/x\to L$ and that $h_a(x,0)/x\to -s(\beta)$ as $x\to\infty$.)

Consequently, there is a $t<\infty$ such that
\[
   \inf_{\partial M^+} H_a = \inf_{(\partial M^+)\cap \{|x|\le t\}} H_a.
\]
By Lemma~\ref{pre-chini-lemma} (applied to the infimum on the right), 
the infimum is attained: there is a $q\in \partial M^+$ such that
\[
   \inf_{\partial M^+} H_a = H_a(q).
\]
Thus by Theorem~\ref{chini-theorem-1},
\[
  \min_{M^+}H_a = H_a(q).
\]
By symmetry, 
\begin{equation}
\min_M H_a = H_a(q). 
\end{equation}
 By the strong maximum principle, $q\in \partial M=M\cap\{x=0\}$.  
Thus
\begin{align*}
H_a(q) 
&= z(q) - h_a(0,y(q)) 
\ge 0,
\end{align*}
since $z(q)\ge 0$ by~\eqref{disjoint} and $h_a(0,y(q)) \le h_a(0,0)=0$.

Thus
\[
   H_a(p) \ge 0
\]
for all $p\in M$.
In particular, for $p=(x,0,\phi(x))$,
\[
  \phi(x) - h_a(x,0) \ge 0
\]
As $a\to\infty$, $h_a(x,0)\to s(\beta) x$.  Thus
\[
   \phi(x) \ge s(\beta)x
\]
for all $x$, so $L\ge s(\beta)$.
Letting $\beta\to b$ gives $L\ge s(b)$ and thus $L=s(b)$ by~\eqref{slope-range}.

\end{proof}

%%%%%%%%%%

\subsection{Some properties of  annuloids}
\label{subsec:annuloid}
Recall that an {\it annuloid} is a complete, properly embedded translating annulus $M$ in 
 $\RR^3$ with the following properties:
\begin{enumerate}[(a)]
\item $M$ is contained in a slab $\{|y| < B\}$, for some $B>0$.
\item  $M$ is symmetric under reflection in the planes $\{x=0\}$ and $\{y=0 \}$.
\item $M$ is disjoint from the axis $Z$.
\item As $z\to -\infty$, $M$ is smoothly asymptotic to the planes $y=\pm b$ and $y=\pm B$, where $b=b(M)$ and $B=B(M)$ and $0<b\le B<\infty$.
\item $M-(0,0,z)$ converges, as $z \to \infty$, to the empty set.
\end{enumerate}

The aim of this subsection is to prove the following:

\begin{theorem} \label{th:b>pi2}
If $M$ is an annuloid of finite type, then $b(M) \geq \pi/2.$
\end{theorem}

\begin{proof}
Since $M$ is symmetric about $\{y=0\}$, $M\cap\{y=0\}$ is a $1$-manifold.  Since $M$ is connected, $M\cap\{y=0\}$
is nonempty.  Let $\Gamma$ be a component of $M\cap\{y=0\}$.  Since $M\cap Z=\emptyset$, we can assume
(by symmetry) that $\Gamma$ is in $\{y=0, \, x>0\}$.
By Lemma \ref{topology-lemma-1}, $\Gamma$ is not compact, because that would violate the maximum principle.

Since $M$ is of finite type, $x(\cdot)|M$ has finitely many critical points.  Let $a$ be a number 
greater than every critical value of $x(\cdot)|M$.  Note that every critical point of $x(\cdot)|\Gamma$ is also a critical point
of $x(\cdot)|\Gamma$, so $a$ is greater than every critical value of $x(\cdot)|\Gamma$.

By Assertion~\eqref{second-symmetry-5} of Theorem~\ref{second-symmetry-theorem}, $x(\cdot)\to\infty$ on each end of $\Gamma$.
Thus for $t\ge a$,
 $\Gamma\cap\{x=t\}$ consists of two points $(t,0,z_1(t))$ and $(t,0,z_2(t))$, where
$z_1(t)< z_2(t)$.

For $t\ge a$, $M\cap \{x=t\}$ is a smooth manifold with exactly four ends, corresponding to the points at infinity
$(t, \pm b, -\infty)$ and $(t, \pm B,-\infty)$.  (By a slight abuse of notation, we regard these as four distinct 
points even if $b=B$.)   Reasoning as above, $M\cap\{x=t\}$ has no closed curve components.
Thus $M\cap\{x=t\}$ has exactly two components.  By symmetry and embeddedness, the component $C_1(t)$
containing $(t,0,z_1(t))$ has ends tending to $(t, \pm b, -\infty)$, and the component $C_2(t)$ containing $(t,0,z_2(t))$
has ends tending to $(t,\pm B, -\infty)$.

Let $\Sigma = \cup_{t\ge \tilde x} C_1(t)$.   Then $\partial \Sigma = C_1(a)$, a curve (in the plane $\{z=a\}$)
 on which $z(\cdot)$ is bounded above.
 Also, $\Sigma + (0,0,z)$ converges as $z\to\infty$ to the 
to the halfplanes $\{x\ge a\}\cap\{y= \pm b\}$. 
Thus $b\ge \pi/2$ by Theorem~\ref{chini-theorem-2}.
\end{proof}

%%%%%%%

\section{Convergence of Sets}\label{sets-appendix}

Here we describe basic facts about convergence of sets in metric spaces.
According to the \href{https://en.wikipedia.org/wiki/Kuratowski_convergence}{wikipedia article}
on Kuratowski convergence, these notions were introduced in lectures by Painlev\'e in 1902 and popularized in books by Hausdorff
and Kuratowski.

Let $X$ be a compact metric space.  Let $\KK(X)$
be the set of nonempty closed subsets of $X$.
We define a metric $d$ on $\KK(X)$ by
\begin{align*}
d(K_1,K_2) 
&= \| \dist(\cdot,K_1), \dist(\cdot,K_2)\|_{C_0} \\
&= \max_{p\in X}| \dist(p,K_1) - \dist(p,K_2)|.
\end{align*}
Note that $K_i\in \KK(X)$ converges to $K\in \KK(X)$ if and only if the following holds:
$\dist(p,K_i)\to 0$ for each $p\in K$ and $\liminf d(p,K_i)>0$  for each $p\notin K$.

Note that $\KK(X)$ is also compact, because if $K_i\in \KK(X)$, then by Arzela-Ascoli, 
 the functions $\dist(\cdot, K_i)$ will converge uniformly (after passing to a subsequence)
 to a limit $f$.  Let $K=\{p:f(p)=0\}$.  Then $f=\dist(\cdot,K)$ and so $K_i\to K$.

There are also useful notions of $\limsup$ and $\liminf$ in the space $\KK(X)$:

\begin{definition}\label{limsup-definition}
If $K_n$ is a sequence in $\KK(X)$, we let
\begin{align*}
\limsup_nK_n &:= \{p\in X: \liminf_n d(p,K_n) = 0\}, \\
\liminf_nK_n &:= \{p\in X: \limsup_n d(p,K_n)=0\}.
\end{align*}
\end{definition}
Thus $p\in \limsup_nK_n$ if and only there are $n(i)\to\infty$ and $p_{n(i)}\in K_{n(i)}$ such
the $p_{n(i)}$ converge to $p$.  Note that $\liminf_n K_n\subset \limsup_n K_n$. If equality holds, then $K_n$ converges,
and, conversely, if $K_n$ converges then
\[
 \liminf_n K_n = \limsup_n K_n = \lim_n K_n.
\]

\begin{theorem}\label{connected-limits-theorem}
Suppose that $K_n$ is a sequence of nonempty closed subsets of a compact metric space~$X$.
Then, after passing to subsequence, the $K_n$ converge to a closed set $K$.
Furthermore, if each $K_n$ is connected, then $K$ is also connected.
\end{theorem}

\begin{proof} 
We already proved the first assertion.
Assume that each $K_n$ is connected.
Let $C$ be a nonempty closed subset of $K$ such that $K\setminus C$ is also closed.  We must show that $K\setminus C$ is empty.
Let $f:X\to \RR$ be a continuous function that is $0$ on $C$ and $>0$ at each point of $X\setminus C$.  
(For example, $f(p)=\dist(p,C)$).
Now $f(K_n)$ is a compact, connected subset of $\RR$, so $f(K)$ is also.  Thus $f(K)=[0,r]$ for some $r\ge 0$.
Now $f(K\setminus C)= [0,r]\setminus \{0\}$ is compact since $K\setminus C$ is compact.  Thus $r=0$ and so $K\setminus C$ is empty.
\end{proof}

We would like to apply this theory when the metric space $X$ is $\RR^3$, but $\RR^3$ is not compact.
We get around this by using the one-point compactification $\RR^3\cup\{\infty\}$, with the metric
that comes from stereographic projection.
Let $\KK(\RR^3)$ be the family of all closed subsets of $\RR^3$, including the empty set.
If $K\in \KK(\RR^3)$, we let
\[
\tilde K
:=
\begin{cases}
K  &\text{if $K$ is compact and nonempty}, \\
\{\infty\} &\text{if $K=\emptyset$,} \\
K\cup\{\infty\} &\text{if $K$ is non-compact}.
\end{cases}
\]
We define a metric $\delta(\cdot,\cdot)$ on $\KK(\RR^3)$ by setting letting
\[
  \delta(K,K') = d(\tilde K, \tilde K'),
\]
where on the right side, $d(\cdot,\cdot)$ is the metric on $\KK(\RR^3\cup\{\infty\})$.
Then $K_i\in \KK(\RR^3)$ converges to $K\in \KK(\RR^3)$ if and only if
the following holds: every $p\in K$ is a limit of $p_i\in K_i$, and every $q\notin K$
has a neighborhood $U$ such that $U\cap K_i$ is empty for all sufficiently large $i$.

The following theorem is an immediate consequence of Theorem~\ref{connected-limits-theorem}:

\begin{theorem}\label{R3-theorem}
Let $K_n$ be a sequence of closed subsets of $\RR^3$.
After passing to a subsequence, the $K_n$ converge to a closed $K$.
Suppose each $K_n$ is connected.  
If $K$ is compact, it is connected, and if $K$ is noncompact, then $K\cup\{\infty\}$
is connected.
\end{theorem}

\begin{remark}
Theorem~\ref{R3-theorem} remains true, with the same proof, in any locally compact space whose one-point compactification is
metrizable.
\end{remark}

\begin{theorem}\label{two-point-theorem}
Let $p\ne q$ be two points in a compact, connected metric space $M$.
Then there is a compact, connected set $X$ containing $p$ and $q$
such that $X\setminus \{p,q\}$ is also connected.
\end{theorem}

\begin{proof}
Let $\Ff$ be the family of closed sets $S$ in $M$ with the following property:
$\{p,q\}\subset S$, and if $T$ is a clopen subset of $S$ containing
 one of the two points $p$ and $q$, then it also contains the other.
Then $\Ff$ is nonempty since $M\in \Ff$.

\begin{claim} If $X\in \Ff$ and if $X\setminus \{p,q\}$ is not connected, then there is an $X'\in \Ff$
with $X'\subsetneq X$.
\end{claim}

To prove the claim, note that, by hypothesis, $X\setminus \{p,q\}=Y \cup Z$, where $Y$ and $Z$ are nonempty, disjoint,
sets that are relatively closed in $X\setminus \{p,q\}$.
It suffices to show that at least one of the sets $\overline{Y}$ and $\overline{Z}$ belongs to $\Ff$,
since then we can let $X'$ be that set.

If every nonempty clopen subset of $\overline{Y}$ contains both $p$ and $q$, then $\overline{Y}$ is a connected
set containing $p$ and $q$, so $\overline{Y}$ is in $\Ff$, and we are done proving the claim.
Now suppose that there is a nonempty clopen subset $S$ of $\overline{Y}$ that does not contain both $p$ and $q$.
We may assume that $S$ does not contain $q$.  Let $T$ be a clopen subset of $\overline{Z}$ that contains $p$.
Then $\overline{Y}\cup T$ is a clopen subset of $X$ containing $p$, so it also contains $q$.  Since $q$ is not in $\overline{Y}$,
 $q\in T$.  Hence $\overline{Z}$ is in $\Ff$.  This completes the proof of the claim.

By Zorn's Lemma, there is a set $X$ in $\Ff$ such that no proper subset of $X$ is in $\Ff$.
By the claim, $X\setminus\{p,q\}$ is connected.

(Here is the Zorn's Lemma argument.  Let $\Ll$ be nonempty subcolletion of $\Ff$ that is linearly
ordered by inclusion.  Let $K=\cup_{X\in \Ll}X$.
We must show that $K\in \Ff$. It not, we could write $K$ as the disjoint union of two closed sets $S$ and $T$
with $p\in S$ and $q\in T$.  Let $f:M\to [0,1]$ be a continuous function such that $S=f^{-1}$ and $T=f^{-1}(1)$.
Since the sets $X\cap \{f=1/2\}$ (with $X\in \Ll$) are nested compact sets whose intersection $K\cap\{f=1/2\}$ is empty,
we see that $X\cap\{f=1/2\}$ is empty for some $X\in \Ll$.
But then the set $Q:=X\cap\{f\le 1/2\}=X\cap \{f<1/2\}$ is a clopen subset of $X$ containing $p$ but not $q$,
a contradiction.)
\end{proof}

\begin{theorem}\label{connected-extraction-theorem}
Let $X$ be a metric space and $f:X\to I$ be a proper, continuous map to an open interval $I\subset \RR$.
Suppose $X$ has the following property: if $J\subset I$ is a compact interval, then there is 
a compact, connected set $Y\subset X$ such that $f(Y)=J$.
Then there is a closed, connected subset $Q$ of $X$ such that $f(Q)=I$.
\end{theorem}

\begin{proof}
We can assume that the interval is $(0,1)$.
Add two points $p$ and $q$ to $X$ to get $\tilde X = X\cup\{p,q\}$.

Add two points $p$ and $q$ to $X$ to get $\tilde X=X\cup\{p,q\}$.  Extend $f$ to $\tilde X$
by setting $f(p)=0$ and $f(q)=1$.   Let $\delta$ be a metric on $\tilde X$ such that $\tilde X$
is compact, $f:\tilde X\to [0,1]$ is continuous, and $\delta$ and $d$ induce the same topology on $X$.

(Such a metric $\delta$ can be defined as follows.  First we extend $d(\cdot, \cdot)$ to $X\cup\{p,q\}$
by setting $d(y,z)= |f(y)-f(z)|$ if either $y$ or $z$.  This $d(\cdot,\cdot)$ will not satisfy the triangle
inequality.  We define $\delta(y,y')$ to be the infimum of $\sum_{i=1}^n d( y_{i-1}, y_i)$
among finite sequences $y_0, y_1, \dots, y_n$ such that $y_0= y$ and $y_n=y'$.  
 It is straightforward to show that $\delta$ is a metric
with the asserted properties.)

Let $J_1, J_2, \dots$ be a sequence of compact subintervals of $(0,1)$ whose union is $(0,1)$.
By hypothesis, there are compact, connected sets $K_n\subset X$ such that $f(K_n)=J_n$.
By Theorem~\ref{connected-limits-theorem}, there is a subsequence $K_{n(i)}$ that converges to a compact, connected
subset $K$ of $\tilde X$. Note that $f(K)=[0,1]$. 
By Theorem~\ref{two-point-theorem},
 there is a closed connected set $K'$ containing $p$ and $q$ such that $X:=K'\setminus\{p,q\}$
is connected.  Since $f(K')=[0,1]$ (by the intermediate value theorem), $f(X)=(0,1)$.
\end{proof}

\begin{remark}
Theorem~\ref{connected-extraction-theorem} holds for any interval $I$, not just open intervals.
  It holds trivially if $I$ is a closed interval.
The case when $I$ is half-open reduces to the open case by doubling.
\end{remark}

\section{Notation}

For the reader's convenience, we provide a list of notation used in this paper.
Items are listed alphabetically according to how they are spelled in English.
Thus $\Gamma$ (``Gamma'') appears after $F_\vv$ and before ``Grim reaper surface''.
Some notations that are used in only a small portion of the paper are not listed here;
of course, those notations are explained where they occur.

\begin{enumerate}[\upshape{}]
\item $a(M)$, $A(M)$.  Definition~\ref{Cc-definition}.
\item $|A(M,p)|$: norm of the (Euclidean) second fundamental form of $M$ at $p$.
\item $\Aa$, $\Aa(b)$, $\Aa(b,\hat x)$, $\Aa(b,B,\hat x)$.    Definition~\ref{Aa-definition}.
\item Bowl solition.  \S\ref{graphs-section}
\item $\Cc$: a  space of translating annuli.  Definition~\ref{Cc-definition}.
\item $\Cc_{\Gamma(t)}$. Defined just before Lemma~\ref{lem:gap}.
\item $\cin(t)$, $\cout(t)$.  Defined just before Lemma~\ref{lem:gap}.
\item $b(M)$: inner width of $M$.  \S\ref{introduction}, and also Definition~\ref{Cc-definition}.
\item $B(M)$: (outer) width of $M$. \S\ref{introduction}, and  also Definition~\ref{Cc-definition}.
\item $\Delta$-wing. \S\ref{graphs-section}.
\item $f_b(x,y)$: $\Delta$-wing over $\RR\times (-b,b)$ with $f_b(0,0)=0$ and $Df_b(0,0)=0$. 
  \S\ref{graphs-section}.
\item $f_M$. Definition~\ref{definition-myin-myout}.
\item $F_\vv$:  The function $F_\vv(p)=\vv\cdot p$. See~\eqref{definition-F_v} in \S\ref{morse-rado-section}.
\item $\Gamma(t)$, $\Gamma_{\rm in}(t)$, $\Gamma_{\rm out}(t)$. Defined just
    before Lemma~\ref{lem:gap}.
\item Grim reaper surface. \S\ref{graphs-section}
\item Minimal foliation function.  Defined just before Theorem~\ref{morse-rado-theorem}.
\item  $\Mlow$, $\Mup$. Theorems~\ref{components-theorem} and~\ref{connected-theorem}.
\item $\Myin$, $\Myout$. Definition~\ref{definition-myin-myout}.
\item $\mathsf{N}(\Ff|M)$: number of critical points.  Definition~\ref{critical-point-definition}.
\item $\mathsf{N}(F|M)$.   Defined just before Theorem~\ref{morse-rado-theorem}.
\item $\Omegain$, $\Omegaout$.  Theorem~\ref{y-graph-theorem}.
\item $\Pairs$.  Defined just before Lemma~\ref{proper-lemma}.
\item $\partialin M$, $\partialout M$: inner and outer boundaries of $M$.  Definition~\ref{Cc-definition}.
\item $\uup(x)$, $\ulow(x)$.  Theorems~\ref{y-slice-theorem} and~\ref{finite-y-slice-theorem}.
\item $\Rr$: a space of translating annuli with rectangular boundaries.
    Definition~\ref{Rr-definition}.
\item $s(b)$:  slope $\partial u_b/\partial x$ of the grim reaper function $u_b$.  
    See Item~\eqref{u_b} in~\S\ref{graphs-section}.
\item Translator metric. Equation~\eqref{translator-metric}.
\item $u_b(x,y)$:   a grim reaper surface over $\RR\times(-b,b)$.   \S\ref{graphs-section}.
\item $u_M$.  Definition~\ref{u-definition}.
\item $\waist(M)$.  Definition~\ref{waist-definition}.
\item $w_b(x,y)$.  Definition~\ref{reaper-definition}.
\item $x(M)$: necksize of $M$. Definition~\ref{def:x(M)}.
\item $y(M)$.  Equation~\eqref{yM-definition}.
\item $\yin(x,z)$, $\yout(x,z)$.  Theorem~\ref{y-graph-theorem}.
\item $z(M)$.  Definition~\ref{z-of-M-definition}.
\end{enumerate}

\end{document}